\documentclass{amsart}

\usepackage{amsmath,amssymb,amscd}
\usepackage{amsbsy,upgreek}
\usepackage[all]{xy}
\usepackage{epic}
\usepackage{caption,subcaption}
\usepackage{hyperref}
\usepackage[toc]{appendix}
\input{finitetype.sty}

\newtheorem{theorem}{Theorem}[section]
\newtheorem{lemma}[theorem]{Lemma}
\newtheorem{lemmadef}[theorem]{Lemma-Definition}
\newtheorem{corollary}[theorem]{Corollary}
\newtheorem{proposition}[theorem]{Proposition}

\theoremstyle{definition}
\newtheorem{definition}[theorem]{Definition}
\newtheorem{example}[theorem]{Example}

\newtheorem{conjecture}[theorem]{Conjecture}

\theoremstyle{remark}
\newtheorem{remark}[theorem]{Remark}

\numberwithin{equation}{section}

\DeclareMathOperator{\Coker}{Coker}
\DeclareMathOperator{\Conf}{Conf}
\DeclareMathOperator{\E}{E}
\DeclareMathOperator{\End}{End}
\DeclareMathOperator{\Ext}{Ext}

\DeclareMathOperator{\GL}{GL}
\DeclareMathOperator{\Gr}{Gr}
\DeclareMathOperator{\Hom}{Hom}
\DeclareMathOperator{\Id}{Id}
\DeclareMathOperator{\Img}{Im}
\DeclareMathOperator{\ind}{ind}
\DeclareMathOperator{\Irr}{Irr}
\DeclareMathOperator{\Ker}{Ker}

\DeclareMathOperator{\module}{mod}
\DeclareMathOperator{\PHom}{PHom}

\DeclareMathOperator{\Rep}{Rep}

\DeclareMathOperator{\SL}{SL}

\DeclareMathOperator{\Tr}{Tr}

\newcommand{\op}[1]{\operatorname{#1}}
\newcommand{\opp}{{\operatorname{op}}}
\newcommand{\mc}[1]{\mathcal{#1}}
\newcommand{\mb}[1]{\mathbb{#1}}
\newcommand{\mr}[1]{{\sf #1}}
\newcommand{\mf}[1]{\mathfrak{#1}}
\renewcommand{\b}[1]{\bold{#1}}
\newcommand{\bs}[1]{\boldsymbol{#1}}
\newcommand{\br}[1]{\overline{#1}}

\newcommand{\dv}{\underline{\dim}}
\newcommand{\wtd}[1]{\widetilde{#1}}

\newcommand{\e}{{\sf e}}
\newcommand{\f}{{\sf f}}
\newcommand{\g}{{\sf g}}
\newcommand{\h}{{\sf h}}

\newcommand{\zero}{\operatorname{O}}
\newcommand{\proj}{\operatorname{proj}\text{-}}

\newcommand{\ckQ}{\widehat{k\Delta}}

\newcommand{\innerprod}[1]{\langle#1\rangle}

\newcommand{\sm}[1]{{\left(\begin{smallmatrix}#1\end{smallmatrix}\right)}}
\newcommand{\vsm}[1]{{\begin{smallmatrix}#1\end{smallmatrix}}}
\newcommand{\tw}[2]{{\begin{smallmatrix}#2\\#1\end{smallmatrix}}}
\newcommand{\fr}[1]{\framebox[1.2\width]{{$#1$}}}
\newcommand{\lrcoef}{c^\lambda_{\mu\,\nu}}
\newcommand{\tripar}{{\mu,\nu,\lambda}}
\newcommand{\triwtf}{{\e(f),\f_-,\f_+}}
\renewcommand{\P}{{\bs{P}}}
\newcommand{\I}{{\bs{I}}}
\newcommand{\bcsl}{\backslash}
\newcommand{\uca}{\br{\mc{C}}}
\newcommand{\DCQ}{\Delta_{Q}^2}
\newcommand{\ucaCtwoQ}{\uca(\DCQ,\mc{S}_Q^2;\bs{\sigma}_{Q}^2)}
\newcommand{\ucaCtwoQno}{\uca(\DCQ,\mc{S}_Q^2)}
\newcommand{\sCQ}{\bs{\sigma}_{Q}^2}
\newcommand{\sQ}{\bs{\sigma}_{Q}}

\begin{document}
\title{Tensor Product Multiplicities via Upper Cluster Algebras}
\author{Jiarui Fei}
\address{School of Mathematical Sciences, Shanghai Jiao Tong University, Shanghai, China}
\email{jiarui@sjtu.edu.cn}
\thanks{}
%
%

\begin{abstract} For each valued quiver $Q$ of Dynkin type, we construct a valued ice quiver $\DCQ$.
Let $G$ be a simply connected Lie group with Dynkin diagram the underlying valued graph of $Q$.
The upper cluster algebra of $\DCQ$ is graded by the triple dominant weights $(\tripar)$ of $G$.
We prove that when $G$ is simply-laced, the dimension of each graded component counts the tensor multiplicity $c_{\mu,\nu}^\lambda$.
We conjecture that this is also true if $G$ is not simply-laced, and sketch a possible approach.
Using this construction, we improve Berenstein-Zelevinsky's model, or in some sense generalize Knutson-Tao's hive model in type $A$.
\end{abstract}

\maketitle

\setcounter{tocdepth}{1}
\tableofcontents

\section*{Introduction}

Finding the polyhedral model for the tensor multiplicities in Lie theory is a long-standing problem.
By {\em tensor multiplicities} we mean the multiplicities of irreducible summands in the tensor product of any two finite-dimensional irreducible representations of a simply connected Lie group $G$.
The problem asks to express the multiplicity as the number of lattice points in some convex polytope.

Accumulating from the works of Gelfand, Berenstein and Zelevinsky since 1970's,
a first quite satisfying model for $G$ of type $A$ was invented in \cite{BZa}.
Finally around 1999, building upon their work, Knutson and Tao invented their {\em hive model}, which led to the solution of the {\em saturation conjecture} \cite{KT}.
In fact, the reduction of Horn's problem to the Saturation conjecture is an important driving force for the evolution of the models.

Outside type $A$, up to now Berenstein and Zelevinsky's models \cite{BZ} are still the only known polyhedral models.
Those models lose a few nice features of Knutson-Tao's hive model.
We will have a short discussion on this in Section \ref{ss:models}.
Despite of a lot of effort to improve the Berenstein-Zelevinsky model, to author's best knowledge there is no very satisfying further results in this direction.

Recently an interesting link between the hive model and the {\em cluster algebra} theory was established in \cite{Fs1} through the Derksen-Weyman-Zelevinsky's {\em quiver with potential} model \cite{DWZ1,DWZ2} for cluster algebras. A similar but different link between the polyhedral models and tropical geometry was established by Goncharov and Shen in \cite{GS}. In fact, from the work of Berenstein, Fomin and Zelevinsky \cite{BZ,BFZ}, those links may not be a big surprise.

There are two goals in the current paper.
First we want to generalize the work \cite{Fs1} to other types.
More specifically, we hope to prove that the algebras of regular functions on certain configuration spaces are all upper cluster algebras.
Second we want to improve the Berenstein-Zelevinsky's model in the spirit of Knutson-Tao.
In fact, as we shall see, we accomplish these two goals almost simultaneously. Namely, we use our conjectural models to establish the cluster algebra structures. Once the cluster structures are established, the conjectural models are proved as well.

The key to making new models is the construction of the {\em iARt quivers}.
Let $Q$ be a valued quiver of Dynkin type.
Let $C^2Q$ be the category of projective presentations of $Q$.
We can associate this category an {\em Auslander-Reiten quiver} $\Delta(C^2Q)$ with translation (ARt quiver in short).
The ice ARt quiver (iARt quiver in short) $\Delta_Q^2$ is obtained from $\Delta(C^2Q)$ by {\em freezing} three sets of vertices,
which correspond to the {\em negative}, {\em positive}, and {\em neutral} presentations in $C^2Q$.
We can put a (quite canonical) {\em potential} $W_Q^2$ on the iARt quiver $\Delta_Q^2$.

A quiver with potential (or QP in short) $(\Delta,W)$ is related to Berenstein-Fomin-Zelevinsky's {\em upper cluster algebras} \cite{BFZ} through {\em cluster characters} evaluating on {\em $\mu$-supported $\g$-vectors} introduced in \cite{Fs1} (see Definition \ref{D:mu_sup} and \ref{D:mu_supported}).
The cluster character $C_W$ considered in this paper is the {\em generic} one \cite{P,Du}, but it can be replaced by fancier ones.
As we have seen in many different situations \cite{Fs1,Fs2,Fk1} that the set $G(\Delta,W)$ of $\mu$-supported $\g$-vectors is given by lattice points in some rational polyhedral cone.
This is also the case for the iARt QPs $(\Delta_Q^2,W_Q^2)$.

The whole Part \ref{P:I} is devoted to the construction of the iARt QP $(\Delta_Q^2,W_Q^2)$ and the polyhedral cone ${\sf G}_{\Delta_Q^2}$.
It turns out that the cone ${\sf G}_{\Delta_Q^2}$ has a very neat hyperplane presentation $\{x\in \mb{R}^{(\Delta_Q^2)_0} \mid xH\geq 0\}$,
where the columns of the matrix $H$ are given by the dimension vectors of subrepresentations of $3|Q_0|$ representations of $\Delta_Q^2$.
These $3|Q_0|$ representations are in bijection with the frozen vertices of $\Delta_Q^2$.
They also have a very simple and nice description (see Theorem \ref{T:Tv}).
The main result of Part \ref{P:I} is the following.
\begin{theorem}[Theorem \ref{T:GQ2}] The set ${\sf G}_{\Delta_Q^2} \cap \mb{Z}^{(\Delta_Q^2)_0}$ is exactly $G(\Delta_Q^2,W_Q^2)$.
\end{theorem}

The upper cluster algebra $\uca(\Delta_Q^2)$ has a natural grading by the {\em weight vectors} of presentations. This grading can be extended to a triple-weight grading $\sCQ: \mb{Z}^{(\Delta_Q^2)_0} \to \mb{Z}_{\geqslant 0}^{3|Q_0|}$.
This grading slices the cone ${\sf G}_{\Delta_Q^2}$ into polytopes $${\sf G}_{\Delta_Q^2}(\mu,\nu,\lambda):=\left\{\g\in {\sf G}_{\Delta_Q^2}\mid \sCQ(\g) = (\mu,\nu,\lambda) \right\}.$$
Let $G:=G_Q$ be the simply connected simple Lie group with Dynkin diagram the underlying valued graph of $Q$.
Our conjectural model is that the lattice points in ${\sf G}_{\Delta_Q^2}(\mu,\nu,\lambda)$ counts the tensor multiplicity $\lrcoef$ for $G$.
Here, $\lrcoef$ is the multiplicity of the irreducible representation $L(\lambda)$ of highest weight $\lambda$ in the tensor product $L(\mu)\otimes L(\nu)$.
More often than not we identify a dominant weight by a non-negative integral vector.
To prove this model, we follow a similar line as \cite{Fs1}.
However, we do not have a quiver setting to work with in general.
We replace the semi-invariant rings of triple-flag quiver representations by the ring of regular functions on certain configuration space introduced in \cite{FG}.

Fix an opposite pair of maximal unipotent subgroups $(U^-,U)$ of $G$.
The quotient space $\mc{A}:=U^-\bcsl G$ is called {\em base affine space}, and the quotient space $\mc{A}^\vee:=G/U$ is called its {\em dual}.
The configuration space $\Conf_{2,1}$ is by definition $(\mc{A}\times \mc{A}\times \mc{A}^\vee)/G$, where $G$ acts multi-diagonally.
The ring of regular functions $k[\Conf_{2,1}]$ is just the invariant ring $(k[G]^{U^-}\otimes k[G]^{U^-}\otimes k[G]^U)^G$.
The ring $k[\Conf_{2,1}]$ is multigraded by a triple of weights $(\tripar)$.
Each graded component $C_{\mu,\nu}^\lambda:=k[\Conf_{2,1}]_{\tripar}$ is given by the $G$-invariant space $\left(L(\mu)\otimes L(\nu)\otimes L(\lambda)^\vee\right)^G$.
So the dimension of $C_{\mu,\nu}^\lambda$ counts the tensor multiplicity $\lrcoef$. Here is the main result of Part \ref{P:II}.

\begin{theorem}[Theorem \ref{T:main}] \label{T:intro2} Suppose that $Q$ is trivially valued. 
Then the ring of regular functions on $\Conf_{2,1}$ is the graded upper cluster algebra $\ucaCtwoQ$.
Moreover, the generic character maps the lattice points in ${\sf G}_{\Delta_Q^2}$ onto a basis of this algebra.
In particular, $\lrcoef$ is counted by lattice points in ${\sf G}_{\Delta_Q^2}(\tripar)$.
\end{theorem}
\noindent We will show by an example that the upper cluster algebra strictly contains the corresponding cluster algebra in general.
We conjecture that the trivially valued assumption can be dropped in the above theorem and the theorem below.
It is pointed in the end that the only missing ingredient for proving the conjecture is the analogue of \cite[Lemma 5.2]{DWZ2} for {\em species with potentials} \cite{LZ}.

Fock and Goncharov studied in \cite{FG} the similar spaces $\Conf_{3}$ \footnote{They considered the generic part of the quotient stack $[(\mc{A}^\vee)^3/G]$. We will work with the categorical quotient as its partial compactification.} as cluster varieties.
However, to author's best knowledge it is not clear from their discussion what an {\em initial seed} is if $G$ is not of type $A$.
Moreover the equality established in the theorem does not seem to follow from any result there.
In fact, Fock and Goncharov later conjectured in \cite{FGc} that the tropical points in their cluster $\mc{X}$-varieties parametrize bases in the corresponding (upper) cluster algebras.
Our result can be viewed as an algebraic analog of their conjecture for the space $\Conf_{2,1}$. Instead of working with the tropical points, we work with the $\g$-vectors.

To sketch our ideas, we first observe that if we forget the frozen vertices corresponding to the positive and neutral presentations,
then we get a valued ice quiver denoted by $\Delta_Q$ whose cluster algebra is isomorphic to the coordinate ring $k[U]$.
Roughly speaking, this procedure corresponds to an open embedding $i:H\times H\times U\hookrightarrow \Conf_{2,1}$, or more precisely Corollary \ref{C:localization}.
We will define the cluster $\mc{S}_Q^2$ in Theorem \ref{T:intro2} through the pullback map $i^*$.
It is then not hard to show that $k[\Conf_{2,1}]$ contains the upper cluster algebra $\ucaCtwoQ$ as a graded subalgebra.
The detail will be given in Section \ref{ss:standard}.

So far we have the graded inclusions
$$\op{Span}\left(C_W({\sf G}_{\Delta_Q^2})\right)\subseteq \ucaCtwoQ\subseteq k[\Conf_{2,1}].$$
To finish the proof, it suffices to show the containment $k[\Conf_{2,1}]\subseteq \op{Span}\left( C_W({\sf G}_{\Delta_Q^2}) \right)$.
For this, we come back to the cluster structure of $k[U]$.
It turns out that the analog of Theorem \ref{T:intro2} for $U$ is rather easy to prove. The set $G(\Delta_Q,W_Q)$ contains exactly lattice points in the polytope ${\sf G}_{\Delta_Q}$, which is defined by one of the three sets of relations of ${\sf G}_{\Delta_Q^2}$. On the other hand, we have two other embeddings $i_l,\ i_r: U\hookrightarrow \Conf_{2,1}$. They are the map $i_u:=i|_U$ followed by the {\em twisted cyclic shift} of $\Conf_{2,1}$.
Another crucial ingredient in this paper is an interpretation of the twisted cyclic shift in terms of a sequence of mutations $\bs{\mu}_{l}$. Applying $\bs{\mu}_l$ and $\bs{\mu}_l^{-1}$ to the QP $(\Delta_Q,W_Q)$, we get two other QPs $(\Delta_Q^l,W_Q^l)$ and $(\Delta_Q^r,W_Q^r)$.
The analogous polytopes ${\sf G}_{\Delta_Q^l}$ and ${\sf G}_{\Delta_Q^r}$ for them are defined by the other two sets of relations of ${\sf G}_{\Delta_Q^2}$.
Finally, after showing the good behavior of $\g$-vectors under the pullback of the three embeddings, the required inclusion will follow from the fact that
\begin{equation} \label{eq:keycontain} k[\Conf_{2,1}]\subseteq \left\{s\in \mc{L}(\mc{S}_Q^2)\mid i_{\#}^*(s)\in k[U] \text{ for $\#=u,l,r$} \right\}, \end{equation}
where $\mc{L}(\mc{S}_Q^2)$ is the Laurent polynomial ring in the cluster $\mc{S}_Q^2$.
The detail will be given in Section \ref{ss:twist}.

Except for these two main results, we have a side result for the base affine spaces.
The author would like to thank B. Leclerc and M. Yakimov for confirming that the following theorem was an open problem.
It turns out that the cluster structure of $\mc{A}$ lies between that of $U$ and $\Conf_{2,1}$.
Let $\Delta_Q^\sharp$ be the valued ice quiver obtained from $\Delta_Q^2$ by deleting frozen vertices corresponding to neutral presentations.

\begin{theorem}[Theorem \ref{T:side}] Suppose that $Q$ is trivially valued. Then the ring of regular functions on $\mc{A}$ is the graded upper cluster algebra $\uca(\Delta_Q^\sharp,\mc{M}_Q^\sharp; \varpi(\bs{\sigma}_Q^\sharp))$.
Moreover, the generic character maps the lattice points in ${\sf G}_{\Delta_Q^\sharp}$ onto a basis of this algebra.
In particular, the weight multiplicity $\dim L(\mu)_{\lambda}$ is counted by lattice points in ${\sf G}_{\Delta_Q^\sharp}(\mu,\lambda)$.
\end{theorem}

\subsection{The Models} \label{ss:models}
In \cite{KT} Knutson and Tao invented a remarkable polyhedral model called hives or honeycomb. The author personally thinks that it has at least three advantages over Berenstein-Zelevinsky's model \cite{BZ}.
First, the hive polytopes have a nice presentation
$\left\{x\in \mb{R}^{3n}\mid xH\geq 0,\ x\bs{\sigma}=(\mu,\nu,\lambda)\right\}.$
Second, the cyclic symmetry of the type-$A$ tensor multiplicity is lucid from the hive model. Actually other symmetries can also follow from the hive model.
Last and most importantly, there is an operation called {\em overlaying} for honeycombs \cite{KT}.

In appropriate sense, our models share these nice properties. The first one is clear from our result. Our $H$-matrices even have all non-negative entries. However, if readers prefer the rhombus-type inequalities of the hives, one can transform our model through a totally unimodular map as in \cite{Fs1}. However, the rhombus-type inequalities are not always as neat as the ones in type $A$.
We will discuss the transformation and the analogous overlaying elsewhere.
Although the cyclic symmetry is not immediately clear from $H$ itself, we understand from our construction and Appendix \ref{ss:cyclic} that it is just hidden there.
We believe that this is probably the best we can do outside type $A$.

In a more general context of Kac-Moody algebras, the tensor multiplicity problem can be solved by P. Littelmann's path model \cite{Li}.
As pointed out in \cite{BZ}, his model can be transformed into polyhedral ones (with some non-trivial work). However, in general it involves a union of several convex polytopes.

\subsection{Relation to the work of Berenstein-Zelevinsky and Goncharov-Shen} \label{ss:BZGS}
In a groundbreaking work \cite{BZ} Berenstein and Zelevinsky invented their polyhedral model for all Dynkin types. Their main tools are Lusztig's canonical basis and tropical relations in double Bruhat cells.
The polytopes are defined explicitly in terms of their {\em $\b{i}$-trails}. But the author feels that $\b{i}$-trails are hard to compute especially in type $E$. By contrast, the subrepresentations defining our $H$ are rather easy to list in most cases. In few difficult cases, such as type $E_7$ and $E_8$, we provide an algorithm suitable for computers.

Recently Goncharov and Shen made some further progress in \cite{GS}.
Using tropical geometry and geometric Satake, they proved a more symmetric polyhedral model (see {\cite[Theorem 2.6 and (214)]{GS}}).
However, there is no further {\em explicit} description on the polytopes.
The equality of \eqref{eq:keycontain} as an intermediate byproduct of our proof is similar to this result.

Loosely speaking, our work is independent of their results, though the author did benefit a lot from reading their papers. 
The construction of iARt quivers $\Delta_Q^2$ is new. We believe that the construction and results, especially the ideas behind, are beyond just solving the tensor multiplicity problem for simple Lie groups.
The proofs in Part \ref{P:I} are similar to those in \cite{Fs1}.
In Part \ref{P:II} what we heavily rely on is the cluster structure of $k[U]$ and a mutation interpretation of the twisted cyclic shift.
Throughout the quiver with potential model for cluster algebras is most important.

\subsection*{Outline of the Paper} In Section \ref{ss:VQ} we recall the basics on valued quivers and their representations.
We define the graded upper cluster algebra attached to a valued quiver in Section \ref{ss:UCA} and \ref{ss:gv}.
In Section \ref{ss:AR_fun} we recall the Auslander-Reiten theory from a functorial point of view.
We specialize the theory to the category of presentations mostly for hereditary algebras in Section \ref{ss:C2A}.
In Section \ref{ss:iARt} we define the iARt quivers in general.
We then consider the hereditary cases in more detail in Section \ref{ss:gdim1}.
Proposition \ref{P:iARtDCQ} compares the ARt quivers of presentations with the more familiar ARt quivers of representations.
In Section \ref{S:CCQP} we review the generic cluster character in the setting of quivers with potentials.
In Section \ref{S:iARtQP} we study the iARt QPs and their $\mu$-supported $\g$-vectors.
We prove the two main results of Part \ref{P:I} -- Theorem \ref{T:Tv} and \ref{T:GQ2}.
In Appendix \ref{A:list}, we provide more examples of iARt quivers.

In Section \ref{S:Ugp} we review the rings of regular functions on base affine spaces and maximal unipotent groups, especially the cluster structure of the latter (Theorem \ref{T:Ucluster} and Proposition \ref{P:Ulr}).
In Section \ref{S:toU} we study maps relating the configuration spaces to the corresponding unipotent groups.
These are almost all the technical work required for proving the main result.
In Section \ref{S:CS} we prove our main result -- Theorem \ref{T:main}.
In Section \ref{S:epi} we prove the side result -- Theorem \ref{T:side}.
In the end we make some remark on the possible generalization to the non-simply laced cases.
In Appendix \ref{A:cyclic} we prove the mutation interpretation of the twisted cyclic shift in Theorem \ref{T:cyclic}.
As a consequence, we produce an algorithm for computing the ($\mu$-supported) $\g$-vector cones.

\subsection*{Notations and Conventions}
Our vectors are exclusively row vectors. All modules are right modules.
Arrows are composed from left to right, i.e., $ab$ is the path $\cdot \xrightarrow{a}\cdot \xrightarrow{b} \cdot$.
Unless otherwise stated, unadorned $\Hom$ and $\otimes$ are all over the base field $k$, and the superscript $*$ is the trivial dual for vector spaces.
For direct sum of $n$ copies of $M$, we write $nM$ instead of the traditional $M^{\oplus n}$.

\part{Construction of iARt QPs} \label{P:I}
\section{Graded Upper Cluster Algebras} \label{S:GUCA}
\subsection{Valued Quivers and their Representations} \label{ss:VQ}
If you are familiar with the usual quiver representations and only care about our results on the simply laced cases, you can skip this subsection.

\begin{definition} A valued quiver is a triple $Q=(Q_0,Q_1,C)$ where
\begin{enumerate}
\item $Q_0$ is a set of vertices, usually labelled by natural numbers $1,2,\dots,n$;
\item $Q_1$ is a set of arrows, which is a subset of $Q_0\times Q_0$;
\item $C=\{(c_{i,j},c_{j,i})\in \mb{N}\times\mb{N}\mid (i,j)\in Q_1\}$ is called the {\em valuation} of $Q$.
\end{enumerate}
It is called {\em symmetrizable} if there is $d=\{d_i\in\mb{N}\mid i\in Q_0\}$ such that $d_ic_{i,j}=c_{j,i}d_j$ for every $(i,j)\in Q_1$.
\end{definition}
\noindent For such a valued quiver, the pair $(Q_0,Q_1)$ is called its {\em ordinary} quiver.
Throughout this paper, all valued quivers are assumed to have no loops or oriented 2-cycles in their ordinary quivers.
If $c_{i,j}=c_{j,i}$ for every $(c_{i,j},c_{j,i})\in C$, then $Q$ is called {\em equally valued}.
To draw a valued quiver $(Q_0,Q_1,C)$, we first draw its ordinary quiver, then put valuations above its arrows, eg. $i\xrightarrow{(c_{i,j},c_{j,i})} j$. We will omit the valuation if $(i,j)$ is trivially valued, i.e., $c_{i,j}=c_{j,i}=1$.
All valued quivers in this paper will be symmetrizable. We always fix a choice of $d$, so readers may view $d$ as a part of the defining data for $Q$. We let $d_{i,j}=\gcd(d_i,d_j)$.

Let $\mb{F}$ be a finite field. We write $\br{\mb{F}}$ for an algebraic closure of $\mb{F}$.
For each positive integer $k$ denote by $\mb{F}_k$ the degree $k$ extension of $\mb{F}$ in $\br{\mb{F}}$. Note that the largest subfield of $\br{\mb{F}}$ contained in both $\mb{F}_k$ and $\mb{F}_l$ is $\mb{F}_{\gcd(k,l)}=\mb{F}_k\cap \mb{F}_l$.
If $k\mid l$ we can fix a basis of $\mb{F}_l$ over $\mb{F}_k$ and thus freely identify $\mb{F}_l$ as a vector space over $\mb{F}_k$.

A representation $M$ of $Q$ is an assignment for each $i\in Q_0$ a $\mb{F}_{d_i}$-vector space $M(i)$, and for each arrow $(i,j)\in Q_1$ an $\mb{F}_{d_{i,j}}$-linear map $M(i,j)$.
This definition is different from the original one in \cite{DR}, but it is more adapted to the cluster algebra theory (see \cite{R1}).
The equivalence of two definitions was established in \cite[(2.2)]{R1}.
The dimension vector $\dv M$ is the integer vector $(\dim_{\mb{F}_{d_i}} M(i))_{i\in Q_0}$.
Similar to the usual quiver representations, we can define a morphism $\phi:M\to N$ as the set
$$\left\{\phi_i\in \Hom_{\mb{F}_{d_i}}(M(i),N(i))\right\}_{i\in Q_0} \text{ such that } \phi_j M(i,j) = N(i,j)\phi_i \text{ for all } (i,j)\in Q_1.$$
The category $\Rep(Q)$ of all (finite-dimensional) representations of $Q$ is an abelian category, in which the kernels and cokernels are taken vertex-wise.
The category $\Rep(Q)$ is also {\em Krull-Schmidt}, that is, each object is a finite direct sum of indecomposable objects with local endomorphism rings.

Just as with usual quivers it is useful to consider an equivalent category of modules over the path algebra.
Such an analog for valued quivers is the notion of $\mb{F}$-species.
Define $\Gamma_0 = \prod_{i\in Q_0} \mb{F}_{d_i}$ and $\Gamma_1=\bigoplus_{(i,j)\in Q_1} \mb{F}_{d_ic_{i,j}}$.
Notice that $\mb{F}_{d_ic_{i,j}}$ contains both $\mb{F}_{d_i}$ and $\mb{F}_{d_j}$ and thus we have a $\Gamma_0$-$\Gamma_0$-bimodule structure on $\Gamma_1$.
Now we define the $\mb{F}$-species $\Gamma_Q$ to be the tensor algebra $T_{\Gamma_0}(\Gamma_1)$ of $\Gamma_1$ over $\Gamma_0$.
If $\Gamma_Q$ is finite-dimensional, then
it is clear that the indecomposable projective (resp. injective) modules are precisely $P_i=e_i\Gamma_Q$ (resp. $I_i=(\Gamma_Q e_i)^*$) for $i\in Q_0$,
where $e_i$ is the identity element in $\mb{F}_{d_i}$.
The category $\Rep(Q)$ has enough projective and injective objects.
The top of $P_i$ is the simple representation $S_i$ supported on the vertex $i$, which is also the socle of $I_i$.
The minimal projective and injective resolution of simple $S_i$ are given by
\begin{equation} \label{eq:resimple} 0\to \bigoplus_{(i,j)\in Q_1} c_{j,i} P_j \to P_i\to S_i\to 0\quad\text{ and }\quad 0\to S_i\to I_i\to \bigoplus_{(i,j)\in Q_1} c_{i,j} I_j\to 0.\end{equation}

The algebra $\Gamma_Q$ is {\em hereditary}, that is, it has global dimension $1$.
So for $M,N\in\Rep(Q)$, $$\innerprod{M,N}=\dim_{\mb{F}}\Hom_Q(M,N)-\dim_{\mb{F}}\Ext_Q^1(M,N)$$
is a bilinear form only depending on the dimension vectors of $M$ and $N$.
This is called ``Ringel-Euler" form, and we denote the matrix of this form by $E(Q)$.
We also define the matrix $E_l(Q):=(e_{j,i}^l)$ and $E_r(Q):=(e_{i,j}^r)$ by
$$e_{i,j}^l=\begin{cases} 1 & i=j; \\ -c_{j,i} & (i,j)\in Q_1; \\ 0 & \text{otherwise},\end{cases}\qquad
  e_{i,j}^r=\begin{cases} 1 & i=j; \\ -c_{i,j} & (i,j)\in Q_1; \\ 0 & \text{otherwise}.\end{cases}$$
These matrices are related by $E(Q)=E_l(Q)D=DE_r(Q)$, where $D$ is the diagonal matrix with diagonal entries $d_{i,i}=d_i$.

\begin{example}[$G_2$] \label{ex:G2} Consider the valued quiver $1\xrightarrow{3,1} 2$ of type $G_2$ with $d=(1,3)$. Its module category has six indecomposable objects
\begin{enumerate}
\item[$\bullet$] The simple injective $S_1: \mb{F}\to 0$, and its projective cover $P_1: \mb{F}\hookrightarrow \mb{F}_3$;
\item[$\bullet$] The simple projective $S_2: 0\to \mb{F}_3$, and its injective hull $I_2: \mb{F}^3\hookrightarrow \mb{F}_3$;
\item[$\bullet$] The module $M_1: \mb{F}^2\hookrightarrow \mb{F}_3$, which is presented by $P_2\hookrightarrow 2P_1$;
\item[$\bullet$] The module $M_2: \mb{F}^3\hookrightarrow \mb{F}_3^2$, which is presented by $P_2\hookrightarrow 3P_1$.
\end{enumerate}
\end{example}

In this paper, we will encounter two kinds of valued quivers. One is valued quivers $Q$ of Dynkin type,
and the other is bigger valued quivers $\Delta_Q$ and $\DCQ$ constructed from $Q$ (see Section \ref{S:iARt}).
We will define upper cluster algebras attached to the latter.

\subsection{Upper Cluster Algebras} \label{ss:UCA}
We mostly follow \cite{BFZ,FZ4,FP}.
To define the upper cluster algebra, we need to introduce the notion of the quiver mutation.
The {\em mutation} of valued quivers is defined through Fomin-Zelevinsky's mutation of the associated skew-symmetrizable matrix.

Every symmetrizable valued quiver $\Delta$ corresponds to a skew symmetrizable integer matrix $B(\Delta):=-E_l(\Delta)+E_r(\Delta)^T$.
So the entries $(b_{u,v})_{u,v\in \Delta_0}$ are given~by
\begin{align*} b_{u,v}=\begin{cases} c_{u,v}, & \text{if } (u,v)\in \Delta_1\\
-c_{u,v}, & \text{if } (v,u)\in \Delta_1\\
0 & \text{otherwise}.
\end{cases}
\end{align*}
The matrix $B(\Delta)$ is {\em skew symmetrizable} because $DB$ is skew-symmetric for the diagonal matrix $D$.
Conversely, given a skew symmetrizable matrix $B$, a unique valued quiver $\Delta$ can be easily defined such that $B(\Delta)=B$.

\begin{definition} \label{D:Qmu}
The {\em mutation} of a skew symmetrizable matrix $B$ on the direction $u\in \Delta_0$ is given by $\mu_u(B)=(b_{v,w}')$, where
$$b_{v,w}'=\begin{cases} -b_{v,w}, & \text{if } u\in\{v,w\},\\ b_{v,w}+\op{sign}(b_{v,u})\max(0,b_{v,u}b_{u,w}), & \text{otherwise}.\end{cases}$$
\end{definition}
\noindent We denote the induced operation on its valued quiver also by $\mu_u$.

The cluster algebras that we will consider in this paper are skew-symmetrizable cluster algebras of geometric type.
The combinatorial data defining such a cluster algebra is encoded in a {symmetrizable} valued quiver $\Delta$ with {\em frozen vertices}.
Frozen vertices are forbidden to be mutated, and the remaining vertices are {\em mutable}.
Such a valued quiver is called valued {\em ice} quiver (or VIQ in short).
The {\em mutable part} $\Delta^\mu$ is the full subquiver of $\Delta$ consisting of mutable vertices.
In general, to define an (upper) cluster algebra only $\Delta^\mu$ is required to be symmetrizable.
However, in this paper all VIQs happen to be ``globally" symmetrizable.
We usually label the mutable vertices as the first $p$ out of $q$ vertices of $\Delta$.
The {\em restricted} $B$-matrix $B_\Delta$ of $\Delta$ is the first $p$ rows of $B(\Delta)$.

Let $k$ be a field, not necessarily related in any sense to the finite field $\mb{F}$ or the base field in the rest of Part \ref{P:I}.
\begin{definition} \label{D:seeds} 
Let $\mc{F}$ be a field containing $k$.
A {\em seed} in $\mc{F}$ is a pair $(\Delta,\b{x})$ consisting of a VIQ $\Delta$ as above together with a collection $\b{x}=\{x_1,x_2,\dots,x_q\}$, called an {\em extended cluster}, consisting of algebraically independent (over $k$) elements of $\mc{F}$, one for each vertex of $\Delta$.
The elements of $\b{x}$ associated with the mutable vertices are called {\em cluster variables}; they form a {\em cluster}.
The elements associated with the frozen vertices are called
{\em frozen variables}, or {\em coefficient variables}.

A {\em seed mutation} $\mu_u$ at a (mutable) vertex $u$ transforms $(\Delta,\b{x})$ into the seed $(\Delta',\b{x}')=\mu_u(\Delta,\b{x})$ defined as follows.
The new VIQ is $\Delta'=\mu_u(\Delta)$.
The new extended cluster is
$\b{x}'=\b{x}\cup\{x_{u}'\}\setminus\{x_u\}$
where the new cluster variable $x_u'$ replacing $x_u$ is determined by the {\em exchange relation}
\begin{equation} \label{eq:exrel}
x_ux_u' = \prod_{(v,u)\in \Delta_1} x_v^{c_{v,u}} + \prod_{(u,w)\in \Delta_1} x_w^{c_{u,w}}.
\end{equation}
\end{definition}

\noindent We note that the mutated seed $(\Delta',\b{x}')$ contains the same
coefficient variables as the original seed $(\Delta,\b{x})$.
It is easy to check that one can recover $(\Delta,\b{x})$ from $(\Delta',\b{x}')$ by performing a seed mutation again at $u$.
Two seeds $(\Delta,\b{x})$ and $(\Delta',\b{x}')$ that can be obtained from each other by a sequence of mutations are called {\em mutation-equivalent}, denoted by $(\Delta,\b{x})\sim (\Delta',\b{x}')$.

\begin{definition}
The {\em cluster algebra $\mc{C}(\Delta,\b{x})$} associated to a seed $(\Delta,\b{x})$ is defined as the subring of $\mc{F}$
generated by all elements of all extended clusters of the seeds mutation-equivalent to $(\Delta,\b{x})$.
\end{definition}

\noindent Note that the above construction of $\mc{C}(\Delta,\b{x})$ depends only, up to a natural isomorphism, on the mutation equivalence class of the initial VIQ $\Delta$.
In fact, it only depends on the mutation equivalence class of the restricted $B$-matrix of $\Delta$.
So we may drop $\b{x}$ and simply write $\mc{C}(\Delta)$ or $\mc{C}(B_\Delta)$.

An amazing property of cluster algebras is the {\em Laurent Phenomenon}.
\begin{theorem}[\textrm{\cite{FZ1,BFZ}}] \label{T:Laurent}
Any element of a cluster algebra $\mc{C}(\Delta,\b{x})$ can be expressed in terms of the
extended cluster $\b{x}$ as a Laurent polynomial, which is polynomial in coefficient variables.
\end{theorem}

Since $\mc{C}(\Delta,\b{x})$ is generated by cluster variables from the seeds mutation equivalent to $(\Delta,\b{x})$,
Theorem \ref{T:Laurent} can be rephrased as
$$\mc{C}(\Delta,\b{x}) \subseteq \bigcap_{(\Delta',\b{x}') \sim (\Delta,\b{x})}\mc{L}_{\b{x}'},$$
where $\mc{L}_{\b{x}}:=k[x_1^{\pm 1},\dots,x_p^{\pm 1}, x_{p+1}, \dots x_{q}]$.
Note that our definition of $\mc{L}_{\b{x}}$ is slightly different from the original one in \cite{BFZ},
where $\mc{L}_{\b{x}}$ is replaced by the Laurent polynomial $\mc{L}(\b{x}):=k[x_1^{\pm 1},\dots,x_p^{\pm 1}, x_{p+1}^{\pm 1}, \dots x_{q}^{\pm 1}]$.
\begin{definition}
The {\em upper cluster algebra} with seed $(\Delta,\b{x})$ is
$$\uca(\Delta,\b{x}):=\bigcap_{(\Delta',\b{x}') \sim (\Delta,\b{x})}\mc{L}_{\b{x}'}.$$
\end{definition}

Any (upper) cluster algebra, being a subring of a field, is an integral domain (and under our conventions, a $k$-algebra).
Conversely, given such a domain~$R$, one may be interested in identifying $R$ as an (upper) cluster algebra.
The following useful lemma is a specialization of \cite[Proposition 3.6]{FP} to the case when $R$ is a unique factorization domain.

\begin{lemma} \label{L:RCA}
Let $R$ be a finitely generated UFD over $k$.
Suppose that $(\Delta,\b{x})$ is a seed contained in $R$, and each adjacent cluster variable $x_u'$ is also in $R$.
Moreover, each pair in $\b{x}$ and each pair $(x_u,x_u')$ are relatively prime.
Then $R \supseteq \uca(\Delta,\b{x})$.
\end{lemma}


\subsection{$\g$-vectors and Gradings} \label{ss:gv}
Let $\b{x}=\{x_1,x_2,\dots,x_q\}$ be an (extended) cluster.
For a vector $\g\in \mb{Z}^q$, we write $\b{x}^\g$ for the monomial $x_1^{\g(1)}x_2^{\g(2)}\cdots x_q^{\g(q)}$.
For $u=1,2,\dots,p$, we set ${y}_u= \b{x}^{-b_{u}}$ where $b_u$ is the $u$-th row of the matrix $B_\Delta$,
and let ${\b{y}}=\{{y}_1,{y}_2,\dots,{y}_p\}$.

Suppose that an element $z\in \mc{L}(\b{x})$ can be written as
\begin{equation}\label{eq:z} z = \b{x}^{\g(z)} F({y}_1,{y}_2,\dots,{y}_p),
\end{equation}
where $F$ is a rational polynomial not divisible by any ${y}_i$, and $\g(z)\in \mb{Z}^q$.
If we assume that the matrix $B_\Delta$ has full rank,
then the elements ${y}_1,{y}_2,\dots,{y}_p$ are algebraically independent so that the vector $\g(z)$ is uniquely determined \cite{FZ4}.
We call the vector~$\g(z)$~the (extended) {\em $\g$-vector} of $z$ with respect to the pair $(\Delta,\b{x})$.
Definition implies at once that for two such elements $z_1,z_2$ we have that
$\g(z_1z_2) = \g(z_1) + \g(z_2)$.
So the set of all $\g$-vectors in any subalgebra of $\mc{L}(\b{x})$ forms a sub-semigroup of $\mb{Z}^q$.

\begin{lemma}[{\cite[Lemma 5.5]{Fs1}, {\em cf.} \cite{P}}] \label{L:independent} If the matrix $B_\Delta$ has full rank,
then any subset of $\mc{L}(\b{x})$ with distinct well-defined $\g$-vectors is linearly independent over $k$.
\end{lemma}

\begin{definition} \label{D:wtconfig} A {\em weight configuration} $\bs{\sigma}$ of a lattice $\mb{L}\subseteq \mb{R}^m$ on a VIQ $\Delta$ is an assignment for each vertex $v$ of $\Delta$ a weight vector $\bs{\sigma}(v)\in \mb{L}$ such that for each mutable vertex $u$, we have that
\begin{equation} \label{eq:weightconfig}
\sum_{(v,u)\in \Delta_1} c_{v,u}\bs{\sigma}(v) = \sum_{(u,w)\in \Delta_1} c_{u,w}\bs{\sigma}(w).
\end{equation}
The {\em mutation} $\mu_u$ also transforms $\bs{\sigma}$ into a weight configuration $\bs{\sigma}'$ on the mutated quiver $\mu_u(\Delta)$ defined as
\begin{equation*} \label{eq:mu_wt}
\bs{\sigma}'(v) = \begin{cases} \displaystyle \sum_{(u,w)\in \Delta_1} c_{u,w}\bs{\sigma}(w) - \bs{\sigma}(u) & \text{if } v=u, \\ \bs{\sigma}(v) & \text{otherwise}. \end{cases}\end{equation*}
\end{definition}

\noindent
By slight abuse of notation, we can view $\bs{\sigma}$ as a matrix whose $v$-th row is the weight vector $\bs{\sigma}(v)$.
In this matrix notation, the condition \eqref{eq:weightconfig} is equivalent to that $B_\Delta\bs{\sigma}$ is a zero matrix.
So we call the cokernel of $B_\Delta$ as the {\em grading space} of $\uca(\Delta)$.
A weight configuration $\bs{\sigma}$ is called {\em full} if the corank of $B_\Delta$ is equal to the rank of $\bs{\sigma}$.
It is easy to see that for any weight configuration of $\Delta$, the mutation can be iterated.

Given a weight configuration $(\Delta;\bs{\sigma})$,
we can assign a multidegree (or weight) to the upper cluster algebra $\uca(\Delta,\b{x})$ by setting
$\deg(x_v)=\bs{\sigma}(v)$ for $v=1,2,\dots,q$.
Then mutation preserves multihomogeneity.
We say that this upper cluster algebra is $\bs{\sigma}$-graded, and denoted by $\uca(\Delta,\b{x};\bs{\sigma})$.
We refer to $(\Delta,\b{x};\bs{\sigma})$ as a graded seed.
Note that the variables in $\b{y}$ have zero degrees.
So if $z$ has a well-defined $\g$-vector as in \eqref{eq:z}, then $z$ is homogeneous of degree $\g\bs{\sigma}$.

\section{AR-theory of Presentations} \label{S:AR_P}
\subsection{Review of Auslander-Reiten theory} \label{ss:AR_fun}
We briefly review Auslander-Reiten theory for Krull-Schmidt exact categories following \cite{DRSS}. 
The theory were developed originally for module categories of Artin algebras, but without much difficulty most of the theory can be generalized to Krull-Schmidt exact categories.
Readers should consult \cite[Section 2.2]{DRSS} or the standard textbook \cite{ARS} for the basic notions in Auslander-Reiten theory, such as the left and right (minimal) almost split morphisms.

%

Let $k$ be a field, and $\mc{A}$ be a $k$-linear, $\Hom$-finite, and Krull-Schmidt category with an exact structure $\mc{E}$.
So $\mc{E}$ is a class of exact pairs which is closed under isomorphisms satisfying Gabriel-Roiter's axiom (see \cite[1.1]{DRSS}). 
Recall that a pair $(i, d)$ of composable morphisms $L \xrightarrow{i} M \xrightarrow{d} N$ in $\mc{A}$ is called {\em exact} if $i$ is a kernel of $d$ and $d$ is a cokernel of $i$.
If the underlying exact structure $\mc{E}$ is clear, we speak of projective and injective objects rather than $\mc{E}$-projective and $\mc{E}$-injective objects.
The proof of the following proposition coincides with the usual one for module categories.
\begin{proposition}[{\cite[Proposition 2.3]{DRSS}}] \label{P:almostsplit} 
Suppose that $L \xrightarrow{i} M \xrightarrow{d} N$ is an exact pair in $\mc{E}$. Then the following assertions are equivalent. \begin{enumerate}
\item $i$ is left minimal almost split.
\item $d$ is right minimal almost split.
\item $i$ is left almost split and $d$ is right almost split.
\end{enumerate}
\end{proposition}

\begin{definition} An exact pair $L \xrightarrow{i} M \xrightarrow{d} N$ in $\mc{E}$ as in the above proposition is called an {\em almost split pair}.
In this case, $L$ is called the {\em translation} of $N$ denoted by $\tau N$, and $N$ is called the inverse translation of $L$ denoted by $\tau^{-1} L$.
\end{definition}
Such an almost split pair can only exist provided $L$ is indecomposable non-injective and $N$ is indecomposable non-projective.
The exact category $(\mc{A}, \mc{E})$ is said to have almost split pairs if $\mc{A}$ has almost split morphisms and moreover for all indecomposable non-projective objects $N$ 
there exists an almost split pair $L \xrightarrow{i} M \xrightarrow{d} N$ and dually for all indecomposable non-injective objects $L$ there exists an almost split pair $L \xrightarrow{i} M \xrightarrow{d} N$.
The uniqueness of minimal almost split maps shows that almost split pairs $L \xrightarrow{i} M \xrightarrow{d} N$ are uniquely determined by $L$ or $N$. 

\begin{example} Let $A$ be a finite dimensional $k$-algebra, and $\module A$ be the category of finite dimensional (right) $A$-modules.
	{\cite[Theorem V.1.15]{ARS}} says that $\module A$ has almost split pairs, so the translation $\tau$ is defined for every indecomposable non-projective $A$-module.
	It is given by the trivial dual of {\em Auslander's transpose functor} (see \cite[IV.1]{ARS}).
\end{example}

Recall that a morphism $f\in \Hom_{\mc{A}}(M,N)$ is called {\em radical} if $\Id_M + gf$ is invertible for each $g\in \Hom_{\mc{A}}(N,M)$.
If $M$ and $N$ are indecomposable, then this is equivalent to say that $f$ is a non-isomorphism.
We denote by $\op{rad}_{\mc{A}}(M,N)$ the space of all radical morphisms in $\Hom_{\mc{A}}(M,N)$.
We define $\op{rad}_{\mc{A}}^2(M,N)$ to consist of all morphisms of form $gf$, where $f\in\op{rad}_{\mc{A}}(M,L)$ and $g\in\op{rad}_{\mc{A}}(L,N)$ for some $L\in\mc{A}$.
We denote by $\ind(\mc{A})$ the full subcategory of all indecomposable objects in $\mc{A}$.
\begin{definition} For $M,N\in\ind(\mc{A})$, an {\em irreducible morphism} $f:M\to N$ is an element in $\op{rad}_{\mc{A}}(M,N)\setminus \op{rad}_{\mc{A}}^2(M,N)$.
	We denote
	$$\Irr_{\mc{A}}(M,N):=\op{rad}_{\mc{A}}(M,N)/ \op{rad}_{\mc{A}}^2(M,N).$$
\end{definition}

For $M\in\ind(\mc{A})$, $\End_\mc{A}(M)$ is local, then 
$$D_M:= \End_\mc{A}(M)/\op{rad} \End_\mc{A}(M)$$
is a division $k$-algebra.

Let $M=\bigoplus_{i=1}^t m_iM_i$ be an object in $\mc{A}$ with $M_i$ indecomposable and pairwise non-isomorphic.
For $f\in \Hom_{\mc{A}}(M,N)$ with $N$ indecomposable, we can write $f$ as $f=(f_1,\dots,f_t)$ where $f_i=(f_{i,1},\dots,f_{i,m_i}): m_iM_i\to N$.
The following proposition was originally proved for module categories of Artin algebras (see \cite[Proposition VII.1.3]{ARS}) but the proof there also works in our setting.
\begin{proposition} 
$f$ is {\em right minimal almost split} iff. the residual classes of $f_{i,j}$'s in $\Irr_\mc{A}(M_i,N)$ form a $D_N^{\opp}$-basis for all $i$.
There is a similar statement for left minimal almost split morphisms.
\end{proposition}
We also recall a basic fact \cite{ARS} that if $L \xrightarrow{i} M \xrightarrow{d} N$ almost split, then
\begin{equation} \label{eq:irrequal} \dim_{D_{M_i}} \Irr_{\mc{A}}(M_i,N)=\dim_{D_{M_i}^\opp} \Irr_{\mc{A}}(L,M_i).
\end{equation}

\subsection{Presentations} \label{ss:C2A}
In this subsection we briefly review some results from \cite{B} in our setting.
Let $A$ be some finite dimensional $k$-algebra with valued quiver $Q$ (see \cite[III.1]{ARS}).
If you do not know what a valued quiver associated to $A$ is, then you can just take $A$ to be the $\mb{F}$-species defined in Section \ref{ss:VQ}.
Let $C^2A:=\op{Ch}_2(\proj A)$ be the category of projective presentations.
To be more precise, the objects in $C^2A$ are 2-term complexes $P_+\xrightarrow{f} P_-$ in $\proj A$ (with $P_+$ and $P_-$ in some fixed degrees). The morphisms are commutative diagrams.
Let $\mc{E}$ be the class of pairs of morphisms in $C^2A$, which is split exact in both degrees.
It is well known (eg. \cite{B}) that the category $C^2A$ is Krull-Schmidt and $\mc{E}$ is an exact structure on $C^2A$.
By abuse of notation we will denote an exact pair in $C^2A$ by an exact sequence $0\to f\to g\to h\to 0$.

Let $P_i$ be the indecomposable projective module corresponding to $i\in Q_0$.
\begin{definition} For any $\beta\in \mb{Z}_{\geqslant 0}^{Q_0}$ we denote $\bigoplus_{i\in Q_0}\beta(i)P_i$ by $P(\beta)$.
If $P_\pm=P(\beta_\pm)$, then the {\em weight vector} $(\f_-,\f_+)$ of $f$ is $(\beta_-,\beta_+)$.
The {\em reduced weight vector} $\f$ is the difference $\f_+-\f_-$.
\end{definition}

\begin{definition}
Presentations of forms $0\to P,\ P\to 0, $ and $P\xrightarrow{\Id} P$ are called {\em negative, positive}, and {\em neutral}. They are also denoted by $\zero_P^-,\ \zero_P^+$ and $\Id_P$ respectively.
If $P=P_i$, then they are called $i$-th negative, positive, and neutral presentation, and denoted by $\zero_i^-,\ \zero_i^+$ and $\Id_i$ respectively.
\end{definition}


\begin{lemma}\cite{DF} Any presentation $f$ decomposes as $f=f_+\oplus f_{\Id} \oplus f'$, where $f_+$ is positive, $f_{\Id}$ is neutral, and $f'$ is the minimal presentation of $\Coker(f)$.
\end{lemma}

\begin{corollary} \label{C:indK2} An indecomposable presentation is one of the following four kinds.
They are $i$-th negative, positive, neutral presentations, and minimal presentations of indecomposable non-projective representations of $A$.
\end{corollary}

The following lemma is easy to verify.
\begin{lemma}[{\cite[Proposition 3.1]{B}}] \label{L:HomK2} For any $P_+\xrightarrow{f} P_-\in C^2A$ and $P\in\module A$, we have
\begin{enumerate}
\item $\Hom_{C^2A}(\zero_P^-,f) \cong \Hom_A(P,P_-)$;
\item $\Hom_{C^2A}(f,\zero_P^+) \cong \Hom_A(P_+,P)$;
\item $\Hom_{C^2A}(\Id_P, f) \cong \Hom_A(P,P_+)$;
\item $\Hom_{C^2A}(f,\Id_P) \cong \Hom_A(P_-,P)$.
\end{enumerate}
\end{lemma}

\begin{corollary}[{\cite[Corollary 3.1, 3.2]{B}}] \label{C:PI_C2A} The indecomposable $\mc{E}$-projective objects in $C^2 A$ are precisely $\zero_i^-$ and $\Id_i$. 
	The indecomposable $\mc{E}$-injective objects in $C^2 A$ are precisely $\zero_i^+$ and $\Id_i$.
\end{corollary}

Let $f$ and $g$ be two presentations of representations $M$ and $N$, namely, $M=\Coker f$ and $N=\Coker g$. For any morphism in $\varphi\in \Hom_{C^2A}(f,g)$, we get an induced morphism $\phi\in \Hom_A(M,N)$.
$$\vcenter{\xymatrix@R=7ex@C=7ex{
P_+ \ar[r]^f \ar[d]^{\varphi_+} & P_- \ar[r] \ar[d]^{\varphi_-} & M \ar[d]^\phi \ar[r] & 0\\
R_+ \ar[r]^g & R_- \ar[r] & N \ar[r] & 0
}}$$
Conversely, any $\phi\in\Hom_A(M,N)$ lifts to a morphism in $\Hom_{C^2A}(f,g)$.
So we obtain a surjection
$$\pi:\Hom_{C^2A}(f,g)\twoheadrightarrow \Hom_A(\Coker f,\Coker g).$$
$\pi$ maps to a zero morphism if and only if
the image of $\varphi_-$ is contained in the image of $g$.
In this case, $\varphi_-$ lifts to a map in $\Hom_A(P_-,R_+)$ because $P_-$ is projective.
Hence the kernel of $\pi$ is the image of the map $\iota$
$$\iota: \Hom_A(P_-,R_+)\to \Hom_{C^2A}(f,g),\ h\mapsto gh-hf.$$
Recall that $\Hom_A(P_-,R_+)\cong \Hom_{C^2A}(f,\Id_{R^+})$.
So we can summarize the above discussion as follows. The functor $\Coker : C^2A \to \module A$ is full and dense with
the kernel consisting of those morphisms which are factored through positive and neutral presentations.
Let $\overline{C^2A}$ be the category $C^2A$ modulo the morphisms which are factorized through $\mc{E}$-injectives.
\begin{proposition}[{\cite[Proposition 3.3]{B}}] 
	The functor $\Coker$ induces an isomorphism $\overline{C^2A} \cong \module A$.
\end{proposition}

Here is a main result in \cite{B}. 
\begin{theorem}[{\cite[Theorem 5.1]{B}}] \label{T:AR_C2A} The exact category $C^2A$ has almost split pairs.
\end{theorem}
\noindent The next two propositions enable us to construct almost split pairs in $C^2A$.
For an $A$-module $M$, we write $f(M):P_+(M)\to P_-(M)$ for its minimal presentation.
\begin{proposition}[{\cite[Proposition 5.6]{B}}] \label{P:AR_from_mod} If $x: 0\to f\to e \to g\to 0$ is exact in $C^2A$ with $\Coker(f)\neq 0$ and $\Coker(g)\neq 0$,
then $x$ is almost split iff. the induced sequence $\Coker(x): 0\to \Coker(f)\to \Coker(e) \to \Coker(g)\to 0$ is almost split.
\end{proposition}

\begin{proposition}[{\cite[Proposition 5.9 and Corollary 5.3]{B}}] \label{P:AR_Linj} The almost split pair starting at $f(I_i)$ has the form:
\begin{align*}
	0&\to f(I_i) \to f(I_i/\op{soc}(I_i))\oplus \zero_R^+ \to \zero_i^+ \to 0, && \text{ if $I_i$ is not simple;}\\
	0&\to f(I_i) \to  \Id_i\oplus \zero_{P_+(I_i)}^+ \to \zero_i^+ \to 0, && \text{ if $I_i$ is simple,}
\end{align*}
where $\Hom_A(R,A)^*$ is the maximal injective summand of $E$, with $E \to I_i$ a right minimal almost split morphism in $\module A$.
\end{proposition}

\begin{corollary} We have that
	$\tau f(M)=f(\tau M)$ for $M$ non-projective and $\tau(\zero_i^+)=f(I_i)$ in $C^2A$.
\end{corollary}

From now on let us assume $A$ is the $\mb{F}$-species $\Gamma_Q$. We denote $C^2\Gamma_Q$ by $C^2 Q$.

\begin{lemma} \label{L:AR_f} Suppose that $0\to L\to M\to N\to 0$ is an almost split sequence in $\Rep(Q)$. 
Then we have the following almost split pairs in $C^2Q$.
\begin{align} \label{eq:res0} &0\to f(L)\to f(M)\to f(N)\to 0 && \text{ if } L\neq S_i; \\
\label{eq:res00} &0\to f(L)\to f(M)\oplus \Id_i \to f(N)\to 0 && \text{ if } L=S_i. \tag{2.2'}
\end{align}
\end{lemma}

\begin{proof}  Suppose that $L$ is not simple.
We can splice the minimal presentations of $L$ and $N$ together to form a presentation of $M$
$$P_+(L)\oplus P_+(N)\xrightarrow{f} P_-(L)\oplus P_-(N)\to M\to 0.$$
By construction, we have the exact sequence
$0\to f(L)\to f \to f(N)\to 0.$
We claim that $f$ is minimal. This is equivalent to that
$\dim\Hom_Q(M,S_i) = \dim\Hom_Q(L,S_i)+\dim\Hom_Q(N,S_i)$
and $\dim\Ext_Q^1(M,S_i) = \dim\Ext_Q^1(L,S_i) + \dim\Ext_Q^1(N,S_i)$ for each $S_i$.
Since $0\to L\to M\to N\to 0$ is almost split and $L$ is non-simple, it follows that
$$0\to \Hom_Q(N,S_i) \to \Hom_Q(M,S_i)\to \Hom_Q(L,S_i)\to 0,$$
$$0\to \Ext_Q^{1}(N,S_i) \to  \Ext_Q^{1}(M,S_i)\to \Ext_Q^{1}(L,S_i)\to 0$$
are both exact.

In the case when $L=S_i$, by Auslander-Reiten formula \cite[Corollary IV.4.7]{ARS}
$$\Hom_Q(\tau^{-1} L,S_i)=\Ext_Q^1(S_i,S_i)^*=0,\quad \Ext_Q^1(\tau^{-1} L,S_i)=\Hom_Q(S_i,S_i)^*=k.$$
So we have the exact sequence
$$0 \to \Hom_Q(M,S_i) \to \Hom_Q(L,S_i) \cong \Ext_Q^1(\tau^{-1} L,S_i) \to \Ext_Q^1(M,S_i) \to 0.$$
Hence $\Hom_Q(M,S_i)=\Ext_Q^1(M,S_i)=0$.
This implies the exactness of \eqref{eq:res00}.
Finally the claim follows from Proposition \ref{P:AR_from_mod}.
\end{proof}

Similarly the next lemma follows directly from Proposition \ref{P:AR_Linj}.
\begin{lemma} \label{L:res_simpleK2} We have the following almost split pairs in $C^2Q$
	\begin{align}
	\label{eq:res2} &0\to f(I_i)\to \bigoplus_{(k,i)\in Q_1}c_{i,k} f(I_k)\oplus \bigoplus_{(i,j)\in Q_1} c_{j,i} \zero_j^+ \to \zero_i^+  \to 0  \text{ if $i$ is not a source}; \\
	\label{eq:res20} &0\to f(I_i)\to  \Id_i \oplus \bigoplus_{(i,j)\in Q_1} c_{j,i} \zero_j^+   \to \zero_i^+  \to 0  \hspace{.81in} \text{ if $i$ is a source}. \tag{2.3'}
	\end{align}
\end{lemma}

\section{iARt Quivers} \label{S:iARt}

\subsection{iARt Quivers} \label{ss:iARt}
We slightly upgrade the classical Auslander-Reiten quiver by adding the translation arrows. The following definition is basically taken from \cite[VII.1]{ARS}. Let $\mc{A}$ be a category as in Section \ref{ss:AR_fun}.
Recall that for each $M\in \ind(\mc{A})$, $D_M:=\End_\mc{A}(M)/\op{rad} \End_\mc{A}(M)$ is a division $k$-algebra.

\begin{definition}[ARt quiver] The ARt valued quiver $\Delta(\mc{A})$ of $\mc{A}$ is defined as follows: \begin{enumerate}
\item The vertex of $\Delta(\mc{A})$ are the isomorphism classes of objects in $\ind \mc{A}$.
\item There is a {\em morphism arrow} $M\to N$ if $\Irr_{\mc{A}}(M,N)$ is non-empty. We assign the valuation $(a,b)$ to this arrow, where $a=\dim_{D_M}\Irr_{\mc{A}}(M,N)$ and $b=\dim_{D_N^\opp}\Irr_{\mc{A}}(M,N)$.
\item There is a {\em translation arrow} from $N$ to $\tau N$ with trivial valuation if $\tau N$ is defined.
\end{enumerate}
A vertex $u$ in an ARt quiver is called {\em transitive} if the translation and its inverse are both defined at $u$.
\end{definition}
\noindent Note that the number $a$ in the valuation $(a,b)$ can be alternatively interpreted as
the (direct sum) multiplicity of $M$ in $E$ for $E\to N$ right minimal almost split.
Similarly $b$ is the multiplicity of $N$ in $E'$ for $M\to E'$ left minimal almost split.
Moreover, if $\mc{A}$ is {\em $k$-elementary}, i.e., $D_M=k$ for any $M\in \ind(\mc{A})$, then all morphism arrows are equally valued.

\begin{definition}[iARt quivers] \label{D:iARt} {\ }
\begin{enumerate}
\item[$\bullet$] The iARt quiver $\Delta_{\mc{A}}$ is obtained from the ARt quiver $\Delta(\mc{A})$ by freezing all vertices whose translations or inverse translations are not defined.
\item[$\bullet$] The iARt quiver $\Delta_{\mc{A}}^2$ is obtained from the ARt quiver $\Delta(C^2\mc{A})$ by freezing all non-transitive vertices.
\end{enumerate}
\end{definition}

\begin{remark} When $\mc{A}$ is the module category of a finite-dimensional algebra, the frozen vertices of $\Delta_{\mc{A}}$ are precisely indecomposable projective modules.
By Theorem \ref{T:AR_C2A} and Corollary \ref{C:PI_C2A}, the frozen vertices in $\Delta_{\mc{A}}^2$ are precisely the negative, positive and neutral presentations.
\end{remark}

We use the notation \begin{align*}
&L\to M & & \text{if there is a morphism arrow from $L$ to $M$;}\\
&L\dashrightarrow M & &\text{if there is a translation arrow from $L$ to $M$;}\\
&L\rightarrowtail M & &\text{if there is an arrow from $L$ to $M$ in the ARt quiver.}
\end{align*}

\begin{lemma} \label{L:additive} Let $\theta$ be an additive function from $\mc{A}$ to some abelian group, that is,
$\theta(L)+\theta(N)=\theta(M)$ for each exact sequence $0\to L\to M \to N\to 0$ in $\mc{A}$.
Then at each transitive vertex $L$, we have that
$\sum_{M\rightarrowtail L} c_{M,L}\theta(M) = \sum_{L\rightarrowtail N} c_{L,N}\theta(N)$.
\end{lemma}

\begin{proof}  We have two almost split sequences
\begin{align*}
& 0\to \tau L\to \bigoplus_{M\to L} c_{M,L} M \to L\to 0,\\
& 0\to L\to \bigoplus_{L\to N} c_{L,N} N\to \tau^{-1}L \to 0.
\end{align*}
By the additivity of $\theta$, we have that
\begin{gather*}
\theta(L)=\sum_{M\to L} c_{M,L}\theta(M) -\theta(\tau L) = \sum_{L\to N} c_{L,N}\theta(N)-\theta(\tau^{-1}L),\\
\Rightarrow \quad \sum_{M\rightarrowtail L} c_{M,L}\theta(M) = \sum_{L\rightarrowtail N} c_{L,N}\theta(N).
\end{gather*}
\end{proof}

\noindent A typical additive function in $C^2A$ is the weight vector.
In some special cases including examples below, indecomposable presentations are uniquely determined by their weight vectors.
So we can label them on an iARt quiver by their weight vectors.
We will use the ``exponential form" as a shorthand.
For example, a vector $(3,1,0,0,-2,0,-1)$ is written as $(5^27,1^32)$.

\begin{example} \label{ex:D4mu} Let $A$ be the Jacobian algebra of the quiver with potential $(Q,W)$ (see Section \ref{ss:QP}), where
$Q=\vcenter
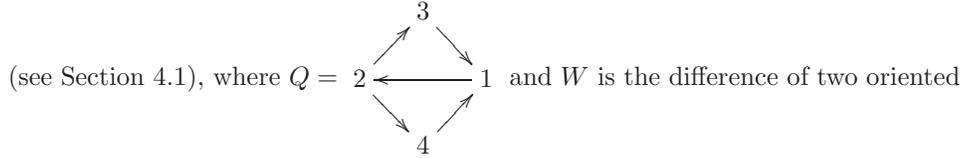
{\xymatrix@R=3ex@C=3ex{
& {3} \ar[dr] &  \\
{2} \ar[ur] \ar[dr] &&  {1} \ar[ll] \\
& {4} \ar[ur] }
}$ and $W$ is the difference of two oriented triangles.
The iARt quiver $\Delta_{\module A}^2$ is drawn below. We always put frozen vertices in boxes.
$$\iARtDfourmu$$
Two vertices with the same weight label $(1,2)$ are identified.
The translation arrow going out from $(0,2)$ ends in $(1,0)$.
\end{example}

\subsection{Hereditary Cases} \label{ss:gdim1}
\noindent In particular, if we take $\mc{A}:=\Rep(Q)$ for some valued quiver $Q$, we get two ARt quivers $\Delta(Q):=\Delta(\Rep(Q))$ and $\Delta(C^2Q)$. We denote the corresponding iARt quivers by $\Delta_Q$ and $\DCQ$.

\begin{proposition} \label{P:iARtDCQ} The ARt quiver $\Delta(C^2 Q)$ can be obtained from $\Delta(Q)$ as follows.
\begin{enumerate}
\item We add $|Q_0|$ vertices corresponding to $\zero_i^+$ and another $|Q_0|$ vertices corresponding to $\Id_i$.
\item For each $i\xrightarrow{(a,b)} j$ in $Q$, we draw morphism arrows $\zero_j^+\xrightarrow{(a,b)} \zero_i^+$,
and $f(I_i)\xrightarrow{(b,a)} \zero_j^+$. We add translation arrows from $\zero_i^+$ to $f(I_i)$.
\item We draw morphism arrows $f(S_i)\xrightarrow{} \Id_i$, and $\Id_i\xrightarrow{} \tau^{-1}f(S_i)$.
\end{enumerate}
\end{proposition}

\begin{proof} The vertices of $\Delta(Q)$ are identified with vertices of $\Delta(C^2Q)$ via minimal presentations.
By Corollary \ref{C:indK2}, we only need to add the vertices as in (1).
Step (2) is due to \eqref{eq:res2} and \eqref{eq:res20}.
Note that $\tau^{-1}f(S_i)$ is equal to $f(\tau^{-1}S_i)$ if $i$ is not a source, otherwise it is equal to $\zero_i^+$.
So Step (3) is due to \eqref{eq:res00} and \eqref{eq:res20}.
We do not need anything else because of \eqref{eq:res0} and the easy fact that $\bigoplus_{(i,j)\in Q_1}c_{j,i}\zero_j^- \to \zero_i^-$ is right minimal almost split.
\end{proof}
\noindent Due to this proposition, we will freely identify $\Delta_{Q}$ as a subquiver of $\DCQ$.

From now on, we let $Q$ be a valued quiver of Dynkin type.
In this case, any indecomposable presentation $f$ is uniquely determined by its weight vector.
The quiver $\Delta_Q$ was already consider in \cite{BFZ,GLSa}.
In \cite{BFZ} the authors associated an ice quiver to any reduced expression of the longest element $w_0$ in the Weyl group of $Q$.
The iARt quiver $\Delta_Q$ only corresponds to those reduced expressions {\em adapted} to $Q$.

\begin{example} The iARt quiver $\DCQ$ for $Q$ of type $A_n$ is the ice hive quiver $\Delta_n$ constructed in \cite{Fs1} up to some arrows between frozen vertices.
\end{example}

\begin{example}[The iARt quiver $\DCQ$ for $Q$ a $D_4$-quiver] 
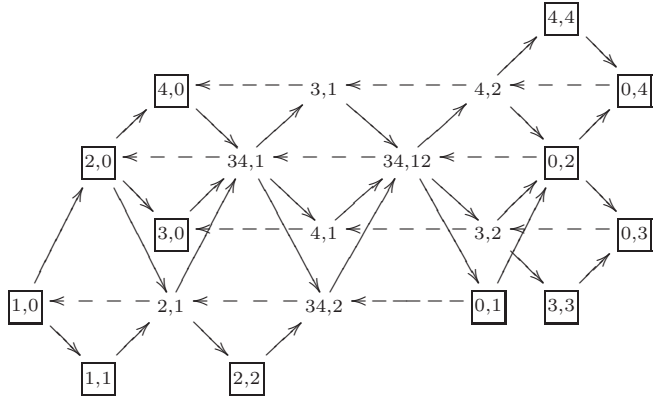
\label{ex:iARtD4}
$$\iARtDfour$$
Readers can find a few other iARt quivers in Appendix \ref{A:list}.
\end{example}

\begin{remark} One natural question is whether the iARt quivers $\Delta_Q^2$ and $\Delta_{Q'}^2$ (or $\Delta_Q$ and $\Delta_{Q'}$) are mutation-equivalent if $Q$ and $Q'$ are reflection-equivalent.
The answer is positive at least in trivially valued cases. We conjecture that this is also true in general.
As pointed in \cite[Remark 2.14]{BFZ}, for $\Delta_Q$ with $Q$ trivially valued, by the Tits lemma every two reduced words can be obtained form each other by a sequence of elementary 2- and 3-moves (see \cite[Section 2.1]{SSVZ});
by \cite[Theorem 3.5]{SSVZ} every such move either leaves the seed unchanged, or replaces it by an adjacent seed.
Finally, similar to the proof of Corollary \ref{C:mu_sqrtl}, the result can be extended from $\Delta_Q$ to $\Delta_Q^2$.

However, reflection-equivalent quivers can not replaced with mutation-equivalent Jacobian algebras (see Section \ref{ss:QP}). 
The Jacobian algbera in Example \ref{ex:D4mu} is obtained from the above path algebra of $D_4$ by mutating at the vertex 2.
According to \cite{Ke}, the iARt quiver in Example~\ref{ex:iARtD4} is mutation-equivalent to a finite mutation type quiver $E_6^{(1,1)}$, while the one in Example \ref{ex:D4mu} is mutation equivalent to a wild acyclic quiver, which is of infinite mutation type.
\end{remark}

It follows from \eqref{eq:irrequal} and the fact that each $Q$ is symmetrizable that
\begin{lemma} The $B$-matrix of $\DCQ$ is skew-symmetrizable.
\end{lemma}

It was constructed in \cite[Theorem 8.3]{BZ} a family of compatible pairs $\{(B(\mathrm{i}),\Lambda(\mathrm{i}))\}_{\mathrm{i}}$, i.e, $B(\mathrm{i})$ and $\Lambda(\mathrm{i})$ satisfy that $B(\mathrm{i})\Lambda(\mathrm{i}) = (I,0)$.
The family $\{B(\mathrm{i})\}_{\mathrm{i}}$ contains the restricted $B$-matrix of $\Delta_Q$. It follows that
\begin{lemma} The restricted $B$-matrices of $\Delta_Q$, and thus of $\DCQ$, have full ranks.
\end{lemma}

By Lemma \ref{L:additive}, the assignment $f\mapsto (\f_-,\f_+)$ is a weight configuration of $\DCQ$.
However, it is not full (see Section \ref{ss:gv}).
We want to extend it to a full one which is useful for the second half of the paper.
Since $Q$ is of finite representation type, each non-neutral $f\in \ind(C^2Q)$ is translated from a unique indecomposable positive presentation, that is, $f=\tau^t(\zero_i^+)$ for some $i\in Q_0$ and $t\in \mb{Z}_{\geqslant 0}$.
Now for each $f\in \ind(C^2Q)$, we assign a triple-weight vector as follows.
\begin{definition} \label{D:triplewt} If $f$ is translated from $\zero_{i}^+$,
then the {\em triple weights} $\wtd{\f}\in \mb{Z}^{3|Q_0|}$ of $f$ is given by $(\e(f), \f_-, \f_+)$ where $\e(f):=\e_i$ the unit vector supported on $i$.
If $f=\Id_i$, then we set $\wtd{\f}:=(0,\e_i,\e_i)$.
We also define another weight vector $\br{\f}\in \mb{Z}^{|Q_0|}$ attached to $f$ by $\br{\f}:=\e_i+\f_--\f_+$.
\end{definition}

\begin{corollary} \label{C:WC} The assignment $\sCQ: f\mapsto \wtd{\f}$ (resp. $\sQ: f \mapsto \br\f$) defines a full weight configuration for the iARt quiver $\DCQ$ (resp. $\Delta_Q$).
\end{corollary}

\begin{proof} Due to Lemma \ref{L:additive}, it suffices to show that for each mutable $f$
$$\sum_{g\rightarrowtail f} c_{g,f} \e(g) = \sum_{f\rightarrowtail h} c_{f,h} \e(h).$$
We call a vertex $u$ {\em regular} if it is transitive and $\tau^{-1}v$ is defined for each $v\to u$, and $\tau w$ is defined for each $u\to w$.
It is clear from \eqref{eq:irrequal} that the equation holds at each regular vertex.
From the description of Proposition \ref{P:iARtDCQ}, we see that all transitive vertices are regular except for $f(S_i)$ and $\tau^{-1}f(S_i)$.
The problem is that these vertices may have (morphism) arrows to neutral frozen vertices, whose translation is not defined.
But the first component of the triple weights of $\Id_i$ is a zero vector so the equality still holds at these vertices.

For the case of $\Delta_Q$, it is enough to observe that the weight vector $\br{\f}$ is zero on the positive and neutral frozen vertices of $\DCQ$, and $\Delta_Q$ is obtained from $\DCQ$ by deleting these vertices.
\end{proof}

\noindent We shall consider the graded upper cluster algebra $\uca(\DCQ;\sCQ)$ and the graded cluster algebra $\mc{C}(\Delta_Q;\sQ)$ later.

\section{Cluster Character from Quivers with Potentials} \label{S:CCQP}
\subsection{Quivers with Potentials} \label{ss:QP}
The mutation of quivers with potentials is invented in \cite{DWZ1} and \cite{DWZ2} to model the cluster algebras.
In this and next section and Appendix \ref{A:cyclic}, we switch back to the usual quiver notation. A quiver $\Delta$ is a quadruple $(\Delta_0,\Delta_1,h,t)$,
where the maps $h$ and $t$ map an arrow $a\in \Delta_1$ to its head and tail $h(a), t(a)\in \Delta_0$.
Following \cite{DWZ1}, we define a potential $W$ on an ice quiver $\Delta$ as a (possibly infinite) linear combination of oriented cycles in $\Delta$.
More precisely, a {\em potential} is an element of the {\em trace space} $\Tr(\ckQ):=\ckQ/[\ckQ,\ckQ]$,
where $\ckQ$ is the completion of the path algebra $k\Delta$ and $[\ckQ,\ckQ]$ is the closure of the commutator subspace of $\ckQ$.
The pair $(\Delta,W)$ is an {\em ice quiver with potential}, or IQP for short.
For each arrow $a\in \Delta_1$, the {\em cyclic derivative} $\partial_a$ on $\widehat{k\Delta}$ is defined to be the linear extension of
$$\partial_a(a_1\cdots a_d)=\sum_{k=1}^{d}a^*(a_k)a_{k+1}\cdots a_da_1\cdots a_{k-1},$$
where $a^*(b)=1$ if $a=b$ and zero otherwise.
For each potential $W$, its {\em Jacobian ideal} $\partial W$ is the (closed two-sided) ideal in $\ckQ$ generated by all $\partial_a W$.
The {\em Jacobian algebra} $J(\Delta,W)$ is the quotient algebra $\widehat{k\Delta}/\partial W$.
\footnote{Unlike the definition in \cite{BIRS}, we need to include $\partial_a W$ in the ideal $\partial W$ even if $a$ is an arrow between frozen vertices.}
If $W$ is polynomial and the quotient of $k\Delta$ by the unclosed ideal generated by all $\partial_a W$ is finite-dimensional, then the completion is unnecessary to define $J(\Delta,W)$.
This is the case throughout this paper.

The key notion introduced in \cite{DWZ1,DWZ2} is the {\em mutation} of quivers with potentials and their decorated representations.
For an ice quiver with {\em nondegenerate potential} (see \cite{DWZ1}), the mutation in certain sense ``lifts" the mutation in Definition \ref{D:Qmu}.
We have a short review in Appendix \ref{ss:mu}.

\begin{definition} A {\em decorated representation} of a Jacobian algebra $J:=J(\Delta,W)$ is a pair $\mc{M}=(M,M^+)$,
where $M\in \Rep(J)$, and $M^+$ is a finite-dimensional $k^{\Delta_0}$-module.
\end{definition}

Let $\mc{R}ep(J)$ be the set of decorated representations of $J(\Delta,W)$ up to isomorphism.
Let $K^2 J$ be the homotopy category of $C^2 J$.
There is a bijection between two additive categories $\mc{R}ep(J)$ and $K^2 J$ mapping any representation $M$ to its minimal presentation in $\Rep(J)$, and the simple representation $S_u^+$ of $k^{\Delta_0}$ to $P_u\to 0$.
Suppose that $\mc{M}$ corresponds to a projective presentation
$P(\beta_+)\to P(\beta_-)$.

\begin{definition} \label{D:gv} The {\em $\g$-vector} $\g(\mc{M})$ of a decorated representation $\mc{M}$ is the reduced weight vector $\beta_+-\beta_-$.
\end{definition}

\begin{definition} A potential $W$ is called {\em rigid} on a quiver $\Delta$ if
every potential on $\Delta$ is cyclically equivalent to an element in the Jacobian ideal $\partial W$.
Such a QP $(\Delta,W)$ is also called {\em rigid}.
A potential $W$ is called {\em $\mu$-rigid} on an ice quiver $\Delta$ if its restriction to the mutable part $\Delta^\mu$ is rigid.
\end{definition}

\noindent It is known \cite[Proposition 8.1, Corollary 6.11]{DWZ1} that every rigid QP is $2$-acyclic, and the rigidity is preserved under mutations. In particular, any rigid QP is nondegenerate.

\begin{definition} Two QPs $(\Delta,W)$ and $(\Delta',W')$ on the same vertex set $\Delta_0$ are call {\em right-equivalent} if there is an isomorphism $\varphi: k\Delta\to k\Delta'$ such that $\varphi\mid_{k\Delta_0}=\Id$ and $\varphi(W)$ is cyclically equivalent to $W'$.
Two IQPs $(\Delta,W)$ and $(\Delta',W')$ are call {\em $\mu$-right-equivalent} if they are right-equivalent when restricted to their mutable parts.
\end{definition}


\begin{definition} \label{D:mu_sup} A representation is called {\em $\mu$-supported} if its supporting vertices are all mutable.
We denote by $\mc{R}ep^{\mu}(J)$ the full subcategory of all $\mu$-supported decorated representations of $J$.
\end{definition}

\begin{remark} \label{r:mu_equi} If $(\Delta,W)$ and $(\Delta',W')$ are $\mu$-right-equivalent,
then $\mc{R}ep^\mu(J)$ and $\mc{R}ep^\mu(J')$ are equivalent.
Indeed, we write $W=W_\mu+W_\diamond$ where $W_\mu$ is the restriction of $W$ to the mutable part of $\Delta$.
We find that any cyclic derivative $\partial_a W_\diamond$ is a sum of paths passing some frozen vertices.
Such a sum gives rise to a trivial relation on the $\mu$-supported representations.
\end{remark}

\subsection{The Generic Cluster Character} \label{ss:CC}
\begin{definition}
To any $\g\in\mathbb{Z}^{\Delta_0}$ we associate the {\em reduced} presentation space $$\PHom_J(\g):=\Hom_J(P([\g]_+),P([-\g]_+)),$$
where $[\g]_+$ is the vector satisfying $[\g]_+(u) = \max(\g(u),0)$.
We denote by $\Coker(\g)$ the cokernel of a general presentation in $\PHom_J(\g)$.
\end{definition}
\noindent Reader should be aware that $\Coker(\g)$ is just a notation rather than a specific representation.
If we write $M=\Coker(\g)$, this simply means that we take a presentation general enough (according to context) in $\PHom_J(\g)$, then let $M$ to be its cokernel.

\begin{definition} \label{D:mu_supported}
A $\g$-vector $\g$ is called {\em $\mu$-supported} if $\Coker(\g)$ is $\mu$-supported.
Let $G(\Delta,W)$ be the set of all $\mu$-supported $\g$-vectors in $\mb{Z}^{\Delta_0}$.
\end{definition}
\noindent It turns out that for a large class of IQPs the set $G(\Delta,W)$ is given by lattice points in some rational polyhedral cone.
Such a class includes the IQPs introduced in \cite{Fs1,Fs2,Fk1}, and the ones to be introduced in Section \ref{ss:iARt2}.

\begin{definition}[\cite{P,Du}]
We define the {\em generic character} $C_W:G(\Delta,W)\to \mb{Z}(\b{x})$~by
\begin{equation} \label{eq:genCC}
C_W(\g)=\b{x}^{\g} \sum_{\e} \chi\big(\Gr^{\e}(\Coker(\g)) \big) {\b{y}}^{\e},
\end{equation}
where $\Gr^{\e}(M)$ is the variety parametrizing $\e$-dimensional quotient representations of $M$, and $\chi(-)$ denotes the topological Euler-characteristic.
\end{definition}

\begin{theorem}[{\cite[Corollary 5.14]{Fs1}, {\em cf.} \cite[Theorem 1.1]{P}}] \label{T:GCC} Suppose that IQP $(\Delta,W)$ is non-degenerate and $B_\Delta$ has full rank.
The generic character $C_W$ maps $G(\Delta,W)$ (bijectively) to a set of linearly independent elements in $\br{\mc{C}}(\Delta)$ containing all cluster monomials.
\end{theorem}
%

\begin{definition} \label{D:model} We say that an IQP $(\Delta,W)$ {\em models} an algebra $\mc{C}$ if the generic cluster character maps $G(\Delta,W)$ (bijectively) onto a basis of $\mc{C}$.
If $\mc{C}$ is the upper cluster algebra $\uca(\Delta)$, then we simply say that $(\Delta,W)$ is a {\em cluster model}.
\end{definition}

\begin{remark} \label{r:cluster_model}
Suppose that $(\Delta,W)$ and $(\Delta',W')$ are $\mu$-right-equivalent.
By Remark \ref{r:mu_equi}, $\mc{R}ep^\mu(J)$ and $\mc{R}ep^\mu(J')$ are equivalent via some isomorphism $\varphi: k\Delta^\mu\to k(\Delta')^\mu$.
By abuse of notation we denote the equivalence also by $\varphi$.
Since $\varphi\mid_{\Delta_0}=\Id$, $\varphi(M)$ and $M$ are isomorphic and have the same $\g$-vector.
We see from \eqref{eq:genCC} that if $(\Delta,W)$ is a cluster model, then so is $(\Delta',W')$.
\end{remark}


\section{iARt QPs} \label{S:iARtQP}

\subsection{The iARt QP $(\Delta_Q^2,W_Q^2)$} \label{ss:iARt2}
For the time being, let us assume $Q$ is a trivially valued Dynkin quiver.
A {\em translation triangle} in an iARt quiver is an oriented cycle of the form
$$\vcenter
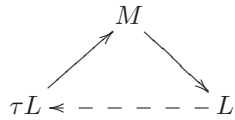
{\xymatrix@R=5ex@C=5ex{
& M \ar[dr]\\
\tau L \ar[ur] &&  L \ar@{-->}[ll]
}}$$
For each iARt quiver $\DCQ$, we define the potential $W_Q^2$ as an alternating sum of all translation triangles.
We make this more precise as follows.
We can also label each non-neutral $f=\tau^{-t} \zero_i^-$ by the pair $(i,t)^*=(i^*(f),t^*(f))$.
The arrows of $\DCQ$ are thus classified into three classes
\begin{align*}
& \emph{Type $A$ arrows} & (i,t)^* &\to (j,t+1)^*,\ f(S_i)\to \Id_i;\\
& \emph{Type $B$ arrows} & (i,t)^* &\to (j,t)^*,\ \Id_i\to \tau^{-1}f(S_i); \\
& \emph{Type $C$ arrows} & (i,t)^* &\dashrightarrow (i,t-1)^*.
\end{align*}
Let $\dot{a}$ (resp. $\dot{b}$ and $\dot{c}$) denote the sum of all type $A$ ($B$ and $C$) arrows.
The potential $W_Q^2$ is defined as $\dot{a}\dot{c}\dot{b}-\dot{a}\dot{b}\dot{c}$.
Thus the Jacobian ideal is generated by the elements
\begin{align}
\label{eq:rel} & e_u(\dot{a}\dot{c}-\dot{c}\dot{a})e_v,\ e_u(\dot{c}\dot{b}-\dot{b}\dot{c})e_v,\\
\label{eq:rel'} & e_u(\dot{b}\dot{a}-\dot{a}\dot{b})e_v. & &  
\end{align}
\noindent Let $J:=J(\Delta_Q^2,W_Q^2)$ be the Jacobian algebra of $(\Delta_Q^2,W_Q^2)$.
For the rest of this section, we denote a single arrow by the lowercase letter of its type with some superscript (eg. $a$ and $a'$).
We observe that for each non-neutral vertex in $\Delta_Q^2$ there are exactly one incoming arrow and one outgoing arrow of type $C$.
Moreover, if non-neutral $u$ and $v$ are connected by an arrow of type $B$ or $A$, then the relations \eqref{eq:rel} say that 
\begin{equation} \label{eq:movec} \tag{5.1.1}
\text{$e_u ac e_v= e_u c'a' e_v\ $ or $\ e_u cb e_v= e_ub'c'e_v$ \  for some $a,a',b,b',c,c'$.}
\end{equation}
In general, the relations \eqref{eq:rel'} do not have a similar implication because there is a trivalent vertex for $Q$ of type $D$ or $E$.
We have that 
\begin{equation} \label{eq:exchangeab} \tag{5.2.1}
\text{$e_u ba e_v= \sum_{a',b'} (e_u a'b' e_v$ or $e_u b'a' e_v$).}
\end{equation}
If $u$ and $v$ are not trivalent, then the right sum has only one summand.
If both $u$ and $v$ are negative (resp. positive), then some translation arrows are undefined so the relations \eqref{eq:movec} reduce to the following 
\begin{equation} \tag{5.1.2} \label{eq:relpm}
\text{$e_uac e_v=0$ (resp. $e_u ca e_v= 0$ or $e_u cb e_v= 0$).}
\end{equation}
Similarly if $u$ (resp. $v$) is neutral, then
\begin{equation} \tag{5.1.3} \label{eq:relne}
\text{$e_u bc e_v= 0$ (resp. $e_u ca e_v=0$).}
\end{equation}

\begin{lemma} \label{L:fdrigid} The IQP $(\DCQ,W_{Q}^2)$ is rigid and $J(\DCQ,W_{Q}^2)$ is finite-dimensional.
\end{lemma}
\begin{proof} To show that $J(\DCQ,W_{Q}^2)$ is finite-dimensional, it suffices to observe that any nonzero path from $f$ to $g$ in $J$ can be uniquely identified as an element in $\bigoplus_{t=0}^{t^*(f)} \Hom_{C^2Q}(\tau^t f, g)$.
Indeed, suppose that $p$ is a path from $f$ to $g$. By \eqref{eq:exchangeab} and \eqref{eq:relne}, we can make $p$ avoid any neutral vertex.
By \eqref{eq:movec} we can move all arrows of type $C$ to the left. If we remove all arrows of type $C$, then the truncated path can be interpreted as a morphism from $\tau^t f$ to $g$.
	
Due to relations \eqref{eq:movec} and \eqref{eq:exchangeab}, any cycle in the Jacobian algebra is equivalent to a sum of composition of cycles $e_u(acb)e_{u'}$ with $u$ and $u'$ mutable. 
It suffices to show that each $e_uacbe_{u'}$ is in fact zero in the Jacobian algebra. 
Applying the relation \eqref{eq:movec} twice (if $u$ is not negative), we see that $e_uacbe_{u'}$ is equivalent to $e_{w}a'c'b'e_u$ where $u$ and $w$ are connected by an arrow of type $C$.
If $u$ is mutable, then there is some negative vertex $v$ connected to $u$ by arrows of type $C$.
So $e_uacbe_{u'}$ is equivalent to $e_va''c''b''$, which is zero by \eqref{eq:relpm}. 
\end{proof}

We delete all translation arrows of $\Delta_Q^2$, and obtain a subquiver denoted by $\wedge_Q^2$.
Let $R$ be the direct sum of all presentations in $\ind(C^2Q)$.
It is well-known that the {\em Auslander algebra} $A_{Q}^2:=\End_{C^2Q}(R)$ is equal to $k{\wedge}_Q^2$ modulo the {\em mesh relations} \cite[VII.1]{ARS}.
The Auslander algebra $A_{Q}^2$ is the quotient of Jacobian algebra $J$ by the ideal generated by translation arrows.

Let $f:P_+ \xrightarrow{} P_-$ be a presentation in $\ind(C^2Q)$.
We denote by $\P_f$ (resp. $\I_f$) the indecomposable projective (resp. injective) representation of $J$ corresponding to the vertex $f$.
\begin{lemma} \label{L:iARt_path} We have the following for the module $\P_f$ in $J(\Delta_Q^2,W_Q^2)$
\begin{enumerate}
\item $\P_f(\zero_i^-)\cong \Hom_Q(P_{i^*(f)}, P_i)$;
\item $\P_f(\zero_i^+)\cong \Hom_Q(P_+,P_i)$;
\item $\P_f(\Id_i)\cong \Hom_Q(P_-, P_i)$.
\end{enumerate}
\end{lemma}

\begin{proof} (1). Recall from the proof of Lemma \ref{L:fdrigid} that $\P_f(\zero_i^-)\subset \bigoplus_{t=0}^{t^*(f)} \Hom_{C^2Q}(\tau^t f, \zero_i^-)$.
We know that $\Hom_{C^2(Q)}(f,\zero_i^-)=0$ unless $f$ is negative. This implies that $\P_f(\zero_i^-)\subset \Hom_Q(P_{i^*(f)}, P_i)$.
Conversely, we identify an element in $\Hom_Q(P_{i^*(f)}, P_i)$ by a path consisting solely of arrows of type $B$. 
By adjoining $t^*(f)$ arrows of type $C$ we get a path from $f$ to $\zero_i^-$, which is easily seen to be nonzero.

(2). We observe that any path from $f$ to $\zero_i^+$ containing a translation arrow must be equivalent to zero due to the relations \eqref{eq:rel} and \eqref{eq:relpm}.
So $\P_f(\zero_i^+)$ is the same as $\P_f(\zero_i^+)$ restricted to the Auslander algebra.
The result follows from Lemma \ref{L:HomK2}.(4).

(3). Similar to (2), any path from $f$ to $\Id_i$ containing a translation arrow must be equivalent to zero.
The result follows from Lemma \ref{L:HomK2}.(2).
\end{proof}

\subsection{The Cone ${\sf G}_{\Delta_Q^2}$} \label{ss:cone}
We consider the following set of representations $T_v$
\begin{align}
\label{eq:injres1} & 0\to T_v \to \I_v \to \bigoplus_{i\to j} \I_{\zero_j^-} & & \text{for $v=\zero_i^-$},\\
\label{eq:injres2} & 0\to T_v \to \I_v \to \bigoplus_{i\to j} \I_{\zero_j^+} & & \text{for $v=\zero_i^+$},\\
\label{eq:injres3} & 0\to T_v \to \I_v \to \bigoplus_{i\to j} \I_{\Id_j} & & \text{for $v=\Id_i$}.
\end{align}
The maps in \eqref{eq:injres1} and \eqref{eq:injres2} are canonical, that is, the map from $\I_{\zero_i^\pm}$ to $\I_{\zero_j^\pm}$ is given by the morphism arrow $\zero_j^\pm\to \zero_i^\pm$.
For the map from $\I_{\Id_i}$ to $\I_{\Id_j}$, let us recall from Lemma \ref{L:AR_f}.(3) that $\Hom(\I_{\Id_i},\I_{\Id_j})\cong \P_{\Id_j}(\Id_i)\cong \Hom_Q(P_j, P_i)$. 
We take the map to be the irreducible map in $\Hom_Q(P_j, P_i)$.
It will follow from the proof of Theorem \ref{T:Tv} that the rightmost maps in \eqref{eq:injres1}--\eqref{eq:injres3} are all surjective.

For $j\in Q_0$, let $j^*=i^*(\zero_j^+)\in Q_0$.
It is well-known that $j\mapsto j^*$ is a (possibly trivial) involution.
The involution does not depend on the orientation of $Q$. Its formula is listed in \cite[Section 2.3]{GLSa}.
For any map between projective modules $f:P(\beta_1)\to P(\beta_2)$,
the {\em $i$-$th$ top restriction} of $f$ is the induced map $\op{top} P(\beta_1(i)) \to \op{top}P(\beta_2(i))$.
\begin{theorem} \label{T:Tv} We have the following description for the modules $T_v$. \begin{enumerate}
\item The module $T_{\zero_{i}^-}$ is the indecomposable module supported on all vertices translated from $\zero_{i^*}^+$ with dimension vector $(1,1,\dots,1)$;
\item The defining linear map $f\to g$ in $T_{\zero_{i}^+}$ is given by the $i$-th top restriction of $\varphi_+$;
\item The defining linear map $f\to g$ in $T_{\Id_{i}}$ is given by the $i$-th top restriction of $\varphi_-$,
\end{enumerate}
where $\varphi=(\varphi_+,\varphi_-)$ is the irreducible morphism from $f$ to $g$.

In particular, the dimension vector $\theta_v$ of $T_v$ is given by \begin{align*}
& \theta_v(f) = \e(f)(i^*) & \text{for } v=\zero_i^-,\\
& \theta_v(f) = \f_+(i) & \text{for } v=\zero_i^+,\\
& \theta_v(f) = \f_-(i) & \text{for } v=\Id_i.
\end{align*}
\end{theorem}

\begin{proof} By Lemma \ref{L:iARt_path}.(1), we can identify $\I_{\zero_i^-}(f)$ with $\Hom_{Q}(P_{i^*(f)},P_i)^*$.
From the definition of $T_{\zero_i^-}$ and the exact sequence
$$ 0 \to \bigoplus_{i\to j} \Hom_{Q}(P_{i^*(f)},P_j) \rightarrow \Hom_{Q}(P_{i^*(f)},P_i) \rightarrow \Hom_{Q}(P_{i^*(f)},S_i)\to 0 ,$$
we conclude that $T_{\zero_i^-}(f) = \Hom_Q(P_{i^*(f)},S_i)^*$.

By Lemma \ref{L:iARt_path}.(2), we can identify $\I_{\zero_i^+}(f)$ with $\Hom_{Q}(P_+,P_i)^*$.
From the definition of $T_{\zero_i^+}$ and the exact sequence
$$ 0 \rightarrow \bigoplus_{i\to j} \Hom_{Q}(P_+,P_j) \rightarrow \Hom_{Q}(P_+,P_i) \rightarrow \Hom_{Q}(P_+,S_i)\to 0,$$
we conclude that $T_{\zero_i^+}(f) = \Hom_Q(P_+,S_i)^*$.
The description of maps follows from the naturality.

By Lemma \ref{L:iARt_path}.(3), we can identify $\I_{\Id_i}(f)$ with $\Hom_{Q}(P_-,P_i)^*$.
From the definition of $T_{\Id_i}$ and the exact sequence
$$ 0 \rightarrow \bigoplus_{i\to j} \Hom_{Q}(P_-,P_j) \rightarrow \Hom_{Q}(P_-,P_i) \rightarrow \Hom_{Q}(P_-,S_i)\to 0,$$
we conclude that $T_{\Id_i}(f) = \Hom_Q(P_-,S_i)^*$ and the description of maps follows from the naturality.
\end{proof}

\begin{definition} A vertex $v$ is called {\em maximal} in a representation $M$
if $\dim M(v)=1$ and all strict subrepresentations of $M$ are not supported on $v$.
\end{definition}
\noindent We note that every $T_v$ contains a maximal vertex $w$, which is another frozen vertex.
For example, $w=\zero_{i^*}^+$ (resp. $\Id_i$ and $\zero_i^-$) for $v=\zero_i^-$ (resp. $\zero_i^+$ and $\Id_i$).
For the rest of this article, when we write $\g(-)$, we view $\g$ as a linear functional via the usual dot product.

\begin{lemma} \cite[Lemma 6.5]{Fs1} \label{L:hom=0} Suppose that a representation $T$ contains a maximal vertex. Let $M=\Coker(\g)$, then $\Hom_J(M,T)=0$ if and only if $\g(\dv S)\geq 0$ for all subrepresentations $S$ of $T$.
\end{lemma}

\begin{lemma} \label{L:TIequi} Let $M=\Coker(\g)$, then $\Hom_J(M,T_v)=0$ for each frozen $v$ if and only if $\Hom_J(M,I_v)=0$ for each frozen $v$.
\end{lemma}

\begin{proof}  Since each subrepresentation of $T_v$ is also a subrepresentation of $I_v$, one direction is clear.
Conversely, let us assume that $\Hom_{J}(M,T_v)=0$ for each frozen vertex $v$.
We prove that $\Hom_{J}(M,I_v)=0$ by induction.

We first notice that for $v=\zero_i^-,\zero_i^+,\Id_i$, $T_v=I_v$ if $i$ is a sink.
In general, for an injective presentation of $T$: $0\to T\to I_1\to I_0$ with $\Hom_{J}(M,I_0)=0$,
we have that $\Hom_{J}(M,I_1)=0$ is equivalent to $\Hom_{J}(M,T)=0$.
Now we perform the induction from a sink $i$ in an appropriate order using \eqref{eq:injres1}--\eqref{eq:injres3}.
We can conclude that $\Hom_{J}(M,I_v)=0$ for all frozen $v$.
\end{proof}

\begin{definition} \label{D:GQ2} We define a cone $\mr{G}_{\Delta_Q^2} \subset \mb{R}^{(\Delta_Q^2)_0}$ by $\g(\dv S)\geq 0$ for all strict subrepresentations $S\subseteq T_v$ and all frozen $v$.
\end{definition}

\begin{theorem} \label{T:GQ2} The set of lattice points ${\sf G}_{\Delta_Q^2} \cap \mb{Z}^{(\Delta_Q^2)_0}$ is exactly $G(\Delta_Q^2,W_Q^2)$.
\end{theorem}

\begin{proof} Due to Lemma \ref{L:hom=0} and \ref{L:TIequi}, it suffices to show that ${\sf G}_{\Delta_Q^2}$ is defined by $\g(\dv S)\geq 0$ for all subrepresentations $S$ of $T_v$ and all $v$ frozen.
We notice that these conditions are the union of the defining conditions of
${\sf G}_{\Delta_Q^2}$ and $\g(\dv T_v)\geq 0$.
But the latter conditions are redundant
because $\dv T_v = \e_w + (\dv T_v-\e_w)$ and $\dv T_v-\e_w$ is the dimension vector of a strict subrepresentation of $T_v$,
where $w$ is the maximal (frozen) vertex of $T_v$.
\end{proof}

\begin{example}[Example \ref{ex:iARtD4} continued] \label{ex:iARtD4_ct}
In this case, it is almost trivial to list all strict subrepresentations for all $T_v$.
Readers can easily find that there are 3 subrepresentations for all $\zero_i^-$, and $7,6,1,1$ subrepresentations for $\zero_1^+,\zero_2^+,\zero_3^+,\zero_4^+$ respectively, and $1,2,7,7$ subrepresentations for $\Id_1,\Id_2,\Id_3,\Id_4$ respectively.
All these subrepresentations are needed to define ${\sf G}_{\Delta_Q^2}$, so there are 44 inequalities. 
Readers can find an extended version of this example in \cite{Ficra}.
This example can be easily generalized to $Q$ of type $D_n$ with a similar orientation.
\end{example}

\begin{remark}[The Valued Cases]
To deal with the general valued quiver $Q$, we could have worked with species analogue of QP, but this require some lengthy preparation.
To avoid this we can define the analogous Jacobian algebra by the Ext-completion algebra $\prod_{i\geq 0}\Ext_A^2(A^*,A)^{\otimes_A i}$ for the Auslander algebra $A:=A_Q^2$.

We can even define ${\sf G}_{\Delta_Q^2}$ without introducing the Jacobian algebra.
Without the Jacobian algebra we are unable to define the module $T_v$ by injective presentations, but it still makes perfect sense to define $T_v$ via Theorem \ref{T:Tv}. Once $T_v$'s are defined, we define the cones ${\sf G}_{\Delta_Q^2}$ by Definition \ref{D:GQ2}. The same remark also applies for the cone ${\sf G}_{\Delta_Q}$ to be defined below.

In general, some of the defining conditions of $\mr{G}_{\Delta_Q^2}$ may be redundant as shown in the following example.

\begin{example} The most complicated $T_v$ for $Q$ of type $G_2$ (see Example \ref{ex:G2} and Appendix \ref{A:list}) is $T_{\Id_1}$.
Its dimension vector is
$\GtwodimIdone$, and it strict subrepresentations have $13$ distinct dimension vectors.
However, we only need $5$ of them to define ${\sf G}_{\Delta_Q^2}$.
Readers can find a full list of inequalities in \cite{Ficra}.
In fact, we can download the full $H$-matrices of all exceptional types from author's web page \cite{Fweb}.
\end{example}
\end{remark}

Now we come back to the iARt quiver $\Delta_Q$.
We define a potential $W_Q$ on $\Delta_Q$ by the same formula defining $W_Q^2$.
The iARt QP $(\Delta_Q,W_Q)$ is nothing but the restriction of $(\Delta_Q^2,W_Q^2)$ to the subquiver $\Delta_Q$.
For each $\zero_i^-$, we define the representation $T_{\zero_i^-}$ of $J(\Delta_Q,W_Q)$ by the same injective presentation \eqref{eq:injres1}.
Similar to Theorem \ref{T:Tv}.(1), the module $T_{\zero_{i}^-}$ is the indecomposable module supported on all vertices translated from $\zero_{i^*}^+$ with dimension vector $(1,1,\dots,1)$.

We define a cone $\mr{G}_{\Delta_Q} \subset \mb{R}^{(\Delta_Q)_0}$ by $\g(\dv S)\geq 0$ for all subrepresentations $S\subseteq T_{\zero_i^-}$ and all $i$.
Note that we ask all subrepresentations not just strict ones.
We observe that the defining conditions of $\mr{G}_{\Delta_Q}$ and those of $\mr{G}_{\Delta_Q^2}$ are related as follows.
We group the defining conditions of $\mr{G}_{\Delta_Q^2}$ into three sets
$$\g H_u\geq 0,\ \g H_l\geq 0,\ \g H_r\geq 0.$$
They arise from the subrepresentations of $T_v$ for $v$ negative, neutral and positive respectively.
Then the defining conditions of $\mr{G}_{\Delta_Q}$ are exactly $\g H_u\geq 0$.

Similar to Theorem \ref{T:GQ2}, we have the following proposition. 
\begin{proposition} \label{P:GQ}
The set of lattice points ${\sf G}_{\Delta_Q} \cap \mb{Z}^{(\Delta_Q)_0}$ is exactly $G(\Delta_Q,W_Q)$.
\end{proposition}

\begin{definition} \label{D:polytope} Given a weight configuration $\bs{\sigma}$ of a quiver $\Delta$ and a convex polyhedral cone $\mr{G} \subset \mb{R}^{\Delta_0}$,
we define the (not necessarily bounded) convex polytope $\mr{G}(\sigma)$ as $\mr{G}$ cut out by the hyperplane sections $\g \bs{\sigma} = \sigma$.
\end{definition}

Our conjectural model for the tensor multiplicity is that the multiplicity $c_{\mu,\nu}^\lambda$ is counted by ${\sf G}_{\Delta_Q^2}(\mu,\nu,\lambda)$ for the weight configuration $\bs{\sigma}_Q^2$.
We will prove this model for $G$ of type ADE in Part \ref{P:II}.


\begin{appendices}
\section{A List of some iARt Quivers} \label{A:list}
For $Q$ of type $B,C$ and $D$, we only draw the iARt quiver $\Delta_Q^2$ for $B_3,C_3$ and $D_5$ ($D_4$ is our running example).
Reader should have no difficulty to draw the general ones.
The cases of $E_7$ and $E_8$ can be found in \cite{Ficra}.
We label the vertices of $Q$ in the same way as the software LiE \cite{LiE} so that you can check things conveniently.

\hspace{0.7in} $G_2$ \hspace{0.2in}
$\iARtGtwo$

\hspace{-0.2in} $B_3$ \hspace{0.1in}
$\iARtBthree$ \hspace{0.2in}
$C_3$ \hspace{0.1in}
$\iARtCthree$
\hspace{-0.0in} $D_5$ \hspace{0.1in}
$\iARtDfive$

\hspace{-0.1in} $F_4$ \hspace{0.1in}
$\iARtFfour$

\hspace{-0.2in}$E_6$ \hspace{-0.4in}
$\iARtEsix$

%
%
%
%
%
%
\end{appendices}
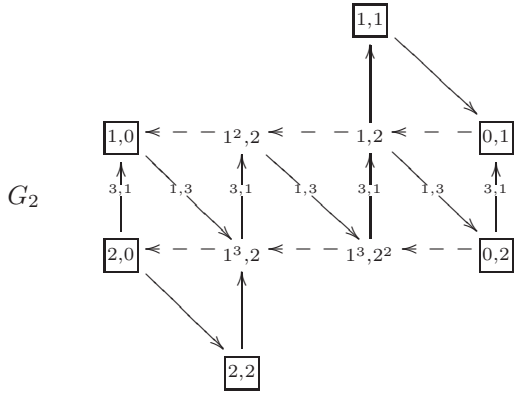
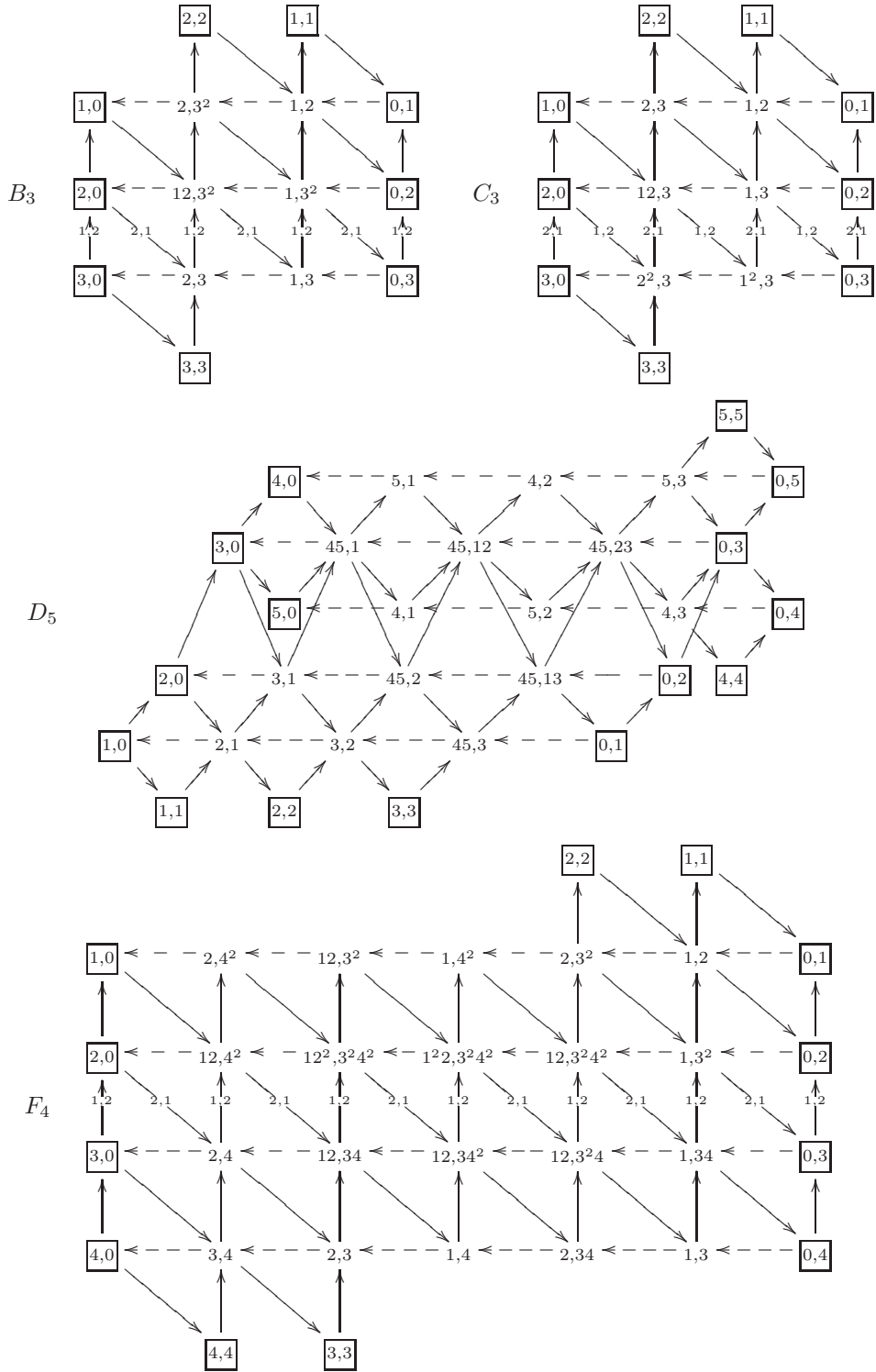
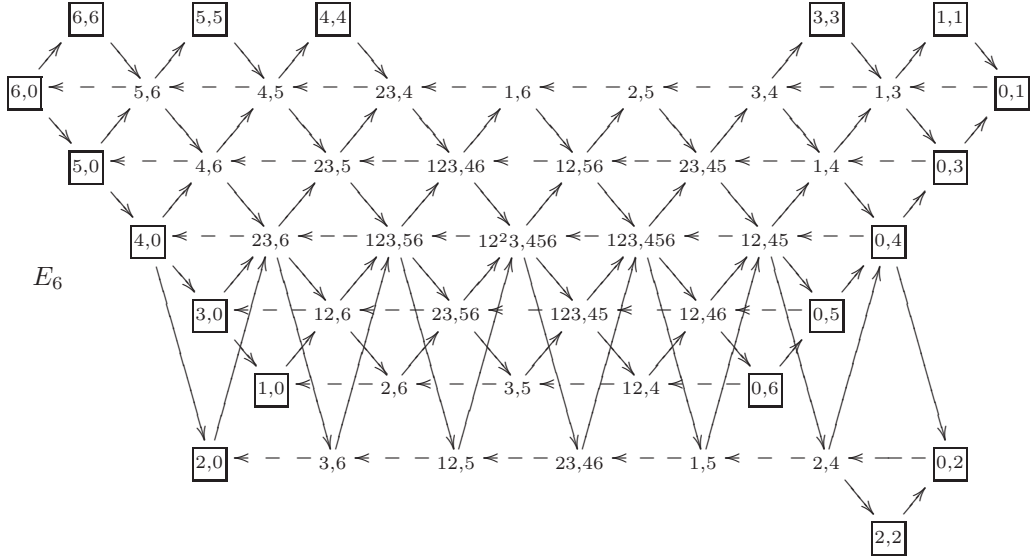
\vspace{.3in}

\part{Isomorphism to $k[\Conf_{2,1}]$} \label{P:II}
\section{Cluster Structure of Maximal Unipotent Groups} \label{S:Ugp}
\subsection{Basic Notation for a Simple Lie Group}
Let $k$ be an algebraically closed field of characteristic zero.
From now on, $G$ will always be a simply connected linear algebraic group over $k$ with Lie algebra $\mf{g}$.
We assume that the Dynkin diagram of $\mf{g}$ is the underlying valued graph of $Q$.
The Lie algebra $\mf{g}$ has the Cartan decomposition $\mf{g}=\mf{n}^-\oplus \mf{h}\oplus \mf{n}$.
Let $e_i, \alpha_i^\vee$, and $f_i$ ($i\in Q_0$) be the Chevalley generators of $\mf{g}$.
The \emph{simple coroots} $\alpha_i^\vee$ of $\mf{g}$ form a basis
of a Cartan subalgebra $\mf{h}$.
The \emph{simple roots} $\alpha_i\ (i\in Q_0)$ form a basis in the dual space
$\mf{h}^*$ such that
$[h, e_i] = \alpha_i (h) e_i$, and $[h,f_i] = - \alpha_i (h) f_i$
for any $h \in \mf{h}$ and $i\in Q_0$.
The structure of $\mf{g}$ is uniquely determined by the \emph{Cartan matrix} $C(G) = (c_{i,j}')$
given by $c_{i,j}' = \alpha_j (\alpha_i^\vee)$.
We have that $c'_{i,i}=2$ and $c_{i,j}'=-c_{j,i}$ so we have that $C(G)=E_l(Q)+E_r(Q)^T$.

Let $U^-$, $H$ and $U:=U^+$ be closed subgroups of $G$ with Lie algebras $\mf{n}^-$, $\mf{h}$ and $\mf{n}$.
Thus $H$ is a \emph{maximal torus}, and $U^-$ and $U$ are two opposite \emph{maximal unipotent subgroups} of~$G$.
Let $U^{\pm}_{i}~ (i\in Q_0)$ be the simple root subgroup of $U^{\pm}$.
By abuse of notation, we let $\alpha_i^{\vee}: \mathbb{G}_m \to H$ be the simple coroot corresponding to the root $\alpha_i: H\to \mathbb{G}_m$.
For all $i\in Q_0$, there are isomorphisms
$x_i: \mathbb{G}_{a}\to U_{i}^{+}$ and $y_{i}: \mathbb{G}_a\to U_i^{-}$ such that the maps
$$\begin{pmatrix}
      1 & a \\
      0 & 1 \\
   \end{pmatrix}
    \mapsto x_i(a), \quad
      \begin{pmatrix}
         1 & 0 \\
         b & 1 \\
      \end{pmatrix}
    \mapsto y_{i}(b),   \quad
       \begin{pmatrix}
          t & 0 \\
          0 & t^{-1} \\
       \end{pmatrix}
    \mapsto \alpha_i^{\vee}(t)$$
provide homomorphisms $\phi_i: {\SL}_2 \to G.$
We denote by $g\mapsto g^T$ the transpose anti-automorphism of $G$ defined by
$$x_i(a)^T=y_{i}(a),\ y_{i}(b)^T=x_i(b),\ h^T=h,\ (i\in Q_0, h\in H).$$

Let $s_i ~(i\in Q_0)$ be the simple reflections generating the Weyl group of $G$. Set
$$\br{s_i}:=x_{i}(-1)y_{i}(1)x_{i}(-1).$$ 
The elements $\br{s_i}$ satisfy the braid relations. So we can associate to each $w\in W$ its representative $\overline{w}$ in such a way that for any reduced decomposition $w=s_{i_1}\cdots s_{i_\ell}$ one has $\br{w}=\br{s_{i_1}}\cdots \br{s_{i_\ell}}$.
Denote by $w_0$ be the longest element of the Weyl group.
In general $s_G:=\br{w_0}^2$ is not the identity but an order two central element in $G$.
It is well-known that $w_0(\alpha_i)=-\alpha_{i^*}$, where $i\mapsto i^*$ is the same involution in Section \ref{ss:iARt2}.

The \emph{weight lattice} $P(G)$ of $G$ consists of all $\gamma \in \mf{h}^*$ such that
$\gamma (\alpha_i^\vee) \in \mb{Z}$ for all $i$.
Thus $P(G)$ has a $\mb{Z}$-basis $\{\varpi_i\}_{i\in Q_0}$ of \emph{fundamental weights}
given by $\varpi_j (\alpha_i^\vee) = \delta_{i,j}$.
We can thus identify a weight by an integral vector $\lambda\in \mb{Z}^{Q_0}$. We write $\varpi(\lambda)$ for $\sum_i \lambda(i) \varpi_i$.
In this notation, $\varpi_i=\varpi(\e_i)$.
We stress that throughout the paper, this identification is widely used.
A weight $\lambda\in \mb{Z}^{Q_0}$ is \emph{dominant} if it is non-negative.

\subsection{The Base Affine Space} \label{ss:baseaff}
The natural $G\times G$-action on $G$:
$$(g_1,g_2)\cdot g = g_1gg_2^{-1},\ $$
induces the left and right translation of $G$ on $k[G]$:
$$(g_1,g_2)\varphi(g)=\varphi(g_1^{-1}gg_2)\ \text{ for $\varphi\in k[G]$}.$$
The algebraic Peter-Weyl theorem \cite[Theorem 4.2.7]{GW} says that as a $G\times G$-module $k[G]$ decomposes as
$$k[G]= \bigoplus_{\lambda\in \mb{Z}_{\geqslant 0}^{Q_0}} L(\lambda)^\vee \otimes L(\lambda),$$
where $L(\lambda)$ is the irreducible $G$-module of highest-weight $\varpi(\lambda)$, and $L(\lambda)^\vee$ is its dual.
Quotienting out the left translation of $U^-$, we get a $G$-module decomposition of the ring of regular functions on the {\em base affine space} $\mc{A}:=U^-\bcsl G$:
$$k[\mc{A}]=\left\{\varphi\in k[G]\mid \varphi(u^-g)=\varphi(g) \text{ for $u^-\in U^-,\ g\in G$} \right\}=\bigoplus_{\lambda\in \mb{Z}_{\geqslant 0}^{Q_0}} L(\lambda).$$
Each $G$-module $L(\lambda)$ can be realized as the subspace of
$k[\mc{A}]$:
\begin{equation} \label{eq:irrG} L(\lambda) \cong \left\{\varphi\in k[\mc{A}]\mid \varphi(hg)=h^{\varpi(\lambda)} \varphi(g)
\text{ for } h\in H,\ g\in G\right\}.
\end{equation}
Similarly for the {\em dual} base affine space $\mc{A}^\vee:=G/U$,
we have that
$$k[\mc{A}^\vee]=\left\{\varphi\in k[G]\mid \varphi(gu)=\varphi(g) \text{ for $u\in U,\ g\in G$} \right\}=\bigoplus_{\lambda\in \mb{Z}_{\geqslant 0}^{Q_0}} L(\lambda)^\vee,$$
\begin{equation}\label{eq:irrGdual} L(\lambda)^\vee \cong \left\{\varphi\in k[\mc{A}^\vee]\mid \varphi(gh)=h^{\varpi(\lambda)} \varphi(g)
\text{ for } h\in H,\ g\in G \right\}.\end{equation}
Keep in mind that the $G$-actions on $k[U^-\bcsl G]$ and $k[G/U]$ are via the right and left translations respectively.

We fix an additive character $U^-\to \mb{G}_a$, then a coset $U^-g$ (resp. $gU$) determines a point in $\mc{A}$ (resp. $\mc{A}^\vee$).
We refer readers to \cite[1.1.1]{GS} for the detail.

For each fundamental representation $L(\e_i)$ and its dual $L(\e_i)^\vee$, we choose a $U^-$-fixed vector $u_i$ and a $U$-fixed vector $u_i^\vee$,
normalized by $\innerprod{u_i^\vee\mid u_i}=1$.
Following~\cite{FZ} we define for each fundamental weight $\varpi_i$ and each pair $w,w'\in W$
the {\em generalized minor} $m_{w,w'}^{\varpi_i}\in k[G]$
\begin{equation} \label{eq:genminor} g\mapsto \innerprod{\br{w}u_{i}^\vee\mid g\br{w'}u_{i}}.\end{equation}
More concretely, the Bruhat decomposition $G=\coprod_{w\in W} U^-HwU$ implies that any $g\in G$ can be written as $g=u^- h \br{w} u$.
Suppose that $g$ lies in the open set $G_0:=U^-HU$, then
the regular function $m^{\varpi_i}:=m_{e,e}^{\varpi_i}$ restricted to $G_0$ is given by
$m^{\varpi_i}(g)=h^{\varpi_i}$
(For the type $A_r$, this is a principal minor of the matrix $g$).
Then the definition of $m^{\varpi_i}$ can be extended to whole $G$ as in \cite[Proposition 2.4]{FZ}.
\begin{proposition} \label{P:pminor} For $g=u^- h \br{w} u \in G$, we have that
$$m^{\varpi_i}(g)=h^{\varpi_i} \text{ if } w(\varpi_i)=\varpi_i\ \text{ and }\ m^{\varpi_i}(g)=0 \text{ otherwise.}$$
\end{proposition}

\noindent The function $m_{w,w'}^{\varpi_i}$ is then given by
\begin{equation} \label{eq:genminor0} m_{w,w'}^{\varpi_i}(g)=m^{\varpi_i}(\br{w}^{-1} g \br{w'}). \end{equation}
It follows from \eqref{eq:genminor0} or \eqref{eq:genminor} that

\begin{lemma} \label{L:wtminor} The weight for the $H\times H$-action on $m_{w,w'}^{\varpi^i}$ is $(w(\varpi_i),-w'(\varpi_i))$, i.e.,
$$m_{w,w'}^{\varpi_i}(h_1gh_2^{-1})=h_1^{w(\varpi_i)}h_2^{-w'(\varpi_i)}m_{w,w'}^{\varpi_i}(g).$$
\end{lemma}


\subsection{Cluster Structure on $k[U]$} \label{ss:clusterU}
Now we are ready to recall the cluster algebra structure of $k[U]$.
For this, we associate to each non-neutral indecomposable presentation $f$, a generalized minor $m_f$ as follows.
Suppose that $\e(f)=\e_i$, i.e., $f$ is translated from $\zero_i^+$,
and moreover $w(\varpi_i)=\varpi(\f)$, then we put $m_f:=m_{e,w}^{\varpi_i}$.
Note that if $f$ is positive, then $f=m_{e,e}^{\varpi_i}$ is a principal minor.
We take the iARt quiver $\Delta_Q$, and let $\mc{M}_Q:=\{m_{f}\}_{f\in (\Delta_Q)_0}$.
It is known \cite{BFZ,GLSa} that $k[U]$ is the upper cluster algebra with this {\em standard} seed $(\Delta_Q, \mc{M}_Q)$.
Here we again identify $\Delta_Q$ as the subquiver of $\Delta_Q^2$.
Moreover, the cluster algebra is equal to its upper cluster algebra when $Q$ is simply-laced \cite{GLSk}.

Recall the weight configuration $\bs{\sigma}_Q$ in Definition \ref{D:triplewt}.
For each vertex $f\in (\Delta_Q)_0$, we set $\bs{\sigma}_Q(f)=\e(f)-\f$.
By Lemma \ref{L:wtminor} the degree of $m_f$ for the conjugation action of $H$ is exactly $\varpi(\bs{\sigma}_Q(f))$.
Let $\bs{\rho}_Q$ be the matrix with rows also indexed by $(\Delta_Q)_0$ such that
$\bs{\rho}_Q(f)$ is the positive root of $G$ corresponding to $f$, i.e., $\bs{\rho}_Q(f)=\sum_i \dim(M(i))\alpha(i)$ where $M=\Coker(f)$.
Note that by the {\em generalized Gabriel's theorem} \cite{DR}, $\bs{\rho}_Q$ contains exactly all positive roots of $G$.
The {\em Kostant's partition function} $p_Q(\gamma)$ by definition counts the lattice points in the polytope
$${\sf K}(\gamma):=\{\h\geq 0 \mid h\bs{\rho}_Q=\gamma\}.$$

Now we consider a labelling dual to the one in Section \ref{ss:iARt2}.
Let $(i,t)=(i(f),t(f))$ be the number such that $f=\tau^{t}(\zero_i^+)$.
Note that $i(f)^* =i^*(f)$ and $t(f)+t^*(f)=t\big(\zero_{i^*(f)}^-\big)=t^*\big(\zero_{i(f)}^+\big)$.

\begin{lemma} \label{L:sigma2rho} $\varpi\left(\bs{\sigma}_Q(f)\right) = \sum_{t=0}^{t(f)-1} \bs{\rho}_Q(\tau^{-t} f)$.
\end{lemma}

\begin{proof} Let $M=\Coker(f)$. The equality is equivalent to that
$$\e(f)-\f = \bigg(\sum_{t=0}^{t(f)-1} \dv(\tau^{-t} M) \bigg)C(G).$$
Recall that the Cartan matrix $C(G)$ is equal to $E_l(Q)+E_r(Q)^T$.
It follows from \eqref{eq:resimple} that matrix $E_l(Q)$ transforms $-\dv M$ to the reduced weight vector of $f(M)$, and
$E_r(Q)^T$ transforms $\dv M$ to the reduced weight vector $\f^\vee(M)$ of the minimal injective presentation of $M$, or equivalently, $\f^\vee(M)=-\f(\tau^{-1} M)$.
So the righthand side is equal to
$$\sum_{t=0}^{t(f)-1} -\f(\tau^{-t} M) - \f^\vee(\tau^{-t} M).$$
After some cancellations, only two terms survive.
They are $\e(f)$ and $-\f=-\f(M)$.
\end{proof}

For the weight configuration $\varpi(\bs{\sigma}_Q):=\{\varpi(\bs{\sigma}_Q(f))\}_{f\in (\Delta_Q)_0}$ and the polyhedral cone ${\sf G}_{\Delta_Q}$, we consider the polytope ${\sf G}_{\Delta_Q}(\gamma)$ as in Definition \ref{D:polytope}.
\begin{lemma} \label{L:parfun} The function $|{\sf G}_{\Delta_Q}(-)\cap \mb{Z}^{(\Delta_Q)_0}|$ is the partition function $p_Q(-)$.
\end{lemma}

\begin{proof} Recall that ${\sf G}_{\Delta_Q}(\gamma)$ is defined by
$\sum_{t=0}^{t^*(f)} \g(\tau^t f) \geq 0$ for $f\in (\Delta_Q)_0$ and $\g\varpi(\bs{\sigma}_Q) = \gamma$.
We introduce new variables $\h(f)$ for each $f$ satisfying that
$\h(f) = \sum_{t=0}^{t^*(f)} \g(\tau^t f)$. Then
$$\h\bs{\rho}_Q=\sum_f\sum_{t=0}^{t^*(f)} \g(\tau^tf)\bs{\rho}_Q(f)=\sum_f\sum_{t=0}^{t(f)-1} \g(f)\bs{\rho}_Q(\tau^{-t}f)=\g\varpi(\bs{\sigma}_Q).$$
The second equality is established through an easy bijection and the last one is due to Lemma \ref{L:sigma2rho}.
So the defining condition for $\mr{G}_{\Delta_Q}(\gamma)$ is equivalent to that for the polytope ${\sf K}(\gamma)$.
Finally, we observe that the transformation from $\g$ to $\h$ is totally unimodular. In particular, the transformation and its inverse preserve lattice points.
\end{proof}

\begin{theorem} \label{T:Ucluster} The coordinate ring of $U$ is equal to the graded cluster algebra $\mc{C}\left(\Delta_Q,\mc{M}_Q;\varpi(\bs{\sigma}_Q)\right)$.
Moreover, if $Q$ is trivially valued, then $(\Delta_Q,W_Q)$ is a cluster model (Definition \ref{D:model}).
\end{theorem}

\begin{proof} We only need to prove the second statement.
Recall that the universal enveloping algebra $U(\mf{n})$ has a standard grading by $\deg(e_i)=\alpha_i$.
It is a classical fact \cite{K} that the partition function $p_Q(\gamma)$ counts the dimension of the homogeneous component $U(\mf{n})_\gamma$.
The algebra $k[U]$ is graded dual to $U(\mf{n})$ with (see \cite{Z})
$$k[U]_\gamma = \left\{\varphi\in k[U]\mid \varphi(huh^{-1})=h^{\gamma}\varphi(u) \text{ for $h\in H,\ u\in U$} \right\}.$$
We have seen that the degree of $\mc{M}_Q(f)=m_f$ is $\varpi(\bs{\sigma}_Q(f))$.
So the second statement follows from Theorem \ref{T:GCC}, Lemma \ref{L:parfun}, and the fact that $\mc{C}(\Delta_Q)=\uca(\Delta_Q)$.
\end{proof}

There are another two seeds of this cluster structure of $k[U]$.
One is called {\em left standard}, and the other is called {\em right standard}.
Both are obtained from the standard seed by a sequence of mutations.
The sequences of mutations $\bs{\mu}_l$ and $\bs{\mu}_r$ will be defined and studied in Appendix \ref{A:cyclic}.
Let
$$\big(\wtd{\Delta}_Q^\#,\wtd{W}_Q^\#,\mc{M}_Q^\#;\bs{\sigma}_Q^\#\big)=\bs{\mu}_\#(\Delta_Q,W_Q,\mc{M}_Q;\bs{\sigma}_Q)\ \text{ for $\#=l,r$.}$$
For what follows in this Section, $\#$ always represents $l$ and $r$.
Let $\Delta_Q^l$ (resp. $\Delta_Q^r$) be the ice quiver obtained from $\Delta_Q^2$ by deleting the negative and positive (resp. negative and neutral) frozen vertices.
By Corollary \ref{C:mu_l}, $\wtd{\Delta}_Q^\#$ and $\Delta_Q^\#$ are equal up to arrows between frozen vertices.
Since the QP $(\Delta_Q,W_Q)$ is rigid, so is $(\wtd{\Delta}_Q^\#,\wtd{W}_Q^\#)$.
Let $W_Q^\#$ be the potential $W_Q^2$ restricted to $\Delta_Q^\#$.
By \cite[Proposition 8.9]{DWZ1} $(\Delta_Q^\#,W_Q^\#)$ is also rigid.
In particular, both $(\wtd{\Delta}_Q^\#,\wtd{W}_Q^\#)$ and $(\Delta_Q^\#,W_Q^\#)$ are $\mu$-rigid.
By the equivalent definition of rigidity (\cite[Definition 6.10, Remark 6.8]{DWZ1}), $(\wtd{\Delta}_Q^\#,\wtd{W}_Q^\#)$ is $\mu$-right-equivalent to $(\Delta_Q^\#,W_Q^\#)$.

\cite[Proposition 5.15]{Fs1} implies that being a cluster model is mutation-invariant.
It follows that $(\wtd{\Delta}_Q^\#,\wtd{W}_Q^\#)$ is also a cluster model. In view of Remark \ref{r:cluster_model}, we can replace $(\wtd{\Delta}_Q^\#,\wtd{W}_Q^\#)$ by $(\Delta_Q^\#,W_Q^\#)$.
\begin{proposition} \label{P:Ulr} $(\Delta_Q^l,\mc{M}_Q^l;\bs{\sigma}_Q^l)$ and $(\Delta_Q^r,\mc{M}_Q^r;\bs{\sigma}_Q^r)$ are another two graded seeds for the cluster algebra $k[U]$.
Moreover, if $Q$ is trivially valued, both $(\Delta_Q^l,W_Q^l)$ and $(\Delta_Q^r,W_Q^r)$ are cluster models.
\end{proposition}

\begin{remark} \label{r:GQlr} Later we will also need a concrete description of $G(\Delta_Q^l,W_Q^l)$ and $G(\Delta_Q^r,W_Q^r)$.
Both are given by lattice points in rational polyhedral cones.
Recall that the three sets of defining conditions of the polyhedral cone ${\sf G}_{\Delta_Q^2}$
$$\g H_u\geq 0,\ \g H_l\geq 0,\ \g H_r\geq 0.$$
We define the polyhedral cones ${\sf G}_{\Delta_Q^l}\subset \mb{R}^{(\Delta_Q^l)_0}$ and ${\sf G}_{\Delta_Q^r}\subset \mb{R}^{(\Delta_Q^l)_0}$ by the relations
$\g H_l\geq 0$ and $\g H_r\geq 0$ respectively.
Almost the same proof as Theorem \ref{T:GQ2} can show that
$G(\Delta_Q^l,W_Q^l)={\sf G}_{\Delta_Q^l}\cap\mb{Z}^{(\Delta_Q^l)_0}$ and $G(\Delta_Q^r,W_Q^r)={\sf G}_{\Delta_Q^r}\cap\mb{Z}^{(\Delta_Q^r)_0}$.
\end{remark}

\section{Maps Relating Unipotent Groups} \label{S:toU}
\subsection{Standard Maps} \label{ss:standard}
Recall the base affine space $\mc{A}$ and its dual $\mc{A}^\vee$ defined in Section \ref{ss:baseaff}.
Let $\Conf_{n,1}:= (\mc{A}^{n} \times \mc{A}^\vee)/G$ be the categorical quotient in the category of varieties \footnote{For our purpose, it is enough to work with the categorical quotient, which has the same ring of regular functions on the corresponding quotient stack.}.
\begin{lemma} \label{L:UFD} The ring of regular functions on $\Conf_{n,1}$ is a unique factorization domain. 
\end{lemma}

\begin{proof} 
It is well-known \cite{KKV} that $k[G]$ is a UFD.
Since the group $U$ and $G$ have no multiplicative characters,
by \cite[Theorem 3.17]{PV} $k[\Conf_{n,1}]$ is also a UFD.
\end{proof}

By the Bruhat decomposition, any pair $(A_1,A_0^\vee)\in\Conf_{1,1}$ has a representative $(U^-h\br{w},U)$ for some $h\in H, w\in W$.
A pair $(A_1,A_0^\vee)\in\Conf_{1,1}$ is called {\em generic} if the $w\in W$ can be chosen as the identity.
Let $\Conf_{n,1}^\circ$ be the open subset of $\Conf_{n,1}$ where each pair $(A_i,A_{0}^\vee)$ is generic (but we do not impose any condition among $A_i$'s).
By definition, we have an isomorphism $H\cong \Conf_{1,1}^\circ$.
Let $\iota:H\hookrightarrow \Conf_{1,1}$ be the open embedding $h\mapsto (U^-h,U)$.
For each $i\in Q_0$ we define the regular function $\wtd{\varpi}_i$ on $\Conf_{1,1}$ by
$\wtd{\varpi}_i(U^-g_1,g_0U) = m^{\varpi_i}(g_1g_0)$.
It is clear that the definition does not depend on the representatives.
It follows from Proposition \ref{P:pminor} that

\begin{lemma} $k[H]$ is exactly the localization of $k[\Conf_{1,1}]$ at all $\wtd{\varpi}_i$.
\end{lemma}

From now on, we will focus on the space $\Conf_{2,1}$.
Its ring of regular functions has a triple-weight decomposition
$$k[\Conf_{2,1}]=\bigoplus_{\tripar\in Z_{\geqslant 0}^{Q_0}} \left(L(\mu)\otimes L(\nu)\otimes L(\lambda)^\vee\right)^{G}.$$
We write very often $C_{\mu,\nu}^\lambda$ for the graded component $k[\Conf_{2,1}]_\tripar$.
It is clear that $\lrcoef=\dim C_{\mu,\nu}^\lambda$.

We recall several rational maps defined in \cite{FG,GS}.
Let $i: H\times H\times U \hookrightarrow \Conf_{2,1}$ be the open embedding
$$(h_1,h_2,u)\mapsto(U^-h_1,U^-h_2u,U).$$
It is an embedding because the stabilizer of the generic pair $(U^-h,U)$ is $1_G$.
It is clear that the image of $i$ is exactly $\Conf_{2,1}^\circ$.
By restriction we get an embedding $i^*: k[\Conf_{2,1}]\hookrightarrow k[H\times H\times U]$.
We will view two $H$'s and $U$ as subgroups of $H\times H\times U$ in the natural way,
and write $i_1,i_2$ and $i_u$ for the restriction of $i$ on first $H$, second $H$ and $U$ respectively.
A function $s$ in $C_{\mu,\nu}^\lambda$ is uniquely determined by its restriction $i_u^*(s)$ on $U$ because
\begin{equation} \label{eq:resU} s(h_1,h_2,u) = h_1^{\varpi(\mu)} h_2^{\varpi(\nu)} s(u)\ \text{ for } (h_1,h_2,u)\in H\times H\times U.
\end{equation}
So $i^*$ embeds $C_{\mu,\nu}^\lambda$ into $k[U]$ (in fact into $k[U]_{\varpi(\mu+\nu-\lambda)}$ by easy calculation).
We note that each $C_{\mu,\nu}^\lambda$ is {\em not} disjoint under this embedding.

We recall a classical interpretation of the multiplicity $\lrcoef$. 
Let $L(\mu)_\gamma$ be the weight-$\gamma$ subspace of the irreducible $G$-module $L(\mu)$.
We denote
$$L(\mu)_{\gamma}^{\nu}:=\left\{\varphi \in L(\mu)_{\gamma}\mid e_i^{\nu(i)+1}(\varphi)=0   \text{ for $i\in Q_0$} \right\}.$$

\begin{lemma}\cite{PRV}  We have that $\lrcoef = \dim L(\mu)_{\varpi(\lambda-\nu)}^\nu$.
\end{lemma}

\begin{lemma} \label{L:c=1} For any $f\in\ind(C^2Q)$, the generalized minor $m_f$ spans the space
	$L\left(\e(f)\right)_{\varpi(\f)}^{\f_-} \subset k[\mc{A}]$.
	In particular, $c_{\e(f),\f_-}^{\f_+}=1$.
\end{lemma}

\begin{proof} Recall that the Chevalley generator $e_i$ acts on $k[G]$ by 
	$$e_i\varphi(g)=\frac{d}{dt}\Big|_{t=0} \varphi(gx_i(t)).$$	
	For any $\varphi\in k[G]$, the coefficient of each $t^n$ in $\varphi(x_i(t))$ is equal to $e_i^n\varphi(1_G)/n!$.
	If $\varphi$ is bihomogeneous of degree $(\gamma,\gamma')$, then $\varphi(1_G)$ can be nonzero only if $\gamma=\gamma'$. 
	It follows that $\varphi(x_i(t))$ contains only $t^n$ with $n\alpha_i=\gamma-\gamma'$.
	So for $(\gamma,\gamma')=\left(\varpi(\e(f)),\varpi(\f)\right)$
	$$e_i^{\f_-(i)+1}\varphi=0\ \text{ if and only if }\ (\f_-(i)+1)c_i\neq \e(f)-\f,$$
	where $c_i$ is the $i$-th column of the Cartan matrix $C(G)$.
	Since $\e(f)(i)\leq 1$ and $c_i(i)=2$,
	we can never have $(\e(f)+\f_--\f_+)(i)=(\f_-(i)+1)c_i(i)$.
	Hence we proved that $e_i^{\f_-(i)+1}m_f=0$ for $i\in Q_0$ so that $m_f\in L\left(\e(f)\right)_{\varpi(\f)}^{\f_-}$.
	
	Suppose that $m_f=m_{e,w}^{\varpi_j}$. 
	The generalized minor $m_f$ spans the space because 
	$$\dim L(\e_j)_{w(\varpi_j)} = \dim L(\e_j)_{\varpi_j} =1.$$
\end{proof}

%

%
%
%

\noindent We recall that each $m_{\zero_i^+}$ is a principal minor, which evaluates to $1$ on $U$;
and now we set $m_{\Id_i}=1$.

\begin{lemmadef}\label{L:i_u} For any $f\in\ind(C^2Q)$, there is a unique function $s_f\in k[\Conf_{2,1}]_\triwtf$ such that
$$i_u^*(s_f)=m_f.$$
Moreover, $s\in C_{\mu,\nu}^\lambda$ satisfies $i_u^*(s)=1$ if and only if $s=s^{\mu}_+ s^{\nu}_0$, where $s^{\mu}_+:= \prod_i s_{\zero_i^+}^{\mu(i)}$ and $s^{\nu}_0:= \prod_i s_{\Id_i}^{\nu(i)}$.
\end{lemmadef}

\begin{proof} The first statement follows from Lemma \ref{L:c=1}.
For the last statement, it suffices to prove the ``only if" part. If $i_u^*(s)=1$, then $\mu+\nu=\lambda$ because $i_u^*(s)$ has degree $\varpi(\mu+\nu-\lambda)$.
It is clear that $c_{\mu,\nu}^{\mu+\nu}=1$ (consider tensoring the highest weight vectors of $L(\mu)$ and $L(\nu)$), but $s^{\mu}_+ s^{\nu}_0$ also has degree $(\mu,\nu,\mu+\nu)$, so $s=s^{\mu}_+ s^{\nu}_0$.
\end{proof}

\begin{remark} For $Q$ of type $A$, the analogous map $i^*$ was also considered in a quiver-invariant theory setting (see \cite[Example 3.10, 3.12]{Fs1}). In that setting, the map comes from a semi-orthogonal decomposition of the module category of a triple flag quiver.
\end{remark}

\begin{corollary} \label{C:irreducible} For each $f\in \ind(C^2Q)$, $s_f$ is irreducible in $k[\Conf_{2,1}]$.
\end{corollary}

\begin{proof} It is clear that the zero-degree component of $k[\Conf_{2,1}]$ is $k$.
If $f$ is positive or neutral, then the degree of $f$ is obviously indecomposable.
The same argument as in \cite[Lemma 1.8]{Fs1} shows that $s_f$ is irreducible.

We remain to consider the case when $f$ is the minimal presentation of a general representation.
It is known that all generalized minors $m_f$ are irreducible in $k[U]$.
Suppose that $s_f$ factors as $s_f=s_1s_2$, then $i_u^*(s_f)=i_u^*(s_1)i_u^*(s_2)=m_f$.
So one of them, say $i_u^*(s_1)$, has to be a unit.
According to Lemma \ref{L:i_u}, $s_1$ is a polynomial in $s_{\zero_i^+}$'s and $s_{\Id_i}$'s.
So there is a homogenous component $s_2'$ of $s_2$ and some $\mu$ and $\nu$ such that
$s_2's_+^{\mu}s_0^{\nu}$ has the same degree as $s_f$, i.e., $s_2's_+^{\mu}s_0^{\nu}\in C_{\e(f),\f_-}^{\f_+}$.
Since $f$ is a minimal presentation of a general representation, we must have $\nu=0$.
If $\mu$ is nonzero, then we must have $\mu=\e(f)$.
Then the weight of $s_2'$ is $(0,\f_-,\f_+-\e(f))$, and $\f_-$ has to be equal to $\f_+-\e(f)$.
But $i_u^*$ embeds $C_{\e(f),\f_-}^{\f_+}$ into $k[U]_{\e(f)-\f}$. 
Since $f$ is not positive or neutral, $\e(f)\neq \f$, which is a contradiction.
\end{proof}

Let $p$ be a rational inverse of $i$, which is regular on $\Conf_{2,1}^\circ=i(H\times H\times U)$.
Let $p_1,p_2$ and $p_u$ be the composition of $p$ with the natural projection to first $H$, second $H$, and $U$ respectively.

\begin{corollary} \label{C:p_u} On the open subset $\Conf_{2,1}^\circ\subset \Conf_{2,1}$
\begin{align*} & p_u^* (m_f) = s_f \left(s^{{\e(f)}}_+ s^{{\f_-}}_0\right)^{-1},\\
&p_1^*(\varpi_i) = s_{\zero_i^+},\ p_2^*(\varpi_i) = s_{\Id_i}.
\end{align*}
for any indecomposable presentation $f$.
\end{corollary}

\begin{proof} They follow from straightforward calculation \begin{align*}
&(p_u^* m_f)\left(i(h_1,h_2,u)\right)=m_f\left(p_ui(h_1,h_2,u)\right)=m_f(u);\\
&s_f (i(h_1,h_2,u))= h_1^{\varpi(\e(f))} h_2^{\varpi(\f_-)} m_f(u),\ (s^{{\e(f)}}_+ s^{{\f_-}}_0)(i(h_1,h_2,u))=h_1^{\varpi(\e(f))} h_2^{\varpi(\f_-)}.
\end{align*}
The statement for $p_1^*,p_2^*$ is rather obvious. For example,
$$s_{\zero_i^+}(i(h_1,h_2,u))=h_1^{\varpi_i}h_2^0i_u^*(s_{\zero_i^+})(u)=h_1^{\varpi_i}=p_1^*(\varpi_i)\left(i(h_1,h_2,u)\right).$$
\end{proof}

\begin{corollary} \label{C:localization} The localization of $k[\Conf_{2,1}]$ at all $s_f$'s for $f$ positive and neutral is exactly $k[H\times H\times U]$.
\end{corollary}

\begin{proof} Let $\wtd{p}_1$ and $\wtd{p}_2$ be the natural projections $\Conf_{2,1}\to \Conf_{1,1}$ given by
$(A_1,A_2,A_0^\vee)\mapsto (A_1,A_0^\vee)$ and $(A_1,A_2,A_0^\vee)\mapsto (A_2,A_0^\vee)$.
Then $\Conf_{2,1}^\circ$ is the intersection $\wtd{p}_1^{-1}(\Conf_{1,1}^\circ)\cap \wtd{p}_2^{-1}(\Conf_{1,1}^\circ)$.
It suffices to show that $\wtd{p}_1^*(\wtd{\varpi}_i)=s_{\zero_i^+}$ and $\wtd{p}_2^*(\wtd{\varpi}_i)=s_{\Id_i}$.
The map $\iota p_i$ agrees with $\wtd{p}_i$ on a dense subset of $\Conf_{2,1}$ for $i=1,2$.
So $\wtd{p}_1^*(\wtd{\varpi}_i) = p_1^* \iota^*(\wtd{\varpi}_i) = p_1^*(\varpi_i) = s_{\zero_i^+}$ and
$\wtd{p}_2^*(\wtd{\varpi}_i) = p_2^* \iota^*(\wtd{\varpi}_i) = p_2^*(\varpi_i) = s_{\Id_i}$.
\end{proof}

\begin{corollary} \label{C:algind} The set of functions $\mc{S}_Q^2:=\{s_f\}_{f\in \ind(C^2Q)}$ is algebraically independent over the base field $k$.
\end{corollary}

\begin{proof} Since the map $p:\Conf_{2,1}\to H\times H\times U$ is birational,
the pullbacks of (two copies of) $\varpi_i\in H$ and $m_f\in U$ are algebraically independent.
We have seen that up to a factor of some monomial in $s_{\zero_i^+}, s_{\Id_i} \in \mc{S}_Q^2$, they are exactly the functions in $\mc{S}_Q^2$.
Our claim follows.
\end{proof}

Suppose that $f$ as above is mutable. Let $m_f'$ be a function in $k[U]$ in the exchange relation
\begin{equation} \label{eq:exrelmf} m_f m_{f}'=\prod_{g\rightarrowtail f} m_{g}^{c_{g,f}} + \prod_{f\rightarrowtail h} m_{h}^{c_{f,h}}.\end{equation}
Let $(\mu',\nu',\lambda')=\sum_{f\rightarrowtail h} c_{f,h}(\e(h),\h_-,\h_+) - (\triwtf)$.
\begin{lemma} \label{L:exchange}  The function $s_{f}':=p_u^*(m_{f}') s^{\mu'}_+ s^{\nu'}_0$ is regular on $\Conf_{2,1}$ satisfying $i_u^*(s_f')=m_f'$ and the exchange relation
$$s_fs_f' = \prod_{g\rightarrowtail f} s_{g}^{c_{g,f}}  + \prod_{f\rightarrowtail h} s_{h}^{c_{f,h}} .$$
Moreover, $s_f'$ and $s_f$ are relatively prime in $k[\Conf_{2,1}]$.
\end{lemma}

\begin{proof} Since $i$ is a regular section of $p$, we have that $i_u^*(s_f')=m_f'$ and $s_f'$ is regular on $\Conf_{2,1}^\circ=i(H\times H\times U)$.
We pull back \eqref{eq:exrelmf} through $p_u$, and multiply $s^{\mu'}_+ s^{\nu'}_0$ on both sides.
Then we get from Corollary \ref{C:p_u} that
\begin{equation} \label{eq:pre_exrel} s_f\left(p_u^*(m_f')s^{\mu'}_+ s^{\nu'}_0\right) = \prod_{g\rightarrowtail f} s_{g}^{c_{g,f}} + \prod_{f\rightarrowtail h} s_{h}^{c_{f,h}}.\end{equation}
We remain to show that $s_f'$ is in fact regular on whole $\Conf_{2,1}$.
By \eqref{eq:pre_exrel}, the locus of indeterminacy of $s_f'$ is contained in the zero locus $Z(s_f)$ of $s_f$.
Since $s_f\mid_U=m_f$, the intersection of $Z(s_f)$ and $\Conf_{2,1}^\circ$ is nonempty.
In other words, $Z(s_f)$ is not contained in the complement of $\Conf_{2,1}^\circ$.
Since $Z(s_f)$ is irreducible (Corollary \ref{C:irreducible}), we conclude that $s_f'$ is regular outside a codimension 2 subvariety of $\Conf_{2,1}$.
Since $\Conf_{2,1}$ is factorial (Lemma \ref{L:UFD}) and thus normal, by the algebraic Hartogs $s_f'$ is a regular function on $\Conf_{2,1}$.

For the last statement, we suppose the contrary. Since $s_f$ is irreducible, $s_f$ is then a factor of $s_f'$.
But this is clearly impossible by comparing the first component of their triple-weight.
\end{proof}

\subsection{Twisted Maps} \label{ss:twist}
Let $\vee$ (written exponentially) be the morphism 
$$ U^-\bcsl G \to G/U\ \text{ given by }\ (U^-g)^\vee \mapsto \br{\omega_0}^{-1} g^{-1} \br{\omega_0} U.$$
Let $\tilde{l}$ be the twisted cyclic shift on $\mc{A}\times \mc{A}\times \mc{A}^\vee$:
$$(A_1,A_2,A_3^\vee)\mapsto (A_2,A_3\br{w_0},(A_1\br{w_0})^\vee).$$
Suppose that $(A_1,A_2,A_3^\vee)=(U^- g_1, U^- g_2, g_3U)$, then
\begin{align}\label{eq:twistcyc} \tilde{l} (g(A_1,A_2,A_3^\vee)) &= \tilde{l} (U^-g_1g^{-1}, U^-g_2g^{-1}, g g_3U) \\ \notag
&= (U^-g_2g^{-1}, U^-\br{w_0} (gg_3)^{-1}, s_G(g_1g^{-1})^{-1} \br{w_0} U) \\ \notag
&= g (U^-g_2, U^-\br{w_0} g_3^{-1}, g_1^{-1}\br{w_0}^{-1} U)
\end{align}
By the universal property of categorical quotients, $\tilde{l}$ descends to an endomorphism $l$ on $\Conf_{2,1}$, which is an automorphism on the generic part of $\Conf_{2,1}$.

We have that
$$\tilde{l}^2(A_1,A_2,A_3^\vee)=(A_3\br{w_0}, s_GA_1,  (A_2\br{w_0})^\vee),$$
and $\tilde{l}^3 = (s_G A_1,s_G A_2,s_G A_3^\vee)= (A_1,A_2,A_3^\vee)$ is the identity.
We define $i_l:= l i_u: U\hookrightarrow \Conf_{2,1}$, that is
$$u\mapsto (U^-u,U^-\br{w_0}, \br{w_0}^{-1}U).$$
We write $r$ for $l^2$. Similarly we denote $i_r:= r i_u: U\hookrightarrow \Conf_{2,1}$.
$$\vcenter{\xymatrix@R=7ex@C=7ex{
& \Conf_{2,1} \ar[ddl]_l\\
& U \ar[u]|{i_u} \ar[dl]^{i_l} \ar[dr]_{i_r} \\
\Conf_{2,1}  \ar[rr]_{l} && \Conf_{2,1} \ar[uul]_{l}
}}$$

It turns out that the twisted cyclic shift is also related to the sequence $\bs{\mu}_l$ of mutations considered in Section \ref{ss:clusterU}.
We will see in Theorem \ref{T:cyclic} that $l^*(s_f)= \bs{\mu}_l(s_f)$ for $f$ mutable. Note that $l^*$ permutes frozen variables
$$s_{\zero_i^-}\mapsto s_{\zero_{i}^+},\ s_{\zero_i^+}\mapsto s_{\Id_{i^*}},\ s_{\Id_i}\mapsto s_{\zero_{i^*}^-}.$$

\begin{lemma} \label{L:itwist} We have the following
\begin{align*}
&i_{l}^*(s_f)=\bs{\mu}_l(m_f),\ i_{r}^*(s_f)=\bs{\mu}_r(m_f)\ \text{ for $f$ mutable}; \\
&i_{l}^*(s_{\Id_i})= m_{\zero_{i^*}^-},\ i_{l}^*(s_{\zero_i^-})=i_{l}^*(s_{\zero_i^+})=1; \\
&i_{r}^*(s_{\zero_i^+})= m_{\zero_{i}^-},\ i_{r}^*(s_{\zero_i^-})=i_{r}^*(s_{\Id_{i}})=1.
\end{align*}
\end{lemma}

\begin{proof} First consider the case when $f$ is mutable, then $l^*$ is given by a sequence of mutations $\bs{\mu}_l$.
The pullback $i_u^*$ clearly commutes with any sequence of mutations, so by Lemma \ref{L:i_u}
$$i_{l}^*(s_f)=i_u^*l^*(s_f)=i_u^*(\bs{\mu}_l(s_f)) =\bs{\mu}_l(i_u^*(s_f))=\bs{\mu}_l(m_f).$$
For $f$ frozen, by Lemma \ref{L:i_u} we have that
$$i_{l}^*(s_{\Id_i})=i_u^*(s_{\zero_{i^*}^-}) = m_{\zero_{i^*}^-},\
i_{l}^*(s_{\zero_i^-})=i_u^*(s_{\zero_i^+})=1,\ i_{l}^*(s_{\zero_i^+})=i_u^*(s_{\Id_{i^*}})=1.$$
The argument for $i_r^*$ is similar.
\end{proof}

We set $\Delta_Q^u:=\Delta_Q$ and $\mc{M}_Q^u:=\mc{M}_Q$. We define three linear natural projections $\mr{i}_\#^*:\mb{R}^{(\Delta_Q^2)_0}\to\mb{R}^{(\Delta_Q^\#)_0}$ for $\#\in\{u,l,r\}$,
where $\mr{i}_u^*$ (resp. $\mr{i}_l^*;\ \mr{i}_r^*$) forgets the coordinates corresponding to the neutral and positive (resp. positive and negative; negative and neutral) frozen vertices.

Recall the definition of the $\g$-vector with respect to a pair $(\Delta,\b{x})$ as in Section \ref{ss:gv}.
Let $\mc{L}(\mc{S}_Q^2)$ be the Laurent polynomial ring in $\mc{S}_Q^2$.
Note that $\mc{L}(\mc{S}_Q^2)$ has a basis parameterized by all possible $\g$-vectors (with respect to $(\Delta_Q^2,\mc{S}_Q^2)$), which can be identified with $\mb{Z}^{(\Delta_Q^2)_0}$.

\begin{lemma} \label{L:proj-g} If $s\in \mc{L}(\mc{S}_Q^2)\subset k(\Conf_{2,1})$ has a well-defined $\g$-vector $\g$ with respect to $(\Delta_Q^2,\mc{S}_Q^2)$, then so is the $\g$-vectors of $i_\#^*(s)$ with respect to $(\Delta_Q^\#,\mc{M}_Q^\#)$, which is equal to $\mr{i}_\#^*(\g)$ for $\#\in\{u,l,r\}$.
\end{lemma}

\begin{proof} We will only prove the statement for $i_l$ because the argument for the other two is similar.
Suppose that $s=\b{x}^\g F(\b{y})$ where $\b{x}(f)=s_f$ and $\b{y}$ is as in Section \ref{ss:gv}.
We have seen in Lemma \ref{L:itwist} that $i_l^*(s_f)=\bs{\mu}_l(m_f)=\mc{M}_Q^l(f)$ if $f$ is not negative or positive. If $f$ is negative or positive, then $i_l^*(s_f)=1$ so $i_l^*(\b{x}^{\g})=(\b{x}_l)^{\mr{i}_l^*(\g)}$ where $\b{x}_l(f)=\bs{\mu}_l(m_f)$.
Recall that the quiver $\Delta_Q^l$ is obtained from $\Delta_Q^2$ by deleting the negative and positive frozen vertices.
According to this description, we have that $i_l^*(\b{y}(f))=i_l^*(\b{x}^{-b_f})=\b{x}_l^{-b_f^l}=\b{y}_l(f)$,
where $b_f$ (resp. $b_f^l$) is the row of $B_{\Delta_Q^2}$ (resp. $B_{\Delta_Q^l}$) corresponding to $f$.
Hence, we get $i_l^*(s)=(\b{x}_l)^{\mr{i}_l^*(\g)} F(\b{y}_l)$.
\end{proof}

\begin{lemma} \label{L:noless} The polytope ${\sf G}_{\Delta_Q^2}(\tripar)$ has lattice points no less than $\lrcoef$.
\end{lemma}

\begin{proof} By Corollary \ref{C:localization}, we have that $k[\Conf_{2,1}]\subset \mc{L}(\mc{S}_Q^2)$.
Since $i_u,\ i_l,\ i_r$ are all regular, we trivially have that
\begin{equation} \label{eq:key} k[\Conf_{2,1}]\subseteq \left\{s\in \mc{L}(\mc{S}_Q^2)\mid i_{\#}^*(s)\in k[U] \text{ for $\#=u,l,r$} \right\}. \end{equation}
If $s\in \mc{L}(\mc{S}_Q^2)$ has a well-defined $\g$-vector, then by the previous lemma the $\g$-vector of $i_{\#}^*(s)$ is equal to $\mr{i}_{\#}^*(\g)$.
So if $i_{\#}^*(s)\in k[U]$, then $\mr{i}_{\#}^*(\g)\in {\sf G}_{\Delta_Q^{\#}}$ by Lemma \ref{L:independent}, Proposition \ref{P:GQ}, \ref{P:Ulr} and Remark \ref{r:GQlr}.

Due to \eqref{eq:key}, it suffices to show that the lattice points in ${\sf G}_{\Delta_Q^2}(\tripar)$ can be identified with
points $\g\in \mb{Z}^{(\Delta_Q^2)_0}$ of weight $(\tripar)$ such that $\mr{i}_{\#}^*(\g)\in {\sf G}_{\Delta_Q^{\#}}$ for $\#\in\{u,l,r\}$.
But this follows from the description of the cones ${\sf G}_{\Delta_Q^2},\ {\sf G}_{\Delta_Q},\ {\sf G}_{\Delta_Q^l}$ and ${\sf G}_{\Delta_Q^r}$ (Theorem \ref{T:GQ2}, Proposition \ref{P:GQ}, and Remark \ref{r:GQlr}).
\end{proof}

\section{Cluster Structure on $k[\Conf_{2,1}]$} \label{S:CS}

\begin{theorem} \label{T:main} Suppose that $Q$ is trivially valued.
Then the ring of regular functions on $\Conf_{2,1}$ is the graded upper cluster algebra $\ucaCtwoQ$.
Moreover, $(\Delta_Q^2,W_Q^2)$ is a cluster model.
In particular, $\lrcoef$ is counted by lattice points in ${\sf G}_{\Delta_Q^2}(\tripar)$.
\end{theorem}

\begin{proof} By Lemma \ref{C:algind} and Corollary \ref{C:WC}, $(\Delta_Q^2,\mc{S}_Q^2;\bs{\sigma}_{Q}^2)$ form a graded seed.
Due to Lemma \ref{L:exchange} and Corollary \ref{C:irreducible}, we can apply Lemma \ref{L:RCA} to conclude that
$\ucaCtwoQ$ is a graded subalgebra of $k[\Conf_{2,1}]$.

By Theorem \ref{T:GCC}, $C_W\left({\sf G}_{\Delta_Q^2}(\tripar)\cap \mb{Z}^{(\Delta_Q^2)_0}\right)$ is a linearly independent set in $C_{\mu,\nu}^{\lambda}$ for any triple weights.
But Lemma \ref{L:noless} says that the cardinality of ${\sf G}_{\Delta_Q^2}(\tripar)\cap \mb{Z}^{(\Delta_Q^2)_0}$ is at least $\lrcoef$.
So they actually span the vector space $C_{\mu,\nu}^{\lambda}$.
We conclude that $\ucaCtwoQ$ is equal to $k[\Conf_{2,1}]$, and $(\Delta_Q^2,W_Q^2)$ is a cluster model.
\end{proof}

\noindent It follows that \eqref{eq:key} is in fact an equality.
The proof shows that $\ucaCtwoQ$ is a graded subalgebra of $k[\Conf_{2,1}]$ no matter if $Q$ is trivially valued.

\begin{conjecture} \label{c:nonSL} The first and last statement of Theorem \ref{T:main} is true even if $Q$ is not trivially valued.
\end{conjecture}

We illustrate by example that in general the upper cluster algebra $\uca(\DCQ)$ strictly contains the cluster algebra $\mc{C}(\Delta_Q^2)$.
\begin{example} \label{ex:contain} Let $Q$ be of type $D_4$ as in Example \ref{ex:iARtD4}.
It can be checked by the 44 inequalities in Example \ref{ex:iARtD4_ct} that $\g=\e_{34,12}-\e_{34,1}+\e_{2,0}$ (in weight labelling) is $\mu$-supported.
Its triple weights is $(\e_2,\e_2,\e_2)$, and moreover $c_{\e_2,\e_2}^{\e_2}=1$. This is clearly an extremal weight.
So it suffices to show that no cluster variable has this $\g$-vector.
For this, we need a result in \cite{DF}, which says that the $\g$-vector of a cluster variable is a $\g$-vector of a {\em rigid} presentation in the Jacobian algebra.
The rigidity is characterized by the vanishing of the $\E$-invariant introduced in \cite{DWZ2,DF}.
It is not hard to show that $\E(f,f):=\Hom_{K^b(J)}(f,f[1])\neq 0$ where $f$ is a general presentation in $\Hom_J(P([\g]_+),P([\g]_-))$.
\end{example}

\section{Epilogue} \label{S:epi}
\subsection{Cluster Structure of Base Affine Spaces} \label{ss:CSBA}
We denote the ice quiver obtained from $\Delta_Q^2$ by deleting neutral frozen vertices by $\Delta_Q^\sharp$.
Let $\mc{M}_Q^\sharp:=\{m_{f}\}_{f\in (\Delta_Q^\sharp)_0}$.
We define the weight configuration $\bs{\sigma}_Q^\sharp$ on $\Delta_Q^\sharp$ by $\bs{\sigma}_Q^\sharp(f)=(\e(f),\f)$.
Recall from Lemma \ref{L:wtminor} that the degree of $m_f$ is $\varpi(\bs{\sigma}_Q^\sharp(f))$.
As a side result, we will show that the ring of regular functions on the base affine space $\mc{A}:=U^-\bcsl G$
is the graded upper cluster algebra $\uca(\Delta_Q^\sharp,\mc{M}_Q^\sharp; \varpi(\bs{\sigma}_Q^\sharp))$.

The proof is similar to but easier than that for $\Conf_{2,1}$ so our treatment may be a little sketchy.
Recall the open set $G_0=U^-HU$ of $G$. We have an open embedding
$$\iota_0: H\times U=U^-\bcsl G_0 \hookrightarrow U^-\bcsl G.$$
The localization of $k[\mc{A}]$ at all $m_f$'s for $f$ positive is exactly $k[H\times U]$.
In particular, $k[\mc{A}]$ is contained in the Laurent polynomial ring in $\mc{M}_Q^\sharp$.
Consider the open embedding
$$\iota=(\iota_1,\iota_2): H\times \mc{A} \hookrightarrow \Conf_{2,1},\ (h,U^-g)\mapsto (U^-h, U^-g, U).$$
Note that the map $\iota(\Id_H,\iota_0):H\times H\times U\to \Conf_{2,1}$ is the map $i$ defined in Section \ref{ss:standard}.
Let $p$ be a rational inverse of $\iota$, and $p_2$ be the map $p$ followed by the second component projection.
We define the birational map $r':U^-\bcsl G \to U^-\bcsl G$ to be the composition $r':=p_2 r \iota_2$.
The map $r'$ can be viewed as a variation of Fomin-Zelevinsky's {\em twist automorphism} \cite{FZ} on the big cell $U^-\bcsl G_0$.
Readers can check that they differ by a (fibrewise) rescaling along the toric fibre of $H\times U$, but we do not need this fact.

Let $i_{u'}:=\iota_0\mid_U$, and $i_{r'}:=r' i_{u'}$. Then we have the following commutative diagram. The map $i_{r'}$ is in fact regular because $r'$ is regular on the image of $i_{u'}$.
$$\vcenter{\xymatrix@R=5ex@C=5ex{
\Conf_{2,1} \ar[dd]_{p_2} \ar[rr]^{r} && \Conf_{2,1} \ar[dd]^{p_2} \\
& U \ar@{_(->}[ul]_{i_u} \ar@{_(->}[ur]^{i_r} \ar@{_(->}[dl]_{i_{u'}} \ar@{_(->}[dr]^{i_{r'}} \\
U^-\bcsl G \ar[rr]^{r'} && U^-\bcsl G
}}$$
\noindent To finish the proof, we only need three maps $i_{u'},\ i_{r'}$ and $r'$.
Analogous to Theorem \ref{T:cyclic}, we have that the pullback $(r')^*$ is related to the sequence of mutations $\bs{\mu}_r$.
More precisely, we have that
\begin{equation*}\bs{\mu}_r(m_f) = (r')^*(m_f) \text{ for any $f$ mutable.}
\end{equation*}
So analogous to Lemma \ref{L:itwist}, we have that $i_{r'}^*(m_f)=(r')^*(m_f)=\bs{\mu}_r(m_f)$ for $f$ mutable.
Obviously we also have that
$i_{r'}^*(m_{\zero_i^+})= m_{\zero_{i}^-}$ and $i_{r'}^*(m_{\zero_i^-})=1$.

Analogous to $\mr{i}_{u}^*$ and $\mr{i}_{r}^*$, we have the maps $\mr{i}_{u'}^*:\mb{R}^{(\Delta_Q^\sharp)_0}\to\mb{R}^{(\Delta_Q)_0}$ and $\mr{i}_{r'}^*:\mb{R}^{(\Delta_Q^\sharp)_0}\to\mb{R}^{(\Delta_Q^r)_0}$. The map $\mr{i}_{u'}^*$ (resp. $\mr{i}_{r'}^*$) forgets the coordinates corresponding to the positive (resp. negative) frozen vertices. We have the following analog of Lemma \ref{L:proj-g}.
If $m\in \mc{A}$ has a well-defined $\g$-vector $\g$ with respect to $(\mc{M}_Q^\sharp,\Delta_Q^\sharp)$, then so are the $\g$-vectors of $i_{u'}^*(m)$ and $i_{r'}^*(m)$ with respect to $(\mc{M}_Q,\Delta_Q)$ and $(\mc{M}_Q^r,\Delta_Q^r)$.
They are equal to $\mr{i}_{u'}^*(\g)$ and $\mr{i}_{r'}^*(\g)$ respectively.

We restrict the iARt QP $(\Delta_Q^2,W_Q^2)$ to the subquiver $\Delta_Q^\sharp$. We denote the restricted potential by $W_Q^\sharp$.
Almost the same proof as in Theorem \ref{T:GQ2} shows that the set $G(\Delta_Q^\sharp,W_Q^\sharp)$ of $\mu$-supported $\g$-vectors are given by the lattice points in the polyhedral cone ${\sf G}_{\Delta_Q^\sharp}$. The cone is defined by the two sets of defining conditions of ${\sf G}_{\Delta_Q^2}$, namely, $\g H_u\geq 0$ and $\g H_r\geq 0$.
By almost the same argument, we get the following analog of Lemma \ref{L:noless}.
For the weight configuration $\varpi(\bs{\sigma}_Q^\sharp)$ we define the polytope ${\sf G}_{\Delta_Q^\sharp}(\mu,\lambda)$ as in Definition \ref{D:polytope}.

\begin{lemma} \label{L:noless1}
The polytope ${\sf G}_{\Delta_Q^\sharp}(\mu,\lambda)$ has lattice points no less than $\dim k[\mc{A}]_{\mu,\lambda}$.
\end{lemma}

\begin{theorem}  \label{T:side} Suppose that $Q$ is trivially valued.
Then the ring of regular functions on $\mc{A}$ is the graded upper cluster algebra $\uca(\Delta_Q^\sharp,\mc{M}_Q^\sharp;\varpi(\bs{\sigma}_Q^\sharp))$.
Moreover, $(\Delta_Q^\sharp,W_Q^\sharp)$ is a cluster model.
In particular, the weight multiplicity $\dim L(\mu)_{\varpi(\lambda)}$ is counted by lattice points in ${\sf G}_{\Delta_Q^\sharp}(\mu,\lambda)$.
\end{theorem}

\begin{proof} We know from \cite{BFZ} that $(\Delta_Q^\sharp,\mc{M}_Q^\sharp)$ is a seed satisfying the condition of Lemma \ref{L:RCA}.
By Lemma \ref{L:RCA}, $\uca(\Delta_Q^\sharp,\mc{M}_Q^\sharp;\varpi(\bs{\sigma}_Q^\sharp))$ is a graded subalgebra of $k[\mc{A}]$.

By Theorem \ref{T:GCC}, $C_W\left({\sf G}_{\Delta_Q^\sharp}(\mu,\lambda)\cap \mb{Z}^{(\Delta_Q^\sharp)_0}\right)$ is a linearly independent set in $k[\mc{A}]_{\mu,\lambda}$ for any $\mu,\lambda$.
But Lemma \ref{L:noless1} implies that they actually span the vector space $k[\mc{A}]_{\mu,\lambda}$.
We conclude that $\uca(\Delta_Q^\sharp,\mc{M}_Q^\sharp;\varpi(\bs{\sigma}_Q^\sharp))$ is equal to $k[\mc{A}]$, and $(\Delta_Q^\sharp,W_Q^\sharp)$ is a cluster model.
\end{proof}

\begin{remark} For $Q$ of type $A$, the theorem was proved in \cite[Example 3.10]{Fs2} using the technique of projection.
The map $p_2$ is induced from a semi-orthogonal decomposition of the module category of a triple flag quiver.

In general, the cluster algebra $\mc{C}(\Delta_Q^\sharp)$ is also strictly contained in its upper cluster algebra.
The example is still given by the $Q$ of type $D_4$. We take the same $\g$-vector as in Example \ref{ex:contain}.
\end{remark}

\begin{remark}
Let $K^2Q:=K^2(\proj Q)$ be the homotopy category of $C^2Q$ as in \cite{DF}.
Then the ice quiver $\Delta_Q^\sharp$ can be obtained from the ARt quiver $\Delta(K^2Q)$ by freezing the negative and positive vertices.

There are two other seeds mutation equivalent to this seed via $\bs{\mu}_l$ and $\bs{\mu}_r$. Their quivers are obtained from $\Delta_Q^2$ by deleting the positive and negative frozen vertices respectively.
\end{remark}

%

\subsection{Remark on the Non-simply Laced Cases}
Suppose that $G$ is not simply laced, or equivalently $Q$ is not trivially valued.
To prove Conjecture \ref{c:nonSL}, let us exam the arguments in proving Theorem \ref{T:main}.
We find that there are two missing ingredients.
One is the straightforward generalization of Theorem \ref{T:cyclic},
and the other is an analogous cluster character for species with potentials.
Although the proof of Theorem \ref{T:cyclic} depends on a result involving preprojective algebras,
the author has a much longer proof for all cases using only a conjectural generalization of Lemma \ref{L:key} to the species with potentials.
For the missing cluster character, as proved in \cite{Fs1} for the usual QPs, the existence of such a map to the upper cluster algebra is also equivalent to Lemma \ref{L:key}. The argument there can be generalized.
So the upshot is that the only missing part for Conjecture \ref{c:nonSL} is certain generalization of Lemma \ref{L:key}.
But to author's best knowledge, the existing theory of species with potentials, such as \cite{LZ}, is far from reaching this goal.

\begin{appendices}
\section{The Twisted Cyclic Shift via Mutations} \label{A:cyclic}
\subsection{Mutation of Representations, $\g$-vectors, and $F$-polynomials} \label{ss:mu}
We review some material from \cite{DWZ2}. Let $(\Delta,W)$ be a QP as in Section \ref{ss:QP}.
The {\em mutation} $\mu_u$ of $(\Delta,W)$ at a vertex $u$ is defined as follows.
The first step is to define the following new QP $\wtd{\mu}_u(\Delta,W)=(\wtd{\Delta},\wtd{W})$.
We put $\wtd{\Delta}_0=\Delta_0$ and $\wtd{\Delta}_1$ is the union of three different kinds
\begin{enumerate}
\item[$\bullet$] all arrows of $\Delta$ not incident to $u$,
\item[$\bullet$] a composite arrow $[ab]$ from $t(a)$ to $h(b)$ for each $a,b$~with~$h(a)=t(b)=u$,
\item[$\bullet$] an opposite arrow $a^*$ (resp. $b^*$) for each incoming arrow $a$ (resp. outgoing arrow $b$) at $u$.
\end{enumerate}
Note that this $\wtd{\Delta}$ is the result of first two steps in Definition \ref{D:Qmu}.
The new potential on $\wtd{\Delta}$ is given by
$$\wtd{W}:=[W]+\sum_{h(a)=t(b)=u}b^*a^*[ab],$$
where $[W]$ is obtained by substituting $[ab]$ for each words $ab$ occurring in $W$.
Finally we define $(\Delta',W')=\mu_u(\Delta,W)$ as the {\em reduced part} (\cite[Definition 4.13]{DWZ1}) of $(\wtd{\Delta},\wtd{W})$.
For this last step, we refer readers to \cite[Section 4,5]{DWZ1} for details.

Now we start to define the mutation of decorated representations of $(\Delta,W)$.
Consider the resolution of the simple module $S_u$
\begin{align}
\label{eq:Su_proj} \cdots \to \bigoplus_{h(a)=u} P_{t(a)}\xrightarrow{_a(\partial_{[ab]})_b} \bigoplus_{t(b)=u} P_{h(b)} \xrightarrow{_b(b)}  P_u \to S_u\to 0,\\
\label{eq:Su_inj} 0\to S_u \to I_u \xrightarrow{(a)_a} \bigoplus_{h(a)=u} I_{t(a)} \xrightarrow{_a(\partial_{[ab]})_b} \bigoplus_{t(b)=u} I_{h(b)} \to \cdots,
\end{align}
where $I_u$ is the indecomposable injective representation of $J(\Delta,W)$ corresponding to a vertex $u$.
We thus have the triangle of linear maps with $\beta_u \gamma_u=0$ and $\gamma_u \alpha_u=0$.
$$\vcenter{\xymatrix@C=5ex{
& M(u) \ar[dr]^{\beta_u} \\
\bigoplus_{h(a)=u} M(t(a)) \ar[ur]^{\alpha_u} && \bigoplus_{t(b)=u} M(h(b)) \ar[ll]^{\gamma_u} \\
}}$$

We first define a decorated representation $\wtd{\mc{M}}=(\wtd{M},\wtd{M}^+)$ of $\wtd{\mu}_u(\Delta,W)$.
We set \begin{align*}
&\wtd{M}(v)=M(v),\quad  \wtd{M}^+(v)=M^+(v)\quad (v\neq u); \\
&\wtd{M}(u)=\frac{\Ker \gamma_u}{\Img \beta_u}\oplus \Img \gamma_u \oplus \frac{\Ker \alpha_u}{\Img \gamma_u} \oplus M^+(u),\quad \wtd{M}^+(u)=\frac{\Ker \beta_u}{\Ker \beta_u\cap \Img \alpha_u}.
\end{align*}
We then set $\wtd{M}(a)=M(a)$ for all arrows not incident to $u$, and $\wtd{M}([ab])=M(ab)$.
It is defined in \cite{DWZ1} a choice of linear maps
$\wtd{M}(a^*), \wtd{M}(b^*)$ making
$\wtd{M}$ a representation of $(\wtd{\Delta},\wtd{W})$.
We refer readers to \cite[Section 10]{DWZ2} for details.
Finally, we define $\mc{M}'=\mu_u(\mc{M})$ to be the {\em reduced part} (\cite[Definition 10.4]{DWZ1}) of $\wtd{\mc{M}}$.

Recall the $\g$-vector form Definition \ref{D:gv}. We can also define the dual $\g$-vectors using the injective presentations.
Let $M$ be a representation of $J(\Delta,W)$ with minimal injective presentation $0\to M\to I(\beta_-^\vee)\to I(\beta_+^\vee)$, then the {\em dual $\g$-vector} $\g^\vee(M)$ of $M$ is the reduced weight vector $\g^\vee=\beta_+^\vee-\beta_-^\vee$.
The definition can also be extended to all decorated representations similar to $\g$-vectors.

Recall the $y$-variables $y_u=\b{x}^{-b_u}$ as in Section \ref{ss:gv}. The seed mutation of Definition \ref{D:seeds} induces the {\em $\b{y}$-seed mutation}.
We recall the mutation rule \cite[(3.8)]{FZ4}.
Let $(\b{y}',\Delta'):=\mu_u(\b{y},\Delta)$ and $B(\Delta)=(b_{u,v})$, then $\Delta'=\mu_u(\Delta)$ and
\begin{equation} \label{eq:mu_y} y_v'=\begin{cases}
y_u^{-1} & \text{if } v=u;\\
y_v y_u^{[-b_{u,v}]_+}(y_u+1)^{b_{u,v}} & \text{if } v\neq u.
\end{cases}
\end{equation}

\begin{definition}[\cite{DWZ2}] 
We define the dual $F$-polynomial of a representation $M$ by
\begin{equation} F_M^\vee(\b{y}) = \sum_\e \chi(\Gr_\e(M)) \b{y}^\e,
\end{equation}
where $\Gr_{\e}(M)$ is the variety parameterizing $\e$-dimensional subrepresentations~of~$M$.
\end{definition}

\begin{remark} The $F$-polynomial of $M$ is $F_M(\b{y}) = \sum_\e \chi(\Gr^\e(M)) \b{y}^\e$. We only need the dual version in \ref{ss:algorithm}.
Our dual $\g$-vector and dual $F$-polynomial is the $\g$-vector and $F$-polynomial in \cite{FZ4,DWZ2}.
\end{remark}

Here is the key Lemma in \cite{DWZ2}.
\begin{lemma} \label{L:key} Let $\mc{M}$ be an arbitrary representation of a nondegenerate QP $(\Delta,W)$, and let $\mc{M}'=\mu_u(\mc{M})$, then \begin{enumerate}
\item The $F$-polynomials of ${M}$ and ${M}'$ are related by
\begin{equation} \label{eq:F-poly} (y_u+1)^{-\beta_-^\vee(u)} F_{M}^\vee(\b{y})=(y_u'+1)^{-(\beta_-^\vee)'(u)} F_{M'}^\vee(\b{y}').
\end{equation}
\item The $\g$-vector and $\g^\vee$-vector of $\mc{M}$ and $\mc{M}'$ are related by
\begin{equation} \label{eq:g_mu}  \g'(v)=\begin{cases}
-\g(u) & \text{if } u=v;\\
\g(v)+[b_{v,u}]_+\g(u)+b_{v,u}\beta_-(u) & \text{if } u\neq v,
\end{cases}
\end{equation}
\begin{equation} \label{eq:gdual_mu}  (\g^\vee)'(v)=\begin{cases}
-\g^\vee(u) & \text{if } u=v;\\
\g^\vee(v)+[-b_{v,u}]_+\g^\vee(u)-b_{v,u}\beta_-^\vee(u) & \text{if } u\neq v.
\end{cases}
\end{equation}
\end{enumerate}
\end{lemma}

\begin{remark} \label{r:gcoh} If $\mc{M}$ is {\em $\g$-coherent} (that is $\min(\beta_+(u),\beta_-(u))=0$ for all~vertices~$u$,
or equivalently, $\beta_+=[\g]_+$ and $\beta_-=[-\g]_+$), then \eqref{eq:g_mu} reads
\begin{equation} \label{eq:g_mu_gen} \g'(v):=\begin{cases}
-\g(u) & \text{if } u=v;\\
\g(v)+b_{v,u}[\g(u)]_+ & \text{if } b_{v,u}\geq 0; \\
\g(v)+b_{v,u}[-\g(u)]_+ & \text{if } b_{v,u}\leq 0.
\end{cases}\end{equation}
It is known \cite{Fs1} that if $\mc{M}$ corresponds to a general presentation, then $\g(\mc{M})$ is $\g$-coherent.
If $\mc{M}$ is obtained from positive simple $(0,S_u)$ via a sequence of mutations, then $\mc{M}$ corresponds to an indecomposable rigid presentation.
In particular, it is general \cite{DF}.
\end{remark}

\subsection{The Twisted Cyclic Shift via Mutations} \label{ss:cyclic}
We recall from Section \ref{ss:twist} the twisted cyclic shift $l$ on $\Conf_{2,1}$ induced by
$$(A_1,A_2,A_3^\vee)\mapsto (A_2, A_3\br{w_0}, (A_1\br{w_0})^\vee),$$
where $\vee: \mc{A} \to \mc{A}^\vee,\ (U^-g)^\vee \mapsto \br{\omega_0}^{-1} g^{-1} \br{\omega_0} U$.
To understand this map, we introduce a ``half" of this map.
Consider the involution $*: G\to G$ as in \cite{GSdt} defined by
$$x_i(a)^*=x_{i^*}(a),\ y_{i}(b)^*=y_{i^*}(b),\ h^*=\br{w_0}^{-1} h^{-1} \br{w_0},\ (i\in Q_0, h\in H).$$
The involution $*$ preserves $U,U^-$ and $H$ so it induces an involution of $\mc{A}$ and $\mc{A}^\vee$.
Thus it acts on $\Conf_{2,1}$.
We consider the automorphism of $\Conf_{2,1}$ induced by
$$(A_1,A_2,A_3^\vee)\mapsto (A_3^*, (A_1\br{w_0})^*, (A_2^*)^\vee).$$
We denote this map by $\sqrt{l}$. It is clear that $l=(\sqrt{l})^2$.
It will turn out (see Remark \ref{r:DT}) that this map is the {\em Donaldson-Thomas transformation} of $\Conf_{2,1}$ in the sense of \cite{GSdt}.

Let us recall a sequence of mutations constructed in \cite[Section 13]{GLSk}.
The sequence is originally defined for the ice quiver $\Delta_Q$ in terms of reduced expression of $w_0$. We now translate it into our setting.
Recall that we can label the non-neutral vertices of $\Delta_Q^2$ by a pair $(i,t)$ (before Lemma \ref{L:sigma2rho}).
Let $$t_i:=t^*(\zero_i^+)=\max\{t \mid (i,t)\in\Delta_Q^2 \}.$$
We first assume that the vertices of $Q$ are ordered such that $i<j$ if $(i,j)\in Q_1$.
Then we totally order the mutable vertices of $\Delta_Q^2$ by the relation that $(i,t)<(i',t')$ if $t<t'$, or $t= t'$ and $i<i'$.

Starting from the minimal vertex $(1,1)$ in the ascending order just defined, we perform a sequence of mutations $\bs{\mu}_{i,t}$ for each (mutable) vertex of $\Delta_Q^2$.
For the vertex $(i,t)$, the sequence of mutations is defined to be $\bs{\mu}_{i,t}:=\mu_{i,t_i-t}\cdots\mu_{i,2}\mu_{i,1}$.
So the whole sequence of mutations $\bs{\mu}_{\sqrt{l}} := \cdots \bs{\mu}_{2,1}\bs{\mu}_{1,1}$.

Let $\pi$ be the permutation on $(\Delta_Q^2)_0$ defined by
\begin{gather*}  (i,t)\mapsto (i,t_i-t)\ \text{ if $(i,t)$ is mutable,}\\
\zero_i^-\mapsto \zero_{i^*}^+,\ \zero_i^+\mapsto \Id_{i},\ \Id_i \mapsto \zero_i^-.
\end{gather*}
It is clear that $\pi$ is an involution on the set of mutable vertices.
When applying it to the quiver $\Delta_Q^2$ or its $B$-matrix, we view $\pi$ as a relabelling of the vertices.

For a (trivially valued) Dynkin quiver $Q$, let $\Pi_Q$ be its associated {\em preprojective algebra}.
Recall that the module category $\module\Pi_Q$ is a {\em Frobenius category}.
Let $\underline{\module}\Pi_Q$ be its {\em stable category}.
Recall that such a stable category is naturally triangulated with the shift functor given by the (relative) inverse Syzygy functor $\Omega^{-1}$.
Readers can find these standard terminology in, for example, \cite{GLSk}.
We will write $\Hom_{\underline{\Pi}_Q}$ for $\Hom_{\underline{\module}{\Pi_Q}}$.
In \cite{GLSa} the authors defined a tilting module $V:=\bigoplus_{u\in (\Delta_Q)_0} V_u$ in $\module\Pi_Q$
(it is denoted by ${\rm I}_Q$ in \cite{GLSa}).
They proved that the quiver of the endomorphism algebra $\End_{\Pi_Q}(V)^\opp$ is exactly $\Delta_Q$ with its frozen vertices of corresponding to the projective-injective objects in $\module\Pi_Q$.
Moreover, it is known \cite{BIRS} that the {\em stable} endomorphism algebra $\End_{\underline{\Pi}_Q}(V)^\opp$ is the Jacobian algebra $J^\mu$ of the QP $(\Delta_Q^2,W_Q^2)$ restricted to its mutable part $\Delta_Q^\mu$.

For each rigid $\Pi_Q$-module $M$, we can associate a function $\varphi_M\in k[U]$ as in \cite{GLSr}.
This function turns out to be a cluster variable for the standard seed $(\Delta_Q,\mc{M}_Q)$.
In \cite{GLSr} the {\em mutation of maximal rigid modules} of $\Pi_Q$ is defined so that
it is compatible with the seed mutation.
In particular, it is compatible with the quiver mutation, so
the quiver of $\End_{\Pi_Q}(\mu_u(V))$ is exactly $\mu_u(\Delta_Q)$.

\begin{proposition}{\cite[Proposition 13.4]{GLSk}} \label{P:mu_sqrtl_Q} The sequence of mutations $\bs{\mu}_{\sqrt{l}}$ takes the modules $V_{\pi(u)}$ to $\Omega^{-1}(V_{u})$ for $u$ mutable in $\Delta_Q$.
\end{proposition}

\begin{corollary} \label{C:mu_sqrtl_Q} Identifying $\Delta_Q$ as a subquiver of $\Delta_Q^2$,
we have that $\bs{\mu}_{\sqrt{l}}(\Delta_Q)$ and $\pi(\Delta_Q)$ have the same restricted $B$-matrix.
\end{corollary}

\begin{proof} Since $\Omega$ is an autoequivalence, $\bs{\mu}_{\sqrt{l}}(\Delta_Q)$ and $\pi(\Delta_Q)$ have the same mutable part.
Let $\br{\bs{\mu}}_{i,t}:=\bs{\mu}_{i,t}\cdots\bs{\mu}_{1,1}$.
In \cite[Section 13.1]{GLSk}, the authors described the each mutated quiver after applying each $\br{\bs{\mu}}_{i,t}$.
It follows easily from their description that after forgetting arrows between frozen vertices,
$\bs{\mu}_{\sqrt{l}}(\Delta_Q)$ is a subquiver of $\Delta_Q^2$,
whose frozen vertex $\zero_{i^*}^+$ is identified with the frozen vertex $\zero_i^-$ of $\bs{\mu}_{\sqrt{l}}(\Delta_Q)$.
\end{proof}

Let $(\Delta_Q^\mu,W_Q^\mu)$ be the restriction of $(\Delta_Q,W_Q)$ to its mutable part.
Next, we compute the $\g$-vectors of cluster variables after applying $\bs{\mu}_{\sqrt{l}}$ to the seed $(\Delta_Q^\mu,\b{x})$.
The $\g$-vector of the initial cluster variable $\b{x}_{i,t}$ at the vertex $(i,t)$ is the unit vector $\e_{i,t}\in \mb{Z}^{(\Delta_Q^\mu)_0}$.
By \cite[Theorem 5.2]{DWZ2},
the $\g$-vector of $\bs{\mu}_{\sqrt{l}}(\b{x}_{i,t})$ is nothing but the $\g$-vector of $\bs{\mu}_{\sqrt{l}}^{-1}(0,S_{i,t})$, where $(0,S_{i,t})$ is the positive simple representation of $\bs{\mu}_{\sqrt{l}}(\Delta_Q^\mu,W_Q^\mu)$.

\begin{lemma} \label{L:mu_sqrtl_QG} We have that $\bs{\mu}_{\sqrt{l}}^{-1}(\e_{i,t}) = -\e_{i,t_i-t}$ for each $\e_{i,t}$.
Here, we view $\e_{i,t}$ as the $\g$-vector of the simple representation $(0,S_{i,t})$ of $\bs{\mu}_{\sqrt{l}}(\Delta_Q^\mu,W_Q^\mu)$.
\end{lemma}
\begin{proof} It is equivalent to show that $\bs{\mu}_{\sqrt{l}}(-\e_{i,t_i-t}) = \e_{i,t}$.
Let $\br{\bs{\mu}}_{i,t}:=\bs{\mu}_{i,t}\cdots\bs{\mu}_{1,1}$.
We study the mutated $\g$-vector $\br{\bs{\mu}}_{i,t}(-\e_{i,t_i-t})$ for each $(i,t)$.
From the description of mutated quiver $\br{\bs{\mu}}_{i,t}(\Delta_Q)$ \cite[Section 13.1]{GLSk} and the formula \eqref{eq:g_mu_gen}, we find that for $j\neq i$,
each $\bs{\mu}_{j,s}$ does not play any role in $\bs{\mu}_{\sqrt{l}}(-\e_{i,t_i-t})$.
In other words, it suffices to compute $\bs{\mu}_{\sqrt{l}}(-\e_{i,t_i-t})$ on the full linear subquiver with vertices $(i,t)$ for $t=1,\cdots,t_i-1$.
Using the formula \eqref{eq:g_mu_gen}, we can easily prove by induction that
$$\br{\bs{\mu}}_{i,s} (-\e_{i,t_i-t})= \begin{cases}
\e_{i,t_i-s} - \e_{i,t_i-s-t} & \text{if $s+t< t_i$,} \\
\e_{i,t} & \text{otherwise.}
\end{cases}$$
We conclude from this that $\bs{\mu}_{\sqrt{l}}(-\e_{i,t_i-t}) = \e_{i,t}.$
\end{proof}

Let $\b{g}_{\sqrt{l}}^\pi$ be the set of $\g$-vectors given by
\begin{equation} \label{eq:extg_sqrtl}
\b{g}_{\sqrt{l}}^\pi(f) = -\e_{f} + \sum_{i\in Q_0} \left(\e(f)(i)\e_{\zero_i^-}+\f_+(i)\e_{\zero_i^+}+\f_-(i)\e_{\Id_i}\right).
\end{equation}
Let $\bs{\sigma}_{\sqrt{l}}^\pi$ be the weight configuration defined by
$$\bs{\sigma}_{\sqrt{l}}^\pi(f)=(\f_+,\e(f)^*,\f_-),$$
where $\f\mapsto \f^*$ is the involution on $\mb{Z}^{Q_0}$ induced by $\e_i\mapsto \e_{i^*}$,
so $-w_0(\varpi(\f)) = \varpi(\f^*)$.

\begin{corollary} \label{C:mu_sqrtl_Q2G} For any $f$ mutable, we have that
$\bs{\mu}_{\sqrt{l}}^{-1}(\e_f) = \b{g}_{\sqrt{l}}^\pi(\pi(f))$.
Here, we view $\e_f$ as the $\g$-vector of the simple representation $(0,S_f)$ of $\bs{\mu}_{\sqrt{l}}(\Delta_Q^2,W_Q^2)$.
So $\bs{\mu}_{\sqrt{l}}(\bs{\sigma}_Q^2)(f)= \bs{\sigma}_{\sqrt{l}}^\pi(\pi(f))$.
\end{corollary}

\begin{proof} In Lemma \ref{L:mu_sqrtl_QG} we have obtained the principal part of $\bs{\mu}_{\sqrt{l}}^{-1}(\e_f)$, which is $-\e_{\pi(f)}$.
Now we are going to recover its extended part from its principal part.

We recall from \cite[Proposition 5.15]{Fs1} that being $\mu$-supported for a $\g$-vector is mutation invariant.
The general presentation with $\g$-vector equal to $-\e_{f}$ corresponds to the indecomposable projective representation $\P_{f}$.
By Lemma \ref{L:iARt_path} $\P_{f}(\zero_i^+) \cong \Hom_Q(P_+,P_i)$.
So to make the $\g$-vector $-\e_{f}$ not supported on $\zero_i^+$, we must add at least $\f_+(i)\e_{\zero_i^+}$ to it.
Similarly we must add at least $\f_-(i)\e_{\Id_i}$ and $\e(f)(i)\e_{\zero_i^-}$ as well.
On the other hand, it is easy to check that \eqref{eq:extg_sqrtl} is $\mu$-supported.
So adding any additional positive component on the frozen vertices will make the general presentation of this $\g$-vector decomposable, which is not the case.
Hence, $\bs{\mu}_{\sqrt{l}}(\e_f) = \b{g}_{\sqrt{l}}^\pi(\pi(f))$.

For the last statement, recall that $\bs{\sigma}_Q^2(f)=(\e(f),\f_-,\f_+)$.
Then \begin{align*}
\b{g}_{\sqrt{l}}^\pi(f) \bs{\sigma}_Q^2 &=
-(\e(f),\f_-,\f_+)+(\e(f),\e(f)^*,0)+(\f_+,0,\f_+)+(0,\f_-,\f_-) \\ &= (\f_+,\e(f)^*,\f_-).\\
\text{Hence,\quad } \bs{\mu}_{\sqrt{l}}(\bs{\sigma}_Q^2)(f) &= \bs{\mu}_{\sqrt{l}}(\e_f)\bs{\sigma}_Q^2 = \b{g}_{\sqrt{l}}^\pi(\pi(f)) \bs{\sigma}_Q^2= \bs{\sigma}_{\sqrt{l}}^\pi(\pi(f)).
\end{align*}
\end{proof}

\begin{corollary} \label{C:mu_sqrtl} $\bs{\mu}_{\sqrt{l}}(\Delta_Q^2)$ and $\pi(\Delta_Q^2)$ have the same restricted $B$-matrix.
\end{corollary}

\begin{proof} By Corollary \ref{C:mu_sqrtl_Q}, $\bs{\mu}_{\sqrt{l}}(\Delta_Q^2)$ and $\pi(\Delta_Q^2)$ have the same mutable part as well.
It remains to show that the blocks of their restricted $B$-matrices corresponding to positive and neutral frozen vertices are equal.
But this follows from an easy linear algebra consideration. Indeed,
let $B_{\sqrt{l}}^2$ and $B_{\sqrt{l}}$ be the restricted $B$-matrices of $\bs{\mu}_{\sqrt{l}}(\Delta_Q^2)$ and $\bs{\mu}_{\sqrt{l}}(\Delta_{Q})$ and denote $\bs{\sigma}_{\sqrt{l}}^2:=\bs{\mu}_{\sqrt{l}}(\sCQ)$.
We have that $B_{\sqrt{l}}^2 \bs{\sigma}_{\sqrt{l}}^2 = 0$.
We write $B_{\sqrt{l}}^2$ and $\bs{\sigma}_{\sqrt{l}}^2$ in blocks:
$(B_{\sqrt{l}},C_{\sqrt{l}})$ and $\sm{\bs{\sigma}_{\sqrt{l}}\\ \bs{\sigma}_{+,\Id}},$
where $\bs{\sigma}_{+,\Id}=\sm{I & 0 & I\\ 0 & I & I}$ contains weights corresponding to the positive and neutral frozen vertices. We have that
\begin{equation} \label{eq:B+0} (B_{\sqrt{l}},C_{\sqrt{l}})\sm{\bs{\sigma}_{\sqrt{l}}\\ \bs{\sigma}_{+,\Id}}=B_{\sqrt{l}}\bs{\sigma}_{\sqrt{l}} + C_{\sqrt{l}}\bs{\sigma}_{+,\Id}= 0. \end{equation}
Since $\bs{\sigma}_{+,\Id}$ is of full rank, this linear equation is underdetermined (for solving $C_{\sqrt{l}}$).
We have seen that $B_{\sqrt{l}}= B_{\pi(\Delta_Q)}$ and $\bs{\sigma}_{\sqrt{l}}(f)=\bs{\sigma}_{\sqrt{l}}^\pi(\pi(f))$,
then it is clear that $C_{\sqrt{l}}=\pi(C)$ satisfies \eqref{eq:B+0}, where $C$ is the block in $B_{\Delta_Q^2}=(B_{\Delta_Q},C)$.
\end{proof}

\begin{definition}Let $\bs{\mu}_{\sqrt{l}}^\pi := \pi\bs{\mu}_{\sqrt{l}}\pi^{-1}$.
We define $\bs{\mu}_l:=\bs{\mu}_{\sqrt{l}}^\pi\bs{\mu}_{\sqrt{l}}$ and $\bs{\mu}_r=\bs{\mu}_l^2$.
\end{definition}

\begin{corollary} \label{C:mu_l} $\bs{\mu}_{l}(\Delta_Q^2)$ and $\pi^2(\Delta_Q^2)$ have the same restricted $B$-matrix.
Moreover, $\bs{\mu}_{l}(\bs{\sigma}_Q^2)(f)= (\f_-,\f_+^*,\e(f)^*)$ for $f$ mutable.
\end{corollary}
\noindent Note that $\pi^{2}$ fixes the mutable vertices and shuffles the frozen vertices of $\Delta_Q^2$
$$\zero_{i}^+\mapsto \zero_{i}^-,\ \zero_i^-\mapsto \Id_{i^*},\ \Id_{i^*}\mapsto \zero_{i}^+.$$

\begin{corollary} \label{C:mu_l^3}
We have that $\bs{\mu}_l^3(B_{\Delta_Q^2},\b{x})=(B_{\Delta_Q^2},\b{x})$.
\end{corollary}

\begin{proof} Since $\pi^6=\Id$, the fact that $\bs{\mu}_l^3(B_{\Delta_Q^2})=B_{\Delta_Q^2}$ follows from Corollary \ref{C:mu_l}.
As pointed in \cite{GLSa}, it is well-known \cite{AR} that the functor $\Omega$ is also 6-periodic.
\cite[Proposition 6.3]{GLSg} says that the mutation of maximal rigid modules of $\Pi_Q$ is compatible with the mutation of the corresponding QP-representations (see remarks before Proposition \ref{P:mu_sqrtl_Q}).
So by Proposition \ref{P:mu_sqrtl_Q} the positive simples $(0,S_f)$ of $(\Delta_Q,W_Q)$ are invariant under $\bs{\mu}_l^3$.
The invariance obviously extends to $(\Delta_Q^2,W_Q^2)$. The desired result follows.
\end{proof}

\noindent By experiment with \cite{Ke} we believe the following conjecture.
\begin{conjecture} \label{c:cyclic_valued} Corollary \ref{C:mu_sqrtl} (and thus Corollary \ref{C:mu_l}, \ref{C:mu_l^3} and Theorem \ref{T:cyclic} below) hold for non-trivially valued $Q$ as well.
\end{conjecture}


We define a rational self-map $\mu_l$ of $\ucaCtwoQno$ induced by 
$$\begin{cases}
s_f\mapsto \bs{\mu}_l(s_f)\ \text{ if $f$ is mutable,}\\
s_{\zero_{i}^+}\mapsto s_{\zero_i^-},\ s_{\zero_{i}^-}\mapsto s_{\Id_{i^*}},\ s_{\Id_{i^*}}\mapsto s_{\zero_i^+}.
\end{cases}$$
Since our main theorem (Theorem \ref{T:main}) use the following theorem, we cannot assume that $\ucaCtwoQno = k[\Conf_{2,1}]$.
But from the proof of Theorem \ref{T:main}, we have that $\ucaCtwoQno\subseteq k[\Conf_{2,1}]$.
So the pullback $l^*$ is defined on $\ucaCtwoQno$.

\begin{theorem} \label{T:cyclic} The map $\mu_l$ is equal to
the pullback $l^*$ of the twisted cyclic shift.
\end{theorem}

\begin{proof} We need to show that $\mu_l(\varphi)= \varphi l$.
It suffices to prove the statement for $\varphi$'s in some extended cluster
because of the definition of the upper cluster algebra.
The cluster we shall take is $\bs{\mu}_l^2(\mc{S}_Q^2)$, which is the same as $\bs{\mu}_l^{-1}(\mc{S}_Q^2)$ by Corollary \ref{C:mu_l^3}.
The statement can be easily checked for the frozen variables. For example, $s_{\zero_{i}^+} = s_{\Id_{i^*}} l$, that is
\begin{equation} \label{eq:rel} s_{\zero_{i}^+}(A_1, A_2, A_3^\vee) = s_{\Id_{i^*}}(A_2, A_3\br{w_0}, (A_1 \br{w_0})^\vee).
\end{equation}
Indeed, by Corollary \ref{C:p_u} we have that $s_{\zero_i^+}=p_1^*(\varpi_i)$ and $s_{\Id_i}=p_2^*(\varpi_i)$. So by \eqref{eq:twistcyc} \begin{align*}
s_{\zero_{i}^+}(U^-h_1, U^-h_2u, U)&=h_1^{\varpi_i},\\
s_{\Id_{i^*}}(U^-h_2u, U^- \br{w_0}, h_1^{-1} \br{w_0}^{-1} U) &= (\br{w_0}h_1^{-1} \br{w_0}^{-1})^{\varpi_{i^*}} = h_1^{\varpi_i}.
\end{align*}

Next let us assume that $\varphi=\bs{\mu}_l^2(s_f)$ for some mutable $f\in \ind(C^2 Q)$.
We need to show that $\bs{\mu}_l(\varphi)(A_1,A_2,A_3^\vee)=\varphi(A_2,A_3\br{w_0}, \br{w_0} A_1^\vee)$.
We argue by multidegrees. By Corollary \ref{C:mu_l}, the degree of $\varphi$ is $({\f_+^*,\e(f),\f_-^*})$.
Then according to \eqref{eq:irrG} and \eqref{eq:irrGdual},
the degree of $\varphi(A_2,A_3\br{w_0}, \br{w_0} A_1^\vee)$ is $(\triwtf)$, which is also the degree of $\bs{\mu}_l(\varphi)=s_f$.
So by Lemma \ref{L:c=1}, we must have that $\bs{\mu}_l(\varphi) = c\varphi l$ for some $c\in k$.
Again by the relation \eqref{eq:rel} we must have $c=1$.

\end{proof}

\begin{remark} \label{r:DT} The similar argument can show that $\pi \bs{\mu}_{\sqrt{l}}$ is equal to the pullback $\sqrt{l}^*$.
In view of Lemma \ref{L:mu_sqrtl_QG}, the automorphism $\sqrt{l}$ is the Donaldson-Thomas transformation of $\Conf_{2,1}$ in the sense of \cite{GSdt}.
\end{remark}

\subsection{An Algorithm} \label{ss:algorithm}
In this section, we present an algorithm to find all subrepresentations of $T_v$ defined in Section \ref{ss:iARt2}.
As said in the introduction, only few $T_v$'s for type $E_7$ and $E_8$ need this algorithm.
With a little more effort, one can show that Corollary \ref{C:mu_l} can be strengthened to
$\bs{\mu}_{l}(\Delta_Q^2)=\pi^2(\Delta_Q^2)$
for any orientation of type $D_n$ and $E_7,E_8$, and for some orientations of type $A_n$ and $E_6$.
For any particular case whether $\bs{\mu}_{l}(\Delta_Q^2)=\pi^2(\Delta_Q^2)$ holds can be checked by computer.

We first observe that by the description of $T_v$'s in Theorem \ref{T:Tv}, all subrepresentations of $T_{\zero_i^-}$ are known.
In particular, we have the dual $F$-polynomial of $T_{\zero_i^-}$.
Even better, $T_{\zero_i^-}$ can be mutated from the positive simple $(0,S_{\zero_i^-})$ of the (unfrozen) QP $\bs{\mu}_{i}\left(\Delta(C^2Q),W_Q^2\right)$ via $\bs{\mu}_{i}^{-1}$,
where $\bs{\mu}_i:=\bs{\mu}_{i,0}\mu_{i,0}$ and the bold $\bs{\mu}_{i,0}$ is defined in the beginning of Appendix \ref{ss:cyclic}.
Indeed, it can be easily checked by \eqref{eq:gdual_mu} that the $\g^\vee$-vector of such a mutated representation is exactly $\e_{\zero_{i}^-}-\sum_{i\to j} \e_{\zero_{j}^-}.$
Since $T_{\zero_i^-}$ is the cokernel of a general presentation of such a weight, our claim follows from Remark \ref{r:gcoh}.
The idea of the algorithm is that we can generate the dual $F$-polynomial of $T_{\Id_{i^*}}$ (resp. $T_{\zero_i^+}$) from that of $T_{\zero_i^-}$ through the sequence of mutations $\bs{\mu}_l$ (resp. $\bs{\mu}_r$) just defined.

\begin{proposition} \label{P:algorithm} If $\bs{\mu}_{l}(\Delta_Q^2)=\pi^2(\Delta_Q^2)$,
then $\bs{\mu}_l(T_{\zero_i^-})=T_{\Id_{i^*}} \text{ and } \bs{\mu}_r(T_{\zero_i^-})=T_{\zero_i^+}.$
Here, we view $T_{\Id_{i^*}}$ and $T_{\zero_i^+}$ as representations of the original QP via the automorphisms $\pi^{-2}$ and $\pi^{2}$.
\end{proposition}

\begin{proof} It is clear from \eqref{eq:gdual_mu} that the $\g^\vee$-vector $\e_{\zero_{i}^-}-\sum_{i\to j} \e_{\zero_{j}^-}$ is unchanged under $\bs{\mu}_{l}$.
If $\bs{\mu}_{l}(\Delta_Q^2)=\pi^2(\Delta_Q^2)$, by Corollary \ref{C:mu_l} the mutated $\g^\vee$-vector can be view as
$\e_{\Id_{i^*}}-\sum_{i\to j} \e_{\Id_{j^*}}$ for the original QP $(\Delta_Q^2,W_Q^2)$ via the automorphism $\pi^{-2}$.
Note that $i\to j$ if and only if $i^*\to j^*$.
Clearly $T_{\Id_{i^*}}$ is the cokernel of the general presentation of weight $\e_{\Id_{i^*}}-\sum_{i^*\to j^*} \e_{\Id_j^*}$.
Hence, $\bs{\mu}_l(T_{\zero_i^-})=T_{\Id_{i^*}}$.
The argument for $\bs{\mu}_r(T_{\zero_i^-})=T_{\zero_i^+}$ is similar.
\end{proof}

\noindent
By the positivity of the cluster variables \cite{GHKK,LS}, the coefficients of $F$-polynomials of these cluster variables are all positive.
So we can compute $F_{T_{\Id_{i^*}}}^\vee$ (resp. $F_{T_{\zero_i^+}}^\vee$) by applying $\bs{\mu}_l$ (resp. $\bs{\mu}_r$) to $F_{T_{\zero_i^-}}^\vee$ using the formula \eqref{eq:F-poly}.
In this way, we find all subrepresentations of $T_v$'s.


\end{appendices}

\section*{Acknowledgement}
The author would like to thank the NCTS (National Center for Theoretical Sciences) and Shanghai Jiao Tong University for the financial support.
He would also like to thank Professor Bernard Leclerc for some interesting discussion during the Taipei Conference in Representation Theory Jan 2016.
Finally he thanks the anonymous referee for the careful review and helpful comments.

\bibliographystyle{amsplain}

\end{document}